\documentclass{amsart}
\usepackage{amscd,amssymb,amsmath,amsfonts}
\renewcommand\lq{\leqslant}
\newcommand\gq{\geqslant}
\newcommand\E{\mathcal{E}}
\newcommand\diag{\operatorname{diag}}
\newcommand\ds{\displaystyle}
\renewcommand\S{\mathbb{S}}

\newcommand\C{\mathbb C}
\newcommand\K{\mathcal{K}}
\newcommand\F{\mathcal{F}}
\newcommand\JR{\mathcal{J}_{\Ga}^{red}}
\newcommand\JM{\mathcal{J}_{\Ga}^{max}}
\newcommand\JJ{\mathcal{J}}
\renewcommand\H{\mathcal{H}}

\newcommand\N{\mathbb N}
\newcommand\ZZ{\mathcal Z}
\newcommand\Id{\mathcal{I}d}
\newcommand\BB{\mathcal B}

\newcommand\R{\mathbb R}
\newcommand\st{\text{ such that }}
\newcommand\Z{\mathbb Z}
\newcommand\T{\mathcal{T}}
\newcommand\TT{\mathcal{T}}
\newcommand\MM{\mathcal{M}}
\renewcommand\L{\mathcal{L}}
\newcommand\D{\mathcal D}
\newcommand\DD{\mathcal D}
\newcommand\G{\mathcal{G}}
\newcommand\ts{{\otimes}}
\newcommand\rtr{{\rtimes_{red}}}
\newcommand\rtm{{\rtimes_{max}}}

\newcommand{\aeq}{\stackrel{(\lambda,h)}{\sim}}
\newcommand{\aeqp}{\stackrel{(\lambda',h')}{\sim}}

\newcommand\Ga{{\Gamma}}

\newcommand\ga{{\gamma}}
\newcommand\lto{{\longrightarrow}}
\newcommand\defi{{\stackrel{\text{def}}{=\!=}}}

\newcommand\AG{{A\rtimes_{red}\Ga}}
\newcommand\M{{M}}
\newcommand\U{\operatorname{U}}
\renewcommand\P{\operatorname{P}}
\newcommand\ka{\kappa}
\newcommand\ue{\operatorname{U}_n^{\varepsilon,r}}
\newcommand\pe{\operatorname{P}_n^{\varepsilon,r}}
\newcommand\eps{\varepsilon}
\newcommand\erp{$\eps$-$r$-projection }
\newcommand\eru{$\eps$-$r$-unitary }

\theoremstyle{plain}
\newtheorem{theorem}{Theorem}[section]
\newtheorem{proposition}[theorem]{Proposition}
\newtheorem{corollary}[theorem]{Corollary}

\newtheorem{lemma}[theorem]{Lemma}
\newtheorem{remark}[theorem]{Remark}
\newtheorem{definition}[theorem]{Definition}
\newtheorem{example}[theorem]{Example}
\newtheorem{notation}[theorem]{Notation}
\begin{document}

\title[On a  quantitative   operator $K$-theory]{On   quantitative operator $K$-theory}

 \author[H. Oyono-Oyono]{Herv\'e Oyono-Oyono}
 \address{Universit\'e de Lorraine, Metz , France}
 \email{herve.oyono@math.univ-bpclermont.fr}
\author[G. Yu]{Guoliang Yu }
 \address{Vanderbilt University, Nashville, USA}
 \email{guoliang.yu@vanderbilt.edu}
\thanks{Oyono-Oyono is partially supported by the ANR ``Kind'' and Yu is partially supported by  a grant from the
U.S. National Science Foundation.}
\begin{abstract}
In this paper, we develop a  quantitative $K$-theory for filtered $C^*$-algebras.
Particularly interesting examples of filtered $C^*$-algebras include group $C^*$-algebras, crossed product
$C^\ast$-algebras and Roe algebras. We prove a  quantitative version of the six term exact sequence
and a quantitative Bott periodicity.
We apply the quantitative $K$-theory to formulate  a quantitative version of the
Baum-Connes conjecture and prove that the  quantitative Baum-Connes conjecture holds for a large class of groups.
\end{abstract}
\maketitle

% \begin{flushleft}{\it Keywords: Baum-Connes Conjecture, Coarse Geometry, Group and  Crossed product $C^*$-algebras
%   Novikov Conjecture, Operator Algebra $K$-theory, Roe Algebras}
% 
% % \medskip
% % 
% % {\it 2000 Mathematics Subject Classification: 19K35,46L80,58J22}
%  \end{flushleft}

\tableofcontents

\section{Introduction}

  The receptacles of higher indices of elliptic differential operators are   $K$-theory of $C^*$-algebras which  encode   the (large  scale) geometry of the underlying spaces. The following examples are important for purpose of applications to geometry and topology.
\begin{itemize}
\item $K$-theory of group $C^*$-algebras is a receptacle for higher index theory of equivariant elliptic differential operators on covering space  \cite{A,BCH,CM,kas};
\item $K$-theory of crossed product $C^*$-algebras  and more generally groupoid $C^*$-algebras for foliations serve as receptacles for longitudinally elliptic operators  \cite{C, CS};
\item the higher indices of elliptic operators on noncompact Riemannian manifolds live in $K$-theory of Roe algebras \cite{r}.
\end{itemize}
The local nature of differential operators implies that these higher indices can be defined in term of idempotents and invertible elements with finite propagation. Using homotopy invariance of the $K$-theory for $C^*$-algebras,  these higher indices give rise to topological invariants.
\medskip

In the context of Roe algebras, a quantitative operator $K$-theory was introduced to compute the higher indices of elliptic operators for noncompact spaces with finite asymptotic dimension \cite{y2}.
The aim of this paper is to develop a quantitative  $K$-theory  for general $C^\ast$-algebras equipped with a filtration. The filtration  structure  allows us to define  the concept of  propagation. Examples of $C^\ast$-algebras with filtrations include group $C^\ast$-algebras, crossed product $C^\ast$-algebras and Roe algebras.  The quantitative $K$-theory for $C^\ast$-algebras with filtrations  is then defined in terms of homotopy of quasi-projections and quasi-unitaries with propagation and norm controls. We introduce controlled morphisms to study quantitative operator $K$-theory. In particular, we derive a quantitative  version  of the six term exact sequence. In the case of crossed product algebras, we also define a quantitative version
of the Kasparov transformation  compatible with Kasparov product.
We end this paper by using the quantitative $K$-theory to formulate  a quantitative version of the Baum-Connes conjecture and prove it  for a large class of groups.

This paper is organized as follows: In section $1$,
 we collect a few notations and definitions including the concept of filtered $C^*$-algebras. We use the concepts of 
almost unitary and almost projection to define a quantitative $K$-theory for
 filtered $C^*$-algebras and we study its elementary properties.
 In section $2$, we introduce the notion of controlled morphism in quantitative $K$-theory.
Section $3$ is devoted to extensions of filtered $C^*$-algebras and to a controlled exact sequence
for quantitative $K$-theory.
  In section $4$, we  prove  a controlled version of the Bott periodicity and as a consequence, we obtain a controlled version of the six-term exact sequence in $K$-theory.
 In section $5$, we apply  $KK$-theory to study the quantitative
$K$-theory of crossed product $C^\ast$-algebras and discuss its application to $K$-amenability. 
Finally in section 8, we formulate a quantitative Baum-Connes conjecture
 and prove  the quantitative Baum-Connes conjecture for a large class of groups.

\section{Quantitative $K$-theory}

In this section, we introduce a notion of quantitative $K$-theory for $C^*$-algebras with a filtration.
Let us fix first some notations about  $C^*$-algebras we shall use throughout  this paper.
\begin{itemize}
\item If $B$ is a $C^*$-algebra and if $b_1,\ldots,b_k$ are respectively
elements of \linebreak $\M_{n_1}(B),\ldots,\M_{n_k}(B)$, we denote by
$\diag(b_1,\ldots,b_k)$ the block diagonal
matrix $\begin{pmatrix}b_1&&\\&\ddots&\\&&&b_k\end{pmatrix}$ of
$\M_{n_1+\cdots+n_k}(B)$.
\item If $X$ is a locally compact space and $B$ is a $C^*$-algebra, we
  denote by $C_0(X,B)$  the $C^*$-algebra of $B$-valued continuous
  functions on $X$ vanishing at infinity. The special cases of
  $X=(0,1]$, $X=[0,1)$, $X=(0,1)$ and $X=[0,1]$,  will be respectively denoted   by $CB$, $B[0,1)$,
  $SB$ and $B[0,1]$.
\item For a separable Hilbert space $\H$, we denote by $\K(\H)$ the
$C^*$-algebra of compact operators on $\H$.
\item If $A$ and $B$ are $C^*$-algebras, we will denote by $A\otimes B$
their spatial tensor product.
\end{itemize}

\subsection{Filtered $C^*$-algebras}

\begin{definition}
A filtered $C^*$-algebra $A$ is a $C^*$-algebra equipped with a family
$(A_r)_{r>0}$ of  linear subspaces   indexed by positive numbers such that:
\begin{itemize}
\item $A_r\subset A_{r'}$ if $r\lq r'$;
\item $A_r$ is stable by involution;
\item $A_r\cdot A_{r'}\subset A_{r+r'}$;
\item the subalgebra $\ds\bigcup_{r>0}A_r$ is dense in $A$.
\end{itemize}
If $A$ is unital, we also require that the identity  $1$ is an element of $ A_r$ for every positive
number $r$. The elements of $A_r$ are said to have {\bf propagation $r$}.
\end{definition}
\begin{itemize}
\item Let $A$ and $A'$ be respectively  $C^*$-algebras filtered by
$(A_r)_{r>0}$ and  $(A'_r)_{r>0}$. A homomorphism of $C^*$
-algebras $\phi:A\lto A'$
is a filtered homomorphism (or a homomorphism of  filtered $C^*$-algebras) if  $\phi(A_r)\subset A'_{r}$ for any
positive number $r$.
\item If $A$ is a  filtered $C^*$-algebra and $X$ is a locally compact
space,
then $C_0(X,A)$ is a  $C^*$-algebra  filtered by $(C_0(X,A_r))_{r>0}$. In
particular
 the algebras $CA$, $A[0,1]$, $A[0,1)$  and   $SA$ are  filtered $C^*$-algebras.
\item
If $A$ is a non unital filtered $C^*$-algebra, then its  unitarization $\widetilde{A}$ is
filtered by $(A_r+\C)_{r>0}$. We define for $A$ non-unital the homomorphism
$$\rho_A:\widetilde{A}\to\C;\, a+z\mapsto z$$ for $a\in A$ and $z\in\C$.
\end{itemize}

Prominent examples of filtered $C^*$-algebra are provided by  Roe algebras associated to
proper metric spaces, i.e.  metric spaces such that closed balls of given radius are compact.   Recall that for such a metric
space $(X,d)$,
a $X$-module  is a Hilbert space $H_X$ together with a $*$-representation
$\rho_X$ of $C_0(X)$ in $H_X$ (we shall  write $f$ instead of
$\rho_X(f)$). If the representation is non-degenerate, the
$X$-module is said to be non-degenerate.
A $X$-module is called standard if no non-zero function of $C_0(X)$
acts as a compact operator on $H_X$.

The following concepts were introduced by Roe in his work on index theory of elliptic operators on noncompact spaces \cite{r}.

\begin{definition}
Let $H_X$ be a standard non-degenerate $X$-module and let $T$ be a
bounded operator on $H_X$.
\begin{enumerate}
\item The support of $T$ is the complement of the open subset of
  $X\times X$
\begin{equation*}\begin{split}\{(x,y)\in X\times X\text{ s.t.  there exist }& f\text{ and }
g\text{ in }C_0(X)\text{  satisfying }\\
&f(x)\neq 0,\,g(y)\neq 0\text{
  and }f\cdot T\cdot g=0\}.\end{split}\end{equation*}
\item The operator $T$  is said to
  have finite propagation (in this case    propagation less than $r$) if  there exists a real $r$ such that for any $x$ and $y$ in $X$
  with  $d(x,y)>r$, then $(x,y)$ is not in the support of $T$.
\item The operator $T$ is said to be locally compact if $f\cdot T$ and
  $T\cdot f$ are compact for any $f$ in
  $C_0(X)$. We then define   $C[X]$ as the set of locally compact and finite
  propagation bounded operators of $H_X$, and for every $r>0$, we define   $C[X]_r$ as  the set of element of $C[X]$  with
propagation less than $r$.
\end{enumerate}
\end{definition}
We clearly have  $C[X]_r\cdot C[X]_{r'}\subset  C[X]_{r+r'}$.
We can check that up to
(non-canonical) isomorphism, $C[X]$ does not depend on the choice of
$H_X$.

\begin{definition}
The  Roe algebra $C^*(X)$ is  the norm closure of $C[X]$ in
the algebra $\L(H_X)$ of bounded operators on $H_X$. The Roe algebra in then filtered by $(C[X]_r)_{r>0}$.
\end{definition}
Although $C^*(X)$
is not canonically defined, it was proved in \cite{hry} that  up to canonical
isomorphisms, its $K$-theory
does not depend on the choice  of a
non-degenerate standard $X$-module. Furthermore, $K_*(C^*(X))$ is the natural
receptacle for higher  indices of elliptic
operators with support on  $X$ \cite{r}.

\medskip
If $X$ has bounded geometry,
then the Roe algebra admits a maximal version  \cite{gwy} filtered by  $(C[X]_r)_{r>0}$.
Other   important examples are  reduced and maximal crossed product  of a $C^*$-algebra by an action of a discrete group by
automorphisms. These  examples  will be studied in detail in section \ref{section-cross-product}.

\subsection{Almost projections/unitaries}
Let $A$ be a unital filtered $C^*$-algebra. For any  positive
numbers $r$ and $\eps$, we call
\begin{itemize}
\item an element $u$ in $A$  a $\eps$-$r$-unitary if $u$
  belongs to $A_r$,  $\|u^*\cdot
  u-1\|<\eps$
and  $\|u\cdot u^*-1\|<\eps$. The set of $\eps$-$r$-unitaries on $A$ will be denoted by $\operatorname{U}^{\varepsilon,r}(A)$.
\item an element $p$ in $A$   a $\eps$-$r$-projection    if $p$
  belongs to $A_r$,
  $p=p^*$ and  $\|p^2-p\|<\eps$. The set of $\eps$-$r$-projections on $A$ will be denoted by $\operatorname{P}^{\varepsilon,r}(A)$.
\end{itemize} For $n$ integer, we set  $\ue(A)=\operatorname{U}^{\varepsilon,r}(M_n(A))$ and
$\pe(A)=\operatorname{P}^{\varepsilon,r}(M_n(A))$.

For any  unital filtered $C^*$-algebra $A$, any
 positive  numbers $\eps$ and $r$ and  any positive integer $n$, we consider inclusions
$$\P_n^{\eps,r}(A)\hookrightarrow \P_{n+1}^{\eps,r}(A);\,p\mapsto
\begin{pmatrix}p&0\\0&0\end{pmatrix}$$ and
$$\U_n^{\eps,r}(A)\hookrightarrow \U_{n+1}^{\eps,r}(A);\,u\mapsto
\begin{pmatrix}u&0\\0&1\end{pmatrix}.$$ This allows us to  define
 $$\U_{\infty}^{\eps,r}(A)=\bigcup_{n\in\N}\ue(A)$$ and
$$\P_{\infty}^{\eps,r}(A)=\bigcup_{n\in\N}\pe(A).$$

\begin{remark}\label{rem-prod} Let $r$ and $\eps$ be positive  numbers with
  $\eps<1/4$;
\begin{enumerate}
\item If $p$ is an \erp in $A$, then the spectrum of $p$ is included in \\ $\left(\frac{1-\sqrt{1+4\eps}}{2},\frac{1-\sqrt{1-4\eps}}{2}\right)\cup
\left(\frac{1+\sqrt{1-4\eps}}{2},\frac{1+\sqrt{1+4\eps}}{2}\right)$ and thus $\|p\|<1+\eps$.
\item If $u$ is  an \eru in $A$, then
  $$1-\eps<\|u\|<1+\eps/2,$$ $$1-\eps/2<\|u^{-1}\|<1+\eps,$$ $$\|u^*-u^{-1}\|<(1+\eps)\eps.$$
\item Let $\ka_{0,\eps}:\R\to\R$ be  a continuous function
  such that
 $\ka_{0,\eps}(t)=0$ if $t\lq \frac{1-\sqrt{1-4\eps}}{2}$ and
 $\ka_{0,\eps}(t)=1$ if $t\gq \frac{1+\sqrt{1-4\eps}}{2}$. If
 $p$ is an \erp in $A$, then $\ka_{0,\eps}(p)$ is a projection such that
 $\|p-\ka_{0,\eps}(p)\|< 2\eps$ which moreover does not depends on the choice
 of $\ka_{0,\eps}$. From now on, we shall denote this projection by $\kappa_0(p)$.
\item If $u$ is an \eru in $A$, set
   $\ka_1(u)=u(u^*u)^{-1/2}$. Then $\ka_1(u)$ is a unitary such that
$\|u-\ka_1(u)\|<\eps$.
\item If $p$ is an \erp in $A$  and $q$ is a projection
 in $A$ such that $\|p-q\|<1-2\eps$, then  $\ka_0(p)$ and $q$ are
 homotopic projections \cite[Chapter 5]{we}.
\item If $u$ and $v$ are $\eps$-$r$-unitaries in $A$, then $uv$ is an
  $\eps(2+\eps)$-$2r$-unitary  in $A$.
\end{enumerate}

\end{remark}
\begin{definition} Let $A$ be a  $C^*$-algebra filtered by $(A_r)_{r>0}$.
 \begin{itemize}
  \item Let $p_0$ and $p_1$ be $\eps$-$r$-projections. We say that $p_0$ and $p_1$ are homotopic
$\eps$-$r$-projections if there exists a $\eps$-$r$-projection $q$  in $A[0,1]$ such that $q(0)=p_0$ and $q(1)=p_1$.
In this case, $q$ is called a homotopy of $\eps$-$r$-projections  in $A$ and will be denoted by $(q_t)_{t\in[0,1]}$.
 \item If $A$ is unital, let $u_0$ and $u_1$ be $\eps$-$r$-unitaries. We say that $u_0$ and $u_1$ are homotopic
$\eps$-$r$-unitaries  if there exists  an $\eps$-$r$-unitary $v$ in $A[0,1]$ such that $v(0)=u_0$ and $v(1)=u_1$.
In this case, $v$ is  called a homotopy of $\eps$-$r$-unitaries  in $A$ and will be denoted by $(v_t)_{t\in[0,1]}$.
\end{itemize}
\end{definition}

\begin{example}\label{example-homotopy}
 Let $p$ be a $\eps$-projection in a filtered unital  $C^*$-algebra  $A$. Set $c_t=\cos \pi t/2$
and $s_t=\sin \pi t/2$  for $t\in[0,1]$  and let us
  considerer the homotopy of projections $
  (h_t)_{t\in[0,1]}$ with $h_t=\begin{pmatrix} c_t^2& c_ts_t\\
c_ts_t&s^2_t\end{pmatrix}$ in $\M_2(\C)$ between
  $\diag(1,0)$ and $\diag(0,1)$. Set  $(q_t)_{t\in[0,1]}=(\diag(p,0)+(1-p)\otimes h_t)_{t\in[0,1]}$. 
Since $q_t^2-q_t=s^2_t(p^2-p)\otimes I_2$,  we see that  $(q_t)_{t\in[0,1]}$ is a
   homotopy of $\eps$-$r$-projections
  between $\diag(1,0)$ and $\diag( p,1-p)$ in $M_2(A)$.
\end{example}

Next result will be frequently used  throughout  the paper and is quite easy to  prove.
\begin{lemma}\label{lemma-almost-closed}
 Let $A$ be a  $C^*$-algebra filtered by $(A_r)_{r>0}$.
\begin{enumerate}

\item  If $p$  is an \erp in $A$ and  $q$ is a self-adjoint element of $A_r$ such
  that
   $\|p-q\|<\frac{\eps-\|p^2-p\|}{4}$, then $q$ is \erp. In
   particular, if $p$ is an \erp in $A$ and if $q$ is a self-adjoint  element in
   $A_r$ such that $\|p-q\|< \eps$, then $q$ is a
   $5\eps$-$r$-projection in $A$ and $p$ and $q$ are connected by a
   homotopy of $5\eps$-$r$-projections.
\item  If $A$ is unital and if $u$ is an \eru and  $v$ is an element of $A_r$ such that
  $\|u-v\|<\frac{\eps-\|u^*u-1\|}{3}$, then $v$ is an \eru. In
  particular, if $u$ is an \eru and $v$ is an element of $A_r$ such that
  $\|u-v\|<\eps$, then $v$ is an $4\eps$-$r$-unitary  in $A$ and $u$
  and $v$ are connected by a homotopy of $4\eps$-$r$-unitaries.
\end{enumerate}
\end{lemma}

\begin{lemma}\label{lem-conjugate-proj}
There exists a real $\lambda>4$ such that for any positive number
$\eps$  with $\eps<1/{\lambda}$, any positive real $r$, any
$\eps$-$r$-projection $p$ and $\eps$-$r$-unitary $W$ in a filtered unital
$C^*$-algebra $A$,  the following
assertions  hold:
\begin{enumerate}
\item $WpW^*$ is a $\lambda\eps$-$3r$-projection of $A$;
\item $\diag(WpW^*,1)$ and   $\diag(p,1)$ are homotopic  $\lambda\eps$-$3r$-projections.
\end{enumerate}
\end{lemma}
\begin{proof}
The first point is straightforward to check from remark \ref{rem-prod}. For the second point,  with notations of
example \ref{example-homotopy},       use  the homotopy of
$\eps$-$r$-unitaries
$$\left(\begin{smallmatrix}W c_t^2 + s_t^2
    &(W-1) s_t c_t \\ (W-1) s_t
    c_t& W s^2_t+c_t^2 \end{smallmatrix}\right)_{t\in[0,1]}=\left(\left(\begin{smallmatrix} c_t
    &-s_t\\ s_t& c_t \end{smallmatrix}\right)\cdot \diag(W,1)  \cdot   \left(\begin{smallmatrix} c_t
    &s_t\\ -s_t& c_t \end{smallmatrix}\right) \right)_{t\in[0,1]}$$
    to connect by conjugation $\diag(WpW^*,1)$ to   $\diag(p,WW^*)$ and then connect to   $\diag(p,1)$
    by a ray.
\end{proof}

Recall that if two projections in a unital $C^*$-algebra are close enough in norm,
then there are conjugated by a canonical unitary. To state  a similar
result in term of $\eps$-$r$-projections and $\eps$-$r$-unitaries, we will need the definition of a control pair.
\begin{definition}\label{def-control-pair}
 A control pair is a pair $(\lambda,h)$, where
\begin{itemize}
 \item $\lambda >1$;
\item  $h:(0,\frac{1}{4\lambda})\to (0,+\infty);\, \eps\mapsto h_\eps$  is a map such that there exists a non-increasing map
$g:(0,\frac{1}{4\lambda})\to (0,+\infty)$, with $h\lq g$.
\end{itemize}
\end{definition}

\begin{lemma}\label{lem-conjugate}
There exists a control pair $(\lambda,h)$   such that  the following holds:

for every positive number $r$, any $\eps$ in $(0,\frac{1}{4\lambda})$  and any $\eps$-$r$-projections $p$ and
$q$ of a filtered unital $C^*$-algebra $A$ satisfying  $\|p-q\|<1/16$, there
exists  an $\lambda\eps$-$h_{\eps} r$-unitary $W$ in $A$  such that
$\|WpW^*-q\|\lq \lambda\eps$.
\end{lemma}
\begin{proof}
We follow the proof of \cite[Proposition 5.2.6]{we}. If we
set $$z=(2\kappa_0(p)-1)(2\kappa_0(q)-1)+1,$$ 
\begin{itemize}
\item then \begin{eqnarray*}
\|z-2\|&\lq &2\|\kappa_0(p)-\kappa_0(q)\| \\
&\lq & 8\eps+2\|p-q\|
\end{eqnarray*} and hence $z$ is invertible for $\eps<1/16$.
\item Moreover, if  we set $U=z|z^{-1}|$ and  since $z\kappa_0(q)=\kappa_0(p)z$, then we have
$\kappa_0(q)=U\kappa_0(p)U^*$.
\end{itemize}
Let us define $z'=(2p-1)(2q-1)+1$. Then we have $\|z-z'\|\lq 9\eps$
and $\|z'\|\lq 3$. If $\eps$ is small enough, then
$\|z'^*z'-4\|\lq 2$ and hence the spectrum of $z'^*z'$ is in
$[2,6]$. Let us consider the expansion in power serie $\sum_{k\in\N}a_k t^k$ of
$t\mapsto (1+t)^{-1/2}$ on $(0,1)$ and let $n_{\eps}$ be the smallest
integer such that $\sum_{ n_{\eps}\lq k} |a_k|/2^k\lq\eps$.
Let us set then $W=z'/2\sum_{k=0}^{n_\eps} a_k
(\frac{z'^*z'-4}{4})^k$. Then  for
a  suitable $\lambda$ (not depending on $A,\,p,\,q$ or $\eps$), we
get that $W$ is a $\lambda\eps$-$(4n_\eps+2)r$-unitary which satisfies the
required condition.
\end{proof}
\begin{remark}
 The order of $h$ when $\eps$ goes to zero in lemma \ref{lem-conjugate} is $C\eps^{-3/2}$ for some constant $C$.
\end{remark}

\subsection{Definition of quantitative $K$-theory}

For a unital filtered $C^*$-algebra $A$, we define the
following
equivalence relations on $\P_\infty^{\eps,r}(A)\times\N$ and on  $\U_\infty^{\eps,r}(A)$:
\begin{itemize}
\item if $p$ and $q$ are elements of $\P_\infty^{\eps,r}(A)$, $l$ and
  $l'$ are positive integers, $(p,l)\sim(q,l')$ if there exists a
  positive integer $k$ and an element $h$ of
  $\P_\infty^{\eps,r}(A[0,1])$ such that $h(0)=\diag(p,I_{k+l'})$
and $h(1)=\diag(q,I_{k+l})$.
\item if $u$ and $v$ are elements of $\U_\infty^{\eps,r}(A)$, $u\sim v$ if
  there exists an element $h$ of
  $\U_\infty^{\eps,r}(A[0,1])$ such that $h(0)=u$
and $h(1)=v$.
\end{itemize}
If $p$ is an  element of $\P_\infty^{\eps,r}(A)$ and  $l$ is an integer, we
denote by $[p,l]_{\eps,r}$ the equivalence class of $(p,l)$ modulo  $\sim$
and if $u$ is an element of $\U_\infty^{\eps,r}(A)$ we denote by
$[u]_{\eps,r}$ its  equivalence class  modulo  $\sim$.
\begin{definition} Let $r$ and $\eps$ be positive numbers with
  $\eps<1/4$.
We define:
\begin{enumerate}
\item $K_0^{\eps,r}(A)=\P_\infty^{\eps,r}(A)\times\N/\sim$ for $A$ unital and
$$K_0^{\eps,r}(A)=\{[p,l]_{\eps,r}\in \P^{\eps,r}(\tilde{A})\times\N/\sim \st
\dim \kappa_0(\rho_{A}(p))=l\}$$ for $A$ non unital.
\item $K_1^{\eps,r}(A)=\U_\infty^{\eps,r}(\tilde{A})/\sim$ (with
  $A=\tilde{A}$ if $A$ is already unital).
\end{enumerate}
\end{definition}

\begin{remark}
 We shall see in lemma \ref{lem-k+} that as it is the case   for $K$-theory,   $K_*^{\eps,r}(\bullet)$ can indeed be defined
 in a uniform way for unital and non-unital filtered $C^*$-algebras.
\end{remark}

It is straightforward to check that for any unital filtered $C^*$-algebra $A$,  if $p$ is an $\eps$-$r$-projection in $A$
and  $u$ is an $\eps$-$r$-unitary  in $A$,  then
$\diag(p,0)$ and
$\diag(0,p)$ are homotopic $\eps$-$r$-projections in
$M_2(A)$ and   $\diag(u,1)$ and
  $\diag(1,u)$ are homotopic $\eps$-$r$-unitaries in
$M_2(A)$. Thus we obtain the following:
\begin{lemma}
 Let $A$ be a filtered $C^*$-algebra. Then $K_0^{\eps,r}(A)$ and $K_1^{\eps,r}(A)$
are equipped with a structure of abelian semi-group such that
$$[p,l]_{\eps,r}+[p',l']_{\eps,r}=[\diag(p,p'),l+l']_{\eps,r}$$ and
$$[u]_{\eps,r}+[u']_{\eps,r}=[\diag(u,v)]_{\eps,r},$$ for any  $[p,l]_{\eps,r}$ and $[p',l']_{\eps,r}$ in $K_0^{\eps,r}(A)$ and
any  $[u]_{\eps,r}$ and $[u']_{\eps,r}$ in $K_1^{\eps,r}(A)$.
\end{lemma}

According to example \ref{example-homotopy}, for every unital filtered $C^*$-algebra $A$, any
 $\eps$-$r$-projection $p$ in $M_n(A)$ and   any integer $l$ with $n\gq l$, we see that  $[I_n-p,n-l]_{\eps,r}$
 is an inverse for $[p,l]_{\eps,r}$. Hence we obtain:
\begin{lemma} If $A$ is a  filtered $C^*$-algebra, then $K_0^{\eps,r}(A)$ is an abelian group.
 \end{lemma}

Although $K_1^{\eps,r}(A)$ is not a group, it is very close to be one. 
\begin{lemma}\label{lem-k1-almost-group}
 Let  $A$ be  a filtered $C^*$-algebra. Then for any $\eps$-$r$-unitary $u$ in $M_n(\tilde{A})$
 (with $A=\tilde{A}$ if $A$ is already unital), we have
  $[u]_{3\eps,2r}+[u^*]_{3\eps,2r}=0$ in $K_1^{3\eps,2r}(A)$.
\end{lemma}
\begin{proof}
 If $u$ is an $\eps$-$r$-unitary in a
unital filtered $C^*$-algebra $A$,
then according to point (vi) of remark \ref{rem-prod}, we see that $\left(\diag(1,u)\left(\begin{smallmatrix} c_t
    &-s_t\\ s_t& c_t \end{smallmatrix}\right)\cdot \diag(1,u^*)  \cdot   \left(\begin{smallmatrix} c_t
    &s_t\\ -s_t& c_t \end{smallmatrix}\right) \right)_{t\in[0,1]}$ is a homotopy of $3\eps$-$2r$-unitaries between 
$\diag(u,u^*)$ and $\diag(uu^*,1)$. Since $\| uu^*-1\|<\eps$,   we deduce from   lemma
\ref{lemma-almost-closed} that $uu^*$ and $1$ are 
       homotopic   $3\eps$-$2r$-unitaries and hence we get the lemma.
\end{proof}

\begin{remark}
 According to lemma \ref{lem-k1-almost-group}, if we define the equivalence relation on $\U_\infty^{\eps,r}(A)$ to be homotopy within
$\U_\infty^{3\eps,2r}(A)$, then $K_1^{\eps,r}(A)$ can be endowed with an abelian group structure.
\end{remark}

We have for any filtered $C^*$-algebra $A$ and any  positive numbers
$r$, $r'$, $\eps$ and $\eps'$  with
  $\eps\lq\eps'<1/4$ and $r\lq r'$  natural semi-group homomorphisms
\begin{itemize}
\item $\iota_0^{\eps,r}:K_0^{\eps,r}(A)\lto K_0(A);\,
[p,l]_{\eps,r}\mapsto [\kappa_0(p)]-[I_l]$;
\item $\iota_1^{\eps,r}:K_1^{\eps,r}(A)\lto K_1(A);\,
  [u]_{\eps,r}\mapsto [u]$;
\item $\iota_*^{\eps,r}=\iota_0^{\eps,r}\oplus \iota_1^{\eps,r}$;
\item $\iota_0^{\eps,\eps',r,r'}:K_0^{\eps,r}(A)\lto K_0^{\eps',r'}(A);\,
[p,l]_{\eps,r}\mapsto [p,l]_{\eps',r'};$
\item $\iota_1^{\eps,\eps',r,r'}:K_1^{\eps,r}(A)\lto K_1^{\eps',r'}(A);\,
  [u]_{\eps,r}\mapsto [u]_{\eps',r'}$.
\item $\iota_*^{\eps,\eps',r,r'}=\iota_1^{\eps,\eps',r,r'}\oplus\iota_1^{\eps,\eps',r,r'}$
\end{itemize}
If some of the indices $r,r'$ or $\eps,\eps'$ are equal, we shall not
repeat it in $\iota_*^{\eps,\eps',r,r'}$.
\begin{remark}\label{rem-hom} Let $p_0$ and $p_1$ be two $\eps$-$r$-projections in a filtered $C^*$-algebra such that
$\ka_0(p_0)$ and $\ka_0(p_1)$ are homotopic projections. Then for any $\eps$ in $(0,1/4)$, this homotopy 
can be approximated  for some $r'$
by a $\eps$-$r'$-projection. Hence, using   point (iii)
   of remark \ref{rem-prod}, there exists a homotopy $(q_t)_{t\in[0,1]}$ of $\eps$-$r'$ projections in $A$ such that
$\|p_0-q_0\|<3\eps$ and $\|p_1-q_1\|<3\eps$. We can indeed assume that $r'\gq r$ and thus by 
   lemma \ref{lemma-almost-closed}, we get that $p_0$ and $p_1$ are homotopic as $15\eps$-$r'$-projections. Proceeding in
the same way for the odd case we eventually obtain:

 there exists $\lambda>1$ such that for any    filtered $C^*$-algebra $A$, any $\eps\in(0,\frac{1}{4\lambda})$ and any positive number $r$,
the following holds:

 \medskip

 Let $x$ and $x'$ be elements in $K_*^{\eps,r}(A)$ such that
$\iota_*^{\eps,r}(x)=\iota_*^{\eps,r}(x')$ in $K_*(A)$, then there
exists a positive number  $r'$ with $r'>r$ such that
$\iota_*^{\eps,\lambda\eps,r,r'}(x)=\iota_*^{\eps,\lambda\eps,r,r'}(x')$ in
$K_*^{\lambda\eps,r'}(A)$.
\end{remark}

\begin{lemma}\label{lem-hom-proj}
Let $p$ be a matrix in $\M_n(\C)$ such that $p=p^*$ and
$\|p^2-p\|<\eps$ for some  $\eps$ in $(0,1/4)$. Then there is a continuous path
$(p_t)_{t\in[0,1]}$ in $\M_n(\C)$ such that
\begin{itemize}
\item $p_0=p$;
\item $p_1=I_k$ with $k=\dim\kappa_0(p)$;
\item $p_t^*=p_t$ and $\|p_t^2-p_t\|<\eps$ for every $t$ in $[0,1]$.
\end{itemize}\end{lemma}
\begin{proof}
The selfadjoint matrix $p$ satisfies
$\|p^2-p\|<\eps$ if and only if the eigenvalues of $p$ satisfy the
inequality $$-\eps<\lambda^2-\lambda<\eps,$$ i.e.
$$\lambda\in
\left(\frac{1-\sqrt{1+4\eps}}{2},\frac{1-\sqrt{1-4\eps}}{2}\right)\bigcup
\left(\frac{\sqrt{1-4\eps}+1}{2},\frac{\sqrt{1+4\eps}+1}{2}\right).$$
Let $\lambda_1,\ldots,\lambda_k$ be the eigenvalues of $p$ lying in
$\left(\frac{1-\sqrt{1+4\eps}}{2},\frac{1-\sqrt{1-4\eps}}{2}\right)$
and let \linebreak $\lambda_{k+1},\ldots,\lambda_n$ be the eigenvalues of $p$
lying in
$\left(\frac{\sqrt{1-4\eps}+1}{2},\frac{\sqrt{1+4\eps}+1}{2}\right)$.
We set for $t\in[0,1]$
\begin{itemize}
\item $\lambda_{i,t}=t\lambda_i$ for  $i=1,\ldots,k$;
\item $\lambda_{i,t}=t\lambda_i+1-t$ for  $i=k+1,\ldots,n$.
\end{itemize}
Since $\lambda\mapsto\lambda^2-\lambda$ is decreasing  on
$\left(\frac{1-\sqrt{1+4\eps}}{2},\frac{1-\sqrt{1-4\eps}}{2}\right)$
and increasing on \linebreak
$\left(\frac{\sqrt{1-4\eps}+1}{2},\frac{\sqrt{1+4\eps}+1}{2}\right)$
then,  $$-\eps<\lambda_{i,t}^2-\lambda_{i,t}<\eps$$ for all $t$ in
$[0,1]$ and $i=1,\ldots,n$.
If we set
$p_t=u\cdot\diag(\lambda_{1,t},\ldots,\lambda_{n,t})\cdot u^*$ where $u$ is a
unitary matrix of $\M_n(\C) $ such that
$p=u\cdot\diag(\lambda_{1},\ldots,\lambda_{n})\cdot u^*$,
then
\begin{itemize}
\item $p_0=p$;
\item $p_1=\kappa_0(p)$;
\item $p_t^*=p_t$ and $\|p_t^2-p_t\|<\eps$ for every $t$ in $[0,1]$.
\end{itemize}
Since there is a homotopy of projections in  $\M_n(\C)$ between
$\kappa_0(p)$ and $I_k$ with $k=\dim\kappa_0(p)$, we get the result.
\end{proof}

As a consequence we obtain:
\begin{corollary}For any positive numbers  $r$ and $\eps$ with $\eps<1/4$, then
$$K_0^{\eps,r}(\C)\to\Z;\,[p,l]_{\eps,r}\mapsto \dim\kappa_0(p)-l$$
is an isomorphism.
\end{corollary}

\begin{lemma}\label{lem-hom-uni}
Let $u$ be a matrix in $\M_n(\C)$ such that $\|u^*u-I_n\|<\eps$ and
$\|uu^*-I_n\|<\eps$ for $\eps$ in $(0,1/4)$. Then there is a continuous path
$(u_t)_{t\in[0,1]}$ in $\M_n(\C)$ such that
\begin{itemize}
\item $u_0=u$;
\item $u_1=I_n$;
\item $\|u_t^*u_t-I_n\|<\eps$ and
$\|u_tu_t^*-I_n\|<\eps$  for every $t$ in $[0,1]$.
\end{itemize}
\end{lemma}
\begin{proof}
Since $u$ is invertible, $u^*u$ and $uu^*$ have the same eigenvalues $\lambda_1,\ldots,\lambda_n$,
and thus $\|u_t^*u_t-I_n\|<\eps$ and
$\|u_tu_t^*-I_n\|<\eps$ if and only if
$\lambda_i\in(1-\eps,1+\eps)$ for $i=1,\ldots,n$. Let us set
\begin{itemize}
\item $h_t=w\cdot\diag( \lambda_1^{-t/2},\ldots, \lambda_n^{-t/2})\cdot w^*$ where $w$
 is a
unitary matrix of $\M_n(\C) $ such that
$u^*u=w\cdot\diag(\lambda_1,\ldots,\lambda_n)\cdot w^*$;
\item
$v_t=u\cdot h_t$ for all $t\in[0,1]$.
Then $v_t^*v_t=w\cdot\diag( \lambda_1^{1-t},\ldots,
\lambda_n^{1-t})\cdot w^*$.
\end{itemize} Since
$|\lambda_i^{1-t}-1|<\eps$ for all all $t\in[0,1]$, we get that
$\|v_t^*v_t-I_n\|<\eps$ and
$\|v_tv_t^*-I_n\|<\eps$  for every $t$ in $[0,1]$. The matrix $v_1$ is
unitary and the result then follows from path-connectness of $\U_n(\C)$.
\end{proof}

As a consequence we obtain:
\begin{corollary}
For any positive numbers $r$ and $\eps$  with $\eps<1/4$, then we have
$K_1^{\eps,r}(\C)=\{0\}$.
\end{corollary}
\subsection{Elementary properties of quantitative $K$-theory}
Let $A_1$ and $A_2$ be two unital $C^*$-algebras respectively filtered by $(A_{1,r})_{r>0}$ and $(A_{2,r})_{r>0}$ and consider
$A_1\oplus A_2$ filtered by $(A_{1,r}\oplus A_{2,r})_{r>0}$. Since we have identifications
 $\P_\infty^{\eps,r}(A_1\oplus A_2)\cong \P_\infty^{\eps,r}(A_1)\times \P_\infty^{\eps,r}( A_2)$  and
 $\U_\infty^{\eps,r}(A_1\oplus A_2)\cong \U_\infty^{\eps,r}(A_1)\times \U_\infty^{\eps,r}( A_2)$ induced by the inclusions
$A_1\hookrightarrow A_1\oplus A_2$ and $A_2\hookrightarrow A_1\oplus A_2$, we see that we have isomorphisms
$K_0^{\eps,r}(A_1)\oplus K_0^{\eps,r}(A_2)\stackrel{\sim}{\longrightarrow}K_0^{\eps,r}(A_1\oplus A_2)$ and
$K_1^{\eps,r}(A_1)\oplus K_1^{\eps,r}(A_2)\stackrel{\sim}{\longrightarrow}K_1^{\eps,r}(A_1\oplus A_2)$.

\begin{lemma}\label{lem-split} Let $A$ be a filtered non unital
  $C^*$-algebra and let $\eps$ and $r$ be positive numbers with $\eps<1/4$.
We have a natural splitting
$$K_0^{\eps,r}(\tilde{A})\stackrel{\cong}{\longrightarrow}K_0^{\eps,r}(A)\oplus\Z.$$
\end{lemma}
\begin{proof}
Viewing $A$ as a subalgebra of $\tilde{A}$, the group  homomorphisms
\begin{eqnarray*}
K_0^{\eps,r}(\tilde{A})&{\longrightarrow}&K_0^{\eps,r}(A)\oplus\Z\\
\lbrack p,l\rbrack_{\eps,r}&\mapsto&\left(\lbrack p,\dim
\kappa_0(\rho_A(p))\rbrack_{\eps,r},\dim \kappa_0(\rho_A(p))-l\right)
\end{eqnarray*}

and
\begin{eqnarray*}
K_0^{\eps,r}(A)\oplus\Z&{\longrightarrow}&K_0^{\eps,r}(\tilde{A})\\
\left(\lbrack p,l\rbrack_{\eps,r},k-k'\right)&\mapsto&\left\lbrack\begin{pmatrix}p&0\\0&I_k\end{pmatrix},l+k'\right\rbrack_{\eps,r}
\end{eqnarray*}
are inverse one of the other.
\end{proof}
Let us set $A^+=A\oplus\C$ equipped with  the multiplication
$$(a,x)\cdot(b,y)=(ab+xb+ya,xy)$$  for $a$ and $b$ in $A$ and $x$ and
$y$ in $\C$. Notice that
\begin{itemize}
\item $A^+$ is isomorphic to $A\oplus\C$
with the algebra structure provided by the direct sum if $A$ is unital;
\item $A^+  =\tilde{A}$ if $A$ is not unital.
\end{itemize}Let us define also $\rho_A$  in the unital
case by $\rho_A:A^+\to\C;\, (a,x)\mapsto x$.
We know that in usual $K$-theory, we can equivalently define for $A$ unital
the $\Z_2$-graded group $K_*(A)$ as
$A^+$ by $$K_0(A)=\ker \rho_{A,*}:K_0(A^+)\to K_0(\C)\cong \Z$$ and
$$K_1(A)=K_1(A^+).$$
Let us check that this is also the case for our  $\Z_2$-graded semi-groups ${K_*}^{\eps,r}(A)$.
If the $C^*$-algebra $A$ is filtered by $(A_r)_{r>0}$, then $A^+$ is
filtered by $(A_r+\C)_{r>0}$. Let us define for a unital filtered
algebra $A$
$${K_0'}^{\eps,r}(A)=\{[p,l]_{\eps,r}\in \P^{\eps,r}({A}^+)\times \N/\sim \st
\dim \kappa_0(\rho_{A}(p))=l\}$$ and
 $${K_1'}^{\eps,r}(A)=\U^{\eps,r}({A}^+)/\sim.$$
Proceeding  as we did in  the proof of lemma \ref{lem-split}, we obtain
a natural splitting $$K_0^{\eps,r}({A}^+)\stackrel{\cong}{\longrightarrow}{K_0'}^{\eps,r}(A)\oplus\Z.$$
But then, using the identification $A^+\cong A\oplus \C$ and in view of lemmas \ref{lem-hom-proj} and \ref{lem-hom-uni}, we get

\begin{lemma}\label{lem-k+}
The $\Z_2$-graded semi-groups
$K_*^{\eps,r}(A)$ and ${K_*'}^{\eps,r}(A)$  are naturally isomorphic.
\end{lemma}

This allows us to state functoriallity properties for quantitative $K$-theory.
 If  $\phi:A\to B$ is a  homomorphism of unital  filtered $C^*$-algebras, then since $\phi$ preserve $\eps$-$r$-projections
 and $\eps$-$r$-unitaries, it obviously induces  for any positive number $r$ and any
$\eps\in(0,1/4)$ a
semi-group homomorphism $$\phi_*^{\eps,r}:K_*^{\eps,r}(A)\longrightarrow
K_*^{\eps,r}(B).$$  In  the non unital case, we can extend any  homomorphism  $\phi:A\to B$ to a
homomorphism  $\phi^+:A^+\to B^+$ of unital  filtered $C^*$-algebras and then we use  lemmas \ref{lem-split}
and \ref{lem-k+} to define $\phi_*^{\eps,r}:K_*^{\eps,r}(A)\longrightarrow
K_*^{\eps,r}(B).$ Hence, for any positive number $r$ and any
$\eps\in(0,1/4)$, we get that $K_0^{\eps,r}(\bullet)$ (resp. $K_1^{\eps,r}(\bullet)$) is a covariant additive functor from the category
of filtered $C^*$-algebras (together with filtered homomorphism) to the category of abelian groups (resp. semi-groups).
\begin{definition}\
\begin{enumerate}
\item Let $A$ and $B$ be filtered $C^*$-algebras. Then two homomorphisms
  of filtered $C^*$-algebras
  $\psi_0:A\to B$ and $\psi_1:A\to B$ are homotopic
if there exists a path of
homomorphisms of filtered $C^*$-algebras $\psi_t:A\to B$ for $0\lq t \lq 1$ between $\psi_0$ and $\psi_1$  and such that
$t\mapsto \psi_t$ is continuous for the pointwise norm
  convergence.
\item
A filtered $C^*$-algebra $A$ is said to be contractible if the identity map and
the zero map of $A$ are homotopic.
\end{enumerate}\end{definition}

\begin{example}
If $A$ is a filtered $C^*$-algebra $A$, then the cone of $A$
$$CA=\{f\in C([0,1],A)\text{ such that } f(0)=0\}$$ is a contractible
filtered $C^*$-algebra.
\end{example}

We have then the following obvious result:
\begin{lemma}\label{lem-contractible}
If $\phi:A\to B$ and $\phi':A\to B$ are two homotopic homomorphisms of
filtered $C^*$-algebras, then $\phi_*^{\eps,r}={\phi'}_*^{\eps,r}$  for every positive numbers $\eps$ and $r$ with
$\eps<1/4$. In particular,
if $A$ is a contractible filtered $C^*$-algebra, then
$K_*^{\eps,r}({A})=0$ for every positive numbers $\eps$ and $r$ with
$\eps<1/4$.
\end{lemma}

Let $A$ be a $C^*$-algebra filtered by $(A_r)_{r>0}$ and let $(B_k)_{k\in\N}$
be an increasing sequence of $C^*$-subalgebras of $A$ such that
$\ds\bigcup_{k\in\N}B_k$ is dense in $A$. Assume that $\bigcup_{r>0}
B_k\cap A_r$ is dense in $B_k$  for every
integer $k$.  Then for every integer $k$,
the $C^*$-algebra $B_k$ is filtered by $(B_k\cap A_r)_{r>0}$. If $A$
is unital, then $B_k$ is unital for some $k$, and thus we will assume
without loss of generality that $B_k$ is unital for every integer $k$.
\begin{proposition}\label{prop-lim}
Let $A$ be a  unital  $C^*$-algebra filtered  by $(A_r)_{r>0}$ and let $(B_k)_{k\in\N}$
be an increasing sequence of $C^*$-subalgebra of $A$ such that
\begin{itemize}
\item  $\ds\bigcup_{r>0}
(B_k\cap A_r)$ is dense in $B_k$  for every
integer $k$,
\item  $\ds \bigcup_{k\in\N}
(B_k\cap A_r)$ is dense in $A_r$  for every positive number $r$.
\end{itemize}
 Then  the $\Z_2$-graded semi-groups
$K_*^{\eps,r}({A})$ and
 $\ds\lim_k
K_*^{\eps,r}({B_k})$ are isomorphic.
\end{proposition}
\begin{proof} In particular, we see that $\ds\bigcup_{k\in\N}B_k$ is dense in $A$.
Let us  denote by $$\Upsilon_{*,\eps,r}:\lim_k K_*^{\eps,r}({B_k})\to
K_*^{\eps,r}({A})$$ the homomorphism of semi-group induced by the family of inclusions
$B_k\hookrightarrow A$ where $k$ runs through integers.
We give the proof in the even case, the odd case being analogous.
Let $p$ be an element of $\pe(A)$ and let $\delta=\|p^2-p\|>0$ and
choose $\alpha<\frac{\eps-\delta}{12}$. Since
$\ds{\bigcup_{k\in\N}(B_k\cap A_r)}$ is dense in $A_r$, there is an integer $k$ and a
selfadjoint element $q$ of $\M_n(B_k\cap A_r)$ such that $\|p-q\|<\alpha$. According to lemma \ref{lem-hom-proj}, $q$ is a
$\eps$-$r$ projection. Let   $q'$ be  another selfadjoint element of $\M_n(B_k\cap A_r)$
  such that $\|p-q'\|<\alpha$. Then $\|q-q'\|<2\alpha$ and if we set
  $q_t=(1-t)q+tq'$ for $t\in[0,1]$, then
\begin{eqnarray*}
\|q_t^2-q_t\|&\lq&\|q_t^2-q_tq\|+\|q_tq-q^2\|+\|q^2-q\|+\|q-q_t\|\\
&\lq&\|q_t-q\|(\|q_t\|+\|q\|+1)+4\alpha+\delta\\
&\lq&12\alpha+\delta\\
&<&\eps,
\end{eqnarray*} and thus $q$ and $q'$ are homotopic in
$\P_n^{\eps,r}(B_k)$.
Therefore, for $p\in\pe(A)$ and $q$ in some $\M_n(B_k\cap A_r)$
satisfying $\|q-p\|< \frac{\|p^2-p\|}{12}$, we define  $\Upsilon'_{0,\eps,r}([p,l]_{\eps,r})$
to be the image  of $[q,l]_{\eps,r}$ in $\ds\lim_k
K_*^{\eps,r}({B_k})$. Then $\Upsilon'_{0,\eps,r}$ is a group homomorphism  and is  an inverse for $\Upsilon_{0,\eps,r}$. We proceed similarly in the odd case.
\end{proof}
\subsection{Morita equivalence}

For any  unital  filtered algebra  $A$,  we get an identification between
$\P^{\eps,r}_n(\M_k(A))$ and $\P^{\eps,r}_{nk}(A)$ and therefore  between
$\P_\infty^{\eps,r}(\M_k(A))$ and $\P_\infty^{\eps,r}(A)$. This
identification gives rise to a  natural  group isomorphism between $K_0^{\eps,r}(A)$ and
$K_0^{\eps,r}(\M_k(A))$, and  this isomorphism is
induced by the inclusion of $C^*$-algebras
$$\iota_A:A\hookrightarrow\M_k(A);\,a\mapsto\diag(a,0).$$ Namely, if we
set $e_{1,1}=\diag(1,0,\ldots,0)\in\M_k(\C)$, definition of the functoriality yields
$${\iota^{\eps,r}_{A,*}}[p,l]_{\eps,r}=[p\otimes e_{1,1}+I_l\otimes
(I_k-e_{1,1}),l]_{\eps,r}\in  K_0^{\eps,r}(\M_k(A))$$ for any $p$ in $\pe(A)$ and any
integer $l$ with $l\lq n$. We can verify that
$$(\iota^{\eps,r}_{A,*})^{-1}[q,l]_{\eps,r}=[q,kl]_{\eps,r}$$ for any $q$ in $\pe(\M_k(A))$ and any
integer $l$ with $l\lq n$, where on the right hand side of the equality, the matrix
$q$ of $\M_n(\M_k(A))$ is viewed as a matrix of $\M_{nk}(A)$.

In a similar way, we  obtain in the odd case an identification between
$\U_\infty^{\eps,r}(\M_k(A))$ and $\U_\infty^{\eps,r}(A)$ providing a natural semi-group isomorphism between $K_1^{\eps,r}(A)$ and
$K_1^{\eps,r}(\M_k(A))$. This isomorphism is also induced by the  inclusion
$\iota_A$ and we have $$\iota_{A,*}[x]_{\eps,r}=[x\otimes
e_{1,1}+I_n\otimes (I_k-e_{1,1})]_{\eps,r}\in
K_1^{\eps,r}(\M_k(A))$$ for any $x$ in $\ue(A)$.

Let us deal now with the non-unital case. For usual  $K$-theory, Morita equivalence for non-unital $C^*$-algebra 
 can be deduced from the unital case  by using
the six-term exact sequence associated to the split extension $0\to A\to \tilde{A}\to \C\to0$. But for  quantitative $K$-theory this
splitting only gives rise (in term of section \ref{subsection-controlled-morphism})      
 to a controlled isomorphism (see corollary \ref{cor-split-ext}).
In order to really have a genuine isomorphism, we have to go through the  tedious following computation. If $B$ is a
non-unital $C^*$-algebra, let us   identify $\M_k(\tilde{B})$ with
$\M_k(B)\oplus \M_k(\C)$ equipped with the product $$(b,\lambda)\cdot
(b',\lambda')=(bb'+\lambda b'+b\lambda',\lambda\lambda')$$ for $b$ and
$b'$ in $\M_k({B})$ and $\lambda$ and $\lambda'$ in $\M_k(\C)$.
Under this identification, if $A$ is not unital, let us check that the
semi-group homomorphism  $$\Phi_1:K_1^{\eps,r}(\tilde{A})\to
K_1^{\eps,r}(\widetilde{\M_k(A)});\, [(x,\lambda)]_{\eps,r}\mapsto
[(x\otimes e_{1,1},\lambda]_{\eps,r}$$
induced by the inclusion $\iota_A$ is  invertible
with  inverse  given by the composition
$$\Psi_1:K_1^{\eps,r}(\widetilde{\M_k(A)})\to
K_1^{\eps,r}({\M_k(\tilde{A}}))\stackrel{\cong}{\to}K_1^{\eps,r}(\tilde{A}),$$
  where the first homomorphism of the composition is induced by the inclusion
  $$\widetilde{\M_k(A)}\to \M_k(\tilde{A});\, (a,z)\mapsto (a,z
  I_k).$$
Let $(x,\lambda)$ be an element of $\ue(\tilde{A})$, with
$x\in\M_n(A)$ and $\lambda\in\M_n(\C)$. Then
$$\Psi_1\circ\Phi_1[(x,\lambda)]_{\eps,r}=[(x\otimes e_{1,1},\lambda\otimes
I_k)]_{\eps,r},$$ where we use the identification
$M_{nk}(\C)\cong M_{n}(\C)\otimes M_{k}(\C)$ to see $x\otimes e_{1,1}$ and $\lambda\otimes
I_k$   respectively as matrices in $\M_{nk}(A)$ and
$\M_{nk}(\C)$. According to lemma \ref{lem-hom-uni},  as a $\eps$-$r$-unitary of $M_n(\C)$,
$\lambda$  is homotopic to $I_n$. Hence
$$[(x\otimes e_{1,1},\lambda\otimes I_{k})]_{\eps,r}=[(x\otimes e_{1,1},\lambda\otimes e_{1,1}+
I_n\ts I_{k-1})] $$ and from this we get that $\Psi_1\circ\Phi_1$ is induced in $K$-theory by the
inclusion map $\tilde{A} \hookrightarrow
\M_k(\tilde{A});\,a\mapsto\diag(a,0) $ which is  the identity homomorphism (according to the unital
  case).
%Let $(\lambda_t)_{t\in[0,1]}$ be a homotopy in $\U_n^{\eps,r}(\C)$
%between $\lambda$ and $I_n$ (see lemme \ref{lem-hom-uni}), then $$(x\otimes e_{1,1},\lambda\otimes
%e_{1,1}+\lambda_t\otimes (I_k- e_{1,1}))_{t\in[0,1]}$$ is a homotopy in
%$\U_{nk}^{\eps,r}(\tilde{A})$ between $$(x\otimes e_{1,1},\lambda\otimes
%I_k)$$ and  $$(x\otimes e_{1,1},\lambda\otimes
%e_{1,1})+(0,I_n\otimes (I_k- e_{1,1}))$$ and thus $$[(x\otimes e_{1,1},\lambda\otimes
%I_k)]_{\eps,r}=[(x,\lambda]_{\eps,r}\in
%K_1^{\eps,r}(\tilde{A}).$$
Conversely, let $(y,\lambda)$ be an element
in $\ue(\widetilde{\M_k(A)})$ with $$y\in M_n(\M_k(A))\cong M_n(A)\otimes \M_k(\C)$$
and $\lambda\in\M_n(\C)$. Then
$$\Phi_1\circ\Psi_1[(y,\lambda)]_{\eps,r}=[(y\otimes
e_{1,1},\lambda\otimes I_k)]_{\eps,r},$$ where
\begin{itemize}
\item  $y\otimes
e_{1,1}$ belongs to $\M_n(\M_k(A))\otimes \M_k(\C)\cong \M_n(A)\otimes \M_k(\C)
\otimes \M_k(\C)$ (the first two  factors provide  the copy of
$\M_n(\M_k(A))$ where $y$ lies  in and $e_{1,1}$ lies in  the last
factor).
\item $\lambda\otimes I_k$ belongs to the algebra $M_n(\M_k(\C))\cong M_n(\C)\otimes
  \M_k(\C)$ that multiplies $\M_n(A)\otimes \M_k(\C)
\otimes \M_k(\C)$ on the first two  factors.
\end{itemize}
Let $$\sigma:\M_n(A)\otimes \M_k(\C)
\otimes \M_k(\C)\to \M_n(A)\otimes \M_k(\C)
\otimes \M_k(\C)$$ be the $C^*$-algebra  homomorphism induced by the flip of
$\M_k(\C)
\otimes \M_k(\C)$. This flip can be realized by conjugation of
a unitary $U$ in $\M_k(\C)
\otimes \M_k(\C)\cong \M_{k^2}(\C)$. Let $(U_t)_{t\in[0,1]}$ be a
homotopy in $\U_{k^2}(\C)$  between $U$ and $I_{k^2}$.
Let us define
$$\mathcal{A}=\{(x,z\otimes I_k);\,x\in
\M_n(A)\otimes\M_k(\C)\otimes\M_k(\C),\,z\in\M_n(\C)\}\subset\M_n(\widetilde{\M_k(A)})\otimes\M_k(\C),$$
where $z\otimes I_k$ is viewed as $z\otimes I_k\otimes I_k$ in
$$\M_n(\widetilde{\M_k(A)})\otimes\M_k(\C)\cong \M_n(\C)\otimes\widetilde{\M_k(A)}\otimes\M_k(\C).$$
Then   for any
$t\in[0,1]$, $$\mathcal{A}\to\mathcal{A};\,(x,z\otimes I_k)\mapsto((I_n\otimes
U_t)\cdot x\cdot (I_n\otimes
U_t)^{-1},z\otimes I_k)$$ is an automorphism of $C^*$-algebra.
Hence,  $$\big((I_n\otimes
U_t)\cdot (y\otimes
e_{1,1})\cdot (I_n\otimes
U_t^{-1}),\lambda\otimes I_k\big)_{t\in[0,1]}$$ is a path in
  $\U_{nk}^{\eps,r}(\widetilde{\M_k(A)})$ between $(y\otimes
e_{1,1},\lambda\otimes I_k)$ and $(\sigma(y\otimes
e_{1,1}),\lambda\otimes I_k)$. The range of $\sigma(y\otimes e_{1,1})$
being in the range of the projection $I_n\otimes e_{1,1}\otimes I_k$, we have an orthogonal sum  decomposition
$$(\sigma(y\otimes
e_{1,1}),\lambda\otimes I_k)=(\sigma(y\otimes
e_{1,1}),\lambda\otimes e_{1,1})+(0,\lambda\otimes (I_k-e_{1,1}))$$ (recall that $\lambda\otimes e_{1,1}$ and
$\lambda\otimes (I_k-e_{1,1})$  multiply $\M_n(A)\otimes \M_k(\C)
\otimes \M_k(\C)$ on the first two  factors).
By lemma \ref{lem-hom-uni}, $\lambda$ is homotopic to $I_n$ in
$\ue(\C)$ and thus   $(\sigma(y\otimes
e_{1,1}),\lambda\otimes I_k)$ is
homotopic to $(\sigma(y\otimes
e_{1,1}),\lambda\otimes e_{1,1})+(0,I_n\otimes (I_k-e_{1,1}))$  in
$\U_{nk}^{\eps,r}(\widetilde{\M_k(A)}))$ which can be viewed as
$$\diag((y,\lambda),(0,I_{k(k-1)})$$ in
$\M_k(\M_n(\widetilde{\M_k(A)})$. From this we deduce  that $[(y,\lambda)]_{\eps,r}=[(y\otimes
e_{1,1},\lambda\otimes I_k)]_{\eps,r}$
in $K_1^{\eps,r}(\widetilde{\M_k(A)})$.

For the even case, by an analogous computation, we can check that the
group homomorphisms
 $$K_0^{\eps,r}(\tilde{A})\to
K_0^{\eps,r}(\widetilde{\M_k(A)});\, [(p,q),l)]_{\eps,r}\mapsto
[(p\otimes e_{1,1}),q,l]_{\eps,r}$$
and
$$K_0^{\eps,r}(\widetilde{\M_k(A)})\to
K_0^{\eps,r}(\tilde{A});\, [(p,q),l)]_{\eps,r}\mapsto[(p,q\otimes I_k),kl]_{\eps,r}, $$   respectively induce by restriction
homomorphisms $\Phi_0:K_0^{\eps,r}({A})\to
K_0^{\eps,r}({\M_k(A)})$ and  $\Psi_0:K_0^{\eps,r}({\M_k(A)})\to
K_0^{\eps,r}({A})$ which are
inverse of each other, where in
the right hand side of the last formula, we have viewed
$p\in\M_n(\M_k(A))$ as a matrix in $\M_{nk}(A)$ and $q\otimes I_k\in
\M_n(\C)\otimes\M_k(\C)$  as a matrix in $\M_{nk}(\C)$. Since $\Phi_0$
is induced by $\iota_A$, we get from lemma \ref{lem-split} that
$\iota^{\eps,r}_{A,*}:K_0^{\eps,r}({A})\to
K_0^{\eps,r}({\M_k(A)})$ is an isomorphism.
\medskip

Let $A$ be a $C^*$-algebra filtered by $(A_r)_{r>0}$. Then $\K(\H)\otimes
A$ is filtered by $(\K(\H)\otimes A_r)_{r>0}$ and
applying proposition \ref{prop-lim} to the increasing family
$(\M_k(A)^+)_{k\in\N}$ of $C^*$-subalgebras of $\widetilde{\K(\H)\otimes
  A}$, lemmas \ref{lem-split} and \ref{lem-k+}, and the discussion
above,  we deduce the Morita equivalence for $K_*^{\eps,r}(\bullet)$.
\begin{proposition}\label{prop-morita}
If $A$ is a filtered algebra and $\H$ is a separable Hilbert space, then the homomorphism
$$A\to \K(\H)\otimes A;\,a\mapsto \begin{pmatrix}a&&\\&0&\\&&\ddots\end{pmatrix}$$
induces a ($\Z_2$-graded) semi-group isomorphism (the Morita equivalence)
$$\MM_A^{\eps,r}:K_*^{\eps,r}(A)\to K_*^{\eps,r}(\K(\H)\otimes A)$$
for any positive number $r$ and any
$\eps\in(0,1/4)$.
\end{proposition}

\subsection{Lipschitz homotopies}

%\begin{definition}
%\item Let $B$ be a $C^*$-algebra and let $L$ be a positive integer. An homotopy $(h_t)_{t\in [0,1]}$ is
%  said to be of length at most $L$ if for every partition
%  $t_0=0<t_1<\ldots<t_p=1$ of $[0,1]$, then
%  $\ds\sum_{i=1}^p\|h_{t_i}-h_{t_{i-1}}\|\lq L$.
%\end{definition}

\begin{definition}
If $A$ is a $C^*$-algebra and $C$ a positive integer, then a map
$h=[0,1]\to A$
is called $C$-Lipschitz if for every $t$ and $s$ in $[0,1]$, then
$\|h(t)-h(s)\|\lq C|t-s|$.
\end{definition}

\begin{proposition}\label{prop-bounded-hom}There exists a number $C$ such that for any
   unital filtered $C^*$-algebra $A$  and any positive number $r$ and
   $\eps<1/4$ then  :
\begin{enumerate}
\item if  $p_0$ and $p_1$ are homotopic in $\pe(A)$, then there
  exist   integers $k$ and $l$  and a $C$-Lipschitz  homotopy in $\P^{\eps,r}_{n+k+l}(A)$  between
$\diag(p_0,I_k,0_l)$ and $\diag(p_1,I_k,0_l)$.
\item if  $u_0$ and $u_1$ are homotopic in $\ue(A)$ then there
  exist an integer $k$ and a  $C$-Lipschitz  homotopy in $\U_{n+k}^{3\eps,2r}(A)$  between
$\diag(u_0,I_k)$ and $\diag(u_1,I_k)$.

\end{enumerate}
\end{proposition}
\begin{proof}\
\begin{enumerate}
\item  Notice first that  if  $p$ is an $\eps$-$r$-projection in  $A$, then the homotopy of $\eps$-$r$-projections
of  $M_2(A)$ between  $\begin{pmatrix}1& 0\\0&0\end{pmatrix}$ and $\begin{pmatrix} p&
      0\\0&1-p\end{pmatrix}$ in example
 \ref{example-homotopy} is
  $2$-Lipschitz.

 Let $(p_t)_{t\in [0,1]}$ be a homotopy between $p_0$ and $p_1$
  in $\pe(A)$. Set $\alpha=\inf_{t\in[0,1]}\frac{\eps-\|p_t^2-p_t\|}{4}$ ant let $t_0=0<t_1<\ldots<t_k=1$ be a partition of
  $[0,1]$ such that $\|p_{t_i}-p_{t_{i-1}}\|<\alpha$ for
  $i\in\{1,\ldots,k\}$.
 We construct a  homotopy of $\eps$-$r$-projections with the required property
between $\diag(p_0,I_{n(k-1)},0)$ and $\diag(p_1,I_{n(k-1)},0)$
  in
$M_{n(2k-1)}(A)$ as the composition of the
following homotopies.
\begin{itemize}
\item We can connect $\diag(p_{t_0},I_{n(p-1)},0)$ and
  $\diag(p_{t_0},I_{n},0,\ldots,I_{n},0)$  within
$\P^{\eps,r}_{n(2k-1)}(A)$ by  a  $2$-Lipschitz homotopy.
\item As we noticed at the beginning of the proof, we can connect \linebreak
  $\diag(p_{t_0},I_{n},0,\ldots,I_{n},0)$ and $\diag(p_{t_0},I_{n}-p_{t_1},p_{t_1},\ldots,I_{n}-p_{t_k},p_{t_k})$   within
$\P^{\eps,r}_{n(2k-1)}(A)$ by   a  $2$-Lipschitz homotopy.
\item 
% In view of lemma \ref{lemma-almost-closed}, there is a $2$-Lipschitz homotopy  between
The $\eps$-$r$-projections $\diag(p_{t_0},I_{n}-p_{t_1},p_{t_1},\ldots,I_{n}-p_{t_k},p_{t_k})$
and
\linebreak
$\diag(p_{t_0},I_{n}-p_{t_0},\ldots,p_{t_{k-1}},I_{n}-p_{t_{k-1}},p_{t_k})$ satisfy the norm estimate of the assumption of 
lemma \ref{lemma-almost-closed}(i) and hence then can be connected  within
$\P^{\eps,r}_{n(2k-1)}(A)$  by a ray which is clearly a 
$1$-Lipschitz homotopy.
\item Using once again the homotopy of example  \ref{example-homotopy}, we see that
$\diag(p_{t_0},I_{n}-p_{t_0},\ldots,p_{t_{k-1}},I_{n}-p_{t_{k-1}},p_{t_k})$  and $\diag(0,I_{n},\ldots,0,I_{n},p_{t_k})$
are connected  within
$\P^{\eps,r}_{n(2k-1)}(A)$ by  a $2$-Lipschitz homotopy.
\item Eventually, $\diag(0,I_{n},\ldots,0,I_{n},p_{t_k})$ and $\diag(p_{t_k},I_{n(k-1)},0)$
are connected  within
$\P^{\eps,r}_{n(2k-1)}(A)$ by  a  $2$-Lipschitz homotopy.
\end{itemize}
\item Let $(u_t)_{t\in [0,1]}$ be a homotopy between $u_0$ and $u_1$
  in $\ue(A)$.  Set $\alpha=\inf_{t\in[0,1]}\frac{\eps-\|u^*_tu_t-I_n\|}{3}$ and let $t_0=0<t_1<\ldots<t_p=1$ be a partition of
  $[0,1]$ such that $\|u_{t_i}-u_{t_{i-1}}\|<\alpha$ for
  $i\in\{1,\ldots,p\}$.
 We construct a homotopy with the required property
between $\diag(u_0,I_{2np})$ and $\diag(u_1,I_{2np})$
  within
$\U^{3\eps,2r}_{n(2p+1)}(A)$ as the composition of the
following homotopies.
\begin{itemize}
\item Since  $I_{np}$ and $\diag(u_{t_1}^*u_{t_1},\ldots,u_{t_p}^*u_{t_p})$ 
satisfy the norm estimate of the assumption of 
lemma \ref{lemma-almost-closed}(ii), then $\diag(u_{t_0},I_{np})$ is a $3\eps$-$2r$-unitary that can be connected to 
 $\diag(u_{t_0},u_{t_1}^*u_{t_1},\ldots,u_{t_p}^*u_{t_p})$  in $\U^{3\eps,2r}_{n(p+1)}(A)$ by a $1$-Lipschitz  homotopy.
\item Proceeding as in the first point of lemma \ref{lem-conjugate-proj},  we see that \linebreak
  $\diag(I_n,u_{t_1}^*,\ldots,u_{t_p}^*,I_{np})$  and $\diag(u_{t_1}^*,\ldots,u_{t_p}^*,I_{n(p+1)})$ can be connected  within
$\U^{\eps,r}_{n(2p+1)}(A)$ by a $2$-Lipschitz homotopy
and thus, in view of remark \ref{rem-prod},
$$\diag(u_{t_0},u_{t_1}^*u_{t_1},\ldots,u_{t_p}^*u_{t_p},I_{np})=$$
 $$\diag(I_n,u_{t_1}^*,\ldots,u_{t_p}^*,I_{np})\cdot\diag(u_{t_0},u_{t_1},\ldots,u_{t_p},I_{np})$$
 and
$$\diag(u_{t_1}^*,\ldots,u_{t_p}^*,I_{n(p+1)})\cdot\diag(u_{t_0},u_{t_1},\ldots,u_{t_p},I_{np})=$$
$$\diag(u_{t_1}^*u_{t_0},\ldots,u_{t_p}^*u_{t_{p-1}},u_{t_p},I_{np})$$
can be  connected  within
$\U^{3\eps,2r}_{n(2p+1)}(A)$ by a   $4$-Lipschitz homotopy.

\item   Since $\|u_{t_i}^*u_{t_{i-1}}-I_n\|<{\eps}$, we get by using once again   lemma \ref{lemma-almost-closed}(ii) that  
$\diag(u_{t_1}^*u_{t_0},\ldots,u_{t_p}^*u_{t_{p-1}},u_{t_p},I_{np})$
   and
$\diag(I_{np},u_{t_p},I_{np})$  
  can be connected 
  within $\U^{3\eps,2r}_{n(2p+1)}(A)$ by  a $1$-Lipschitz
  homotopy. 
% Therefore, we have a $1$-Lipschitz homotopy between 
% \linebreak $\diag(u_{t_1}^*u_{t_0},\ldots,u_{t_n}^*u_{t_{n-1}},u_{t_n},I_{(p-1)n})$
%    and
% $\diag(I_{(p-1)n},u_{t_n},I_{(p-1)n})$   within $\U^{3\eps,2r}_{(2p-1)n}(A)$.

\item  Eventually, $\diag(I_{np},u_{t_p},I_{np})$ can be connected  to  $
\diag(u_{t_p},I_{2np})$  within
$\U^{3\eps,2r}_{(2p+1)n}(A)$  by  a $2$-Lipschitz homotopy.
\end{itemize}
\end{enumerate}
\end{proof}

\begin{corollary}\label{cor-conjugate} There exists a control pair
  $(\alpha_h,k_h)$    such  that the following holds:

For any    unital
   filtered $C^*$-algebra $A$, any  positive numbers $\eps$ and $r$ with $\eps<\frac{1}{4\alpha_h}$ and any
  homotopic $\eps$-$r$-projections $q_0$ and $q_1$
   in $\pe(A)$, then there is for some integers $k$ and $l$  an $\alpha_h\eps$-$k_{h,\eps}r$-unitary
$W$ in  $\U^{\alpha_h\eps,k_{h,\eps}r}_{n+k+l}(A)$ such that
$$\|\diag(q_0,I_k,0_l)-W\diag(q_1,I_k,0_l)W^*\|<\alpha_h\eps.$$
\end{corollary}
\begin{proof}
According to proposition \ref{prop-bounded-hom}, we can assume that
$q_0$ and $q_1$ are connected by a $C$-Lipschitz homotopy $(q_t)_{t\in[0,1]}$,
for some universal constant $C$. Let
$t_0=0<t_1<\cdots<t_p=1$ be a partition of $[0,1]$ such that
$1/32C<|t_i-t_{i-1}|<1/16C$. With notation of  lemma
\ref{lem-conjugate}, pick  for  every integer $i$ in $\{1,\ldots,p\}$ a
$\lambda\eps$-$l_\eps$-unitary $W_i$ in $A$ such that
$\|W_iq_{t_{i-1}}W_i^*-q_{t_i}\|<\lambda\eps$. If we set $W=W_p\cdots W_1$, then
$W$ is a $3^p\lambda\eps$-$pl_\eps r$-unitary such that
$\|Wq_0W^*-q_1\|<2^{p}\lambda\eps$. Since $p<2C$, we get the result.
\end{proof}

%%%%%%%%%%%%%%%%%%%%%%%%%%%%%%%%%%%%%%%%%%%%%%%%%%%%%%%%%%%%%%%%%%%%%%%%%%%%%%%%%%%%%%%%%
%%%%%%%%%%%%%%%%%%%%%%%%%%%%%%%%%%%%%%%%%%%%%%%%%%%%%%%%%%%%%%%%%%%%%%%%%%%%%%%%%%%%%
\section{Controlled morphisms}
As we shall see in section \ref{sec-exact-sequence}, usual maps in $K$-theory
such as boundary maps factorize  through semi-group homomorphism of quantitative $K$-theory groups with expansion of norm
control and propagation
controlled by a control pair.  This motivates the notion of controlled morphisms for quantitative $K$-theory in this section.

Recall that a controlled pair  is a pair $(\lambda,h)$, where
\begin{itemize}
 \item $\lambda >1$;
\item  $h:(0,\frac{1}{4\lambda})\to (0,+\infty);\, \eps\mapsto h_\eps$  is a map such that there exists a non-increasing map
$g:(0,\frac{1}{4\lambda})\to (0,+\infty)$, with $h\lq g$.
\end{itemize}
 The set of control pairs is equipped with a partial order:
 $(\lambda,h)\lq (\lambda',h')$ if $\lambda\lq\lambda'$ and $h_\eps\lq h'_\eps$
for all $\eps\in (0,\frac{1}{4\lambda'})$
\subsection{Definition and main properties}\label{subsection-controlled-morphism}  
For any filtered $C^*$-algebra $A$, let us define the families $\K_0(A)=(K_0^{\eps,r}(A))_{0<\eps<1/4,r>0}$,
$\K_1(A)=(K_1^{\eps,r}(A))_{0<\eps <1/4,r>0}$ and $\K_*(A)=(K_*^{\eps,r}(A))_{0<\eps<1/4,r>0}$.
\begin{definition}
 Let $(\lambda,h)$ be a controlled pair, let $A$ and
$B$ be filtered $C^*$-algebras, and let $i,\,j$ be elements of  $\{0,1,*\}$. A $(\lambda,h)$-controlled morphism
$$\F:\K_i(A)\to\K_j(B)$$ is a family $\F=(F^{\eps,r})_{0<\eps<\frac{1}{4\lambda},r>0}$ of semigroups homomorphisms
 $$F^{\eps,r}:K_i^{\eps,r}(A)\to K_j^{\lambda\eps,h_\eps r}(B)$$ such that for any positive
numbers $\eps,\,\eps',\,r$ and $r'$ with
$0<\eps\lq\eps'<\frac{1}{4\lambda}$ and $h_\eps r\lq h_{\eps'}r'$, we have
$$F^{\eps',r'}\circ \iota_i^{\eps,\eps',r,r'}=\iota_j^{\lambda\eps,\lambda\eps', h_\eps r,h_{\eps'}r'}\circ F^{\eps,r}.$$
\end{definition}
If it is not necessary  to specify the control pair, we will just say that $\F$ is a controlled morphism.

Let $A$ and $B$ be filtered algebras. Then
it  is straightforward to check that if $\F:\K_i(A)\to\K_j(B)$ is a $(\lambda,h)$-controlled morphism,
then there is   group homomorphism
$F:K_i(A)\to K_j(B)$ uniquely defined by  $F\circ \iota_i^{\eps,r}=\iota_{j}^{\lambda\eps,h_\eps r}\circ F^{\eps,r}$.
The homomorphism $F$ will be called the $(\lambda,h)$-controlled homomorphism induced by $\F$.
A homomorphism $F:K_i(A)\to K_j(B)$ is called $(\lambda,h)$-controlled  if  it is induced by a $(\lambda,h)$-controlled morphism. If we don't need to specify
the control pair $(\lambda,h)$, we will just say that $F$ is a controlled homomorphism.
\begin{example}\label{example-controlled-morphism}
\begin{enumerate}\
 \item  Let $A=(A_r)_{r>0}$ and $B=(B_r)_{r>0}$ be two filtered
$C^*$-algebras and let
 $f:A\to B$ be a homomorphism. Assume that there exists   $d>0$ such that $f(A_r)\subset B_{dr}$ for all positive $r$. Then $f$
gives rise to a bunch of semi-group homomorphisms
$\left(f_*^{\eps,r}:K_*^{\eps,r}(A)\to K_*^{\eps,dr}(B)\right)_{0<\eps <\frac{1}{4},r>0}$
and hence to a
$(1,d)$-controlled morphism $f_*:\K_*(A)\to\K_*(B)$.
\item The bunch of semi-group isomorphisms
$$\left( \MM_A^{\eps,r}:K_*^{\eps,r}(A)\to K_*^{\eps,r}(\K(\H)\ts A) \right)_{0<\eps <\frac{1}{4},r>0}$$ of proposition \ref{prop-morita}
defines a
$(1,1)$-controlled morphism $$\MM_A:\K_*(A)\to\K_*(\K(\H)\ts A)$$ and $$\MM_A^{-1}:\K_*(\K(\H)\ts A)\to\K_*(A)$$ inducing
the Morita equivalence in $K$-theory.
\end{enumerate}
\end{example}
If $(\lambda,h)$ and $(\lambda',h')$ are two control pairs, define
$$h*h': (0,\frac{1}{4\lambda\lambda'})\to (0,+\infty);\, \eps\mapsto h_{\lambda'\eps}h'_\eps.$$
Then $(\lambda\lambda',h*h')$  is a control pair.
Let $A$, $B_1$ and $B_2$ be filtered $C^*$-algebras, let  $i,j$ and $l$ be in $\{0,1,*\}$ and let
$\F=(F^{\eps,r})_{0<\eps<\frac{1}{4\alpha_\F},r>0}:\K_i(A)\to \K_j(B_1)$  be
a $(\alpha_\F,k_\F)$-controlled morphism, let   $\G=(G^{\eps,r})_{0<\eps<\frac{1}{4\alpha_\G},r>0}:\K_j(B_1)\to \K_l(B_2)$
be a $(\alpha_\G,k_\G)$-controlled morphism. Then $\G\circ\F:\K_i(A)\to \K_l(B_2)$ is the $(\alpha_\G\alpha_F,k_\G*k_\F)$-controlled
morphism defined by the family $(\G^{\alpha_\F\eps,k_{\F,\eps}r}\circ \F^{\eps,r})_{0<\eps<\frac{1}{4\alpha_\F\alpha_\G},r>0}$.
 \begin{remark}\label{rem-morita-natural}
  The Morita equivalence for quantitative $K$-theory is natural, i.e $$\MM_B\circ f=(Id_{\K(\H))}\ts f)\circ \MM_A$$
for any homomorphism $f:A\to B$ of
filtered $C^*$-algebras.
 \end{remark}

\begin{notation}
 Let $A$ and $B$ be filtered $C^*$-algebras, let $(\lambda,h)$ be a control pair,
 and
let $\F=(F^{\eps,r})_{0<\eps<\frac{1}{4\alpha_\F},r>0}:\K_i(A)\to \K_j(B)$ (resp. $\G=(G^{\eps,r})_{0<\eps<\frac{1}{4\alpha_\G},r>0}$) be a
$(\alpha_\F,k_\F)$-controlled morphism (resp. a $(\alpha_\G,k_\G)$-controlled morphism). Then we write
$\F\aeq\G$ if
\begin{itemize}
 \item $(\alpha_\F,k_\F)\lq (\lambda,h)$ and $(\alpha_\G,k_\G)\lq (\lambda,h)$.
\item for every $\eps$ in $(0,\frac{1}{4\lambda})$ and $r>0$, then
$$\iota^{\alpha_\F\eps,\lambda\eps,k_{\F,\eps}r,h_\eps r}_j\circ F^{\eps,r}=\iota^{\alpha_\G\eps,\lambda\eps,k_{\G,\eps}r,h_\eps r}_j\circ G^{\eps,r}.$$
\end{itemize}
\end{notation}

If $\F$ and $\G$ are controlled morphisms such that $\F\aeq\G$  for a control pair $(\lambda,h)$, then $\F$ and $\G$ induce the same morphism in $K$-theory.

%%%%%%%%%%%%%%%%%%%%%%%%%%%%%%%%%%%%%%%%%%%%%%%%%%%%%%%%%%%%%%%%%%%%%%%%%%%%%%%%%%%%%%%%%
%%%%%%%%%%%%%%%%%%%%%%%%%%%%%%%%%%%%%%%%%%%%%%%%%%%%%%%%%%%%%%%%%%%%%%%%%%%%%%%%%%%%%%%%%%%%

\begin{remark}\label{rem-composition}
 Let $\F:\K_i(A_2)\to \K_j(B_1)$  (resp. $\F':\K_i(A_2)\to \K_j(B_1)$) be
a $(\alpha_\F,k_\F)$-controlled (resp. a $(\alpha_{\F'},k_{\F'})$-controlled) morphisms and let
$\G:\K_{i'}(A_1)\to \K_i(A_2)$  (resp.  $\G':\K_j(B_1)\to \K_l(B_2)$) be
a $(\alpha_\G,k_\G)$-controlled (resp. a $(\alpha_{\G'},k_{\G'})$-controlled)  morphism. Assume that
$\F\aeq\F'$ for a control pair $(\lambda,h)$, then
\begin{itemize}
 \item  $\G'\circ\F\stackrel{(\alpha_{\G'}\lambda,k_{\G'}*h)}{\sim}\G'\circ\F';$
\item  $\F\circ \G \stackrel{(\alpha_{\G}\lambda,h*k_{\G})}{\sim}\F'\circ \G.$
\end{itemize}
\end{remark}
If $i$ is an element in $\{0,1,*\}$ and $A$ a filtered $C^*$-algebra, we denote by $\Id_{K_i(A)}$ the
controlled morphism induced by $Id_A$.

Let $\F:\K_i(A_1)\to \K_{i'}(B_1)$,  $\F':\K_j(A_2)\to \K_l(B_2)$,  $\G:\K_i(A_1)\to \K_j(A_2)$ and
$\G':\K_{i'}(B_1)\to \K_l(B_2)$ be controlled
morphisms and let $(\lambda,h)$ be a control pair. Then the diagram
$$\begin{CD}
\K_{i'}(B_1)@>\G'>> \K_l(B_2)  \\
         @A\F AA
         @AA\F'A\\
 \K_i(A_1)@>\G>> \K_j(A_2)
\end{CD}$$
is called ${\mathbf{(\lambda,h)}}${\bf -commutative} (or  $\mathbf{(\lambda,h)}${\bf-commutes})  if $\G' \circ \F\aeq \F'\circ \G$.

 \begin{definition}
  Let $(\lambda,h)$ be a control pair, and let $\F:\K_i(A)\to \K_j(B)$ be a $(\alpha_\F,k_\F)$-controlled morphism with
$(\alpha_\F,k_\F)\lq (\lambda,h)$.
\begin{itemize}
 \item $ \F$ is called left $(\lambda,h)$-invertible if there exists a controlled morphism $$\G:\K_j(B)\to\K_i(A)$$ such that
$\G\circ\F\aeq \Id_{K_i(A)}$. The controlled morphism $\G$ is then called a left $(\lambda,h)$-inverse for $\F$. 
Notice that definition of $\aeq$  implies that $(\alpha_\F\alpha_\G,k_\F*k_\G)\lq (\lambda,h)$.
\item $ \F$ is called right $(\lambda,h)$-invertible if there exists a controlled morphism $$\G:\K_j(B)\to\K_i(A)$$
such that $\F\circ\G\aeq \Id_{K_i(B)}$.
The controlled morphism $\G$ is then called a
right $(\lambda,h)$-inverse for $\F$.
\item $ \F$ is called  $(\lambda,h)$-invertible or a $(\lambda,h)$-isomorphism if there exists a controlled morphism 
$$\G:\K_j(B)\to\K_i(A)$$ which is
 a left $(\lambda,h)$-inverse and a right  $(\lambda,h)$-inverse for $\F$. The controlled morphism $\G$ is then called a
$(\lambda,h)$-inverse for $\F$ (notice that we have in this case necessarily $(\alpha_\G,k_\G)\lq (\lambda,h)$).
\end{itemize}

 \end{definition}

We can check easily that indeed, if  $ \F$ is  left $(\lambda,h)$-invertible and right $(\lambda,h)$-invertible,
then there exists a control pair $(\lambda',h')$ with $(\lambda,h)\lq(\lambda',h')$, depending only on $(\lambda,h)$
such that $\F$ is $(\lambda',h')$-invertible.
\begin{definition}
  Let $(\lambda,h)$ be a control pair  and let $\F:\K_i(A)\to \K_j(B)$ be a $(\alpha_\F,k_\F)$-controlled
morphism.
\begin{itemize}
 \item
 $\F$ is called $(\lambda,h)$-injective if  $(\alpha_\F,k_\F)\lq(\lambda,h)$ and for any $0<\eps<\frac{1}{4\lambda}$, any $r>0$  and  any $x$ in
$K_i^{\eps,r}(A)$, then $F^{\eps,r}(x)=0$ in $K_j^{\alpha_\F\eps,k_{\F,\eps}r}(B)$ implies that
$\iota_{i}^{\eps,\lambda\eps,r,h_\eps r}(x)=0$ in $K_i^{\lambda\eps,h_\eps r}(A)$;
\item $\F$ is called $(\lambda,h)$-surjective, if for any $0<\eps<\frac{1}{4\lambda\alpha_\F}$, any $r>0$  and  any $y$ in
$K_j^{\eps,r}(B)$, there exists an element $x$ in $K_i^{\lambda\eps,h_\eps r}(A)$
such that $F^{\lambda\eps,h_{\lambda\eps}r}(x)=\iota_j^{\eps,\alpha_\F\lambda\eps,r,k_{\F,\lambda\eps}h_\eps r}(y)$ in
$K_j^{\alpha_\F\lambda\eps,k_{\F,\lambda\eps}h_\eps r}(B)$.
\end{itemize}\end{definition}

\begin{remark}\label{rem-inj-surj-inv}\
\begin{enumerate}
              \item If $\F:\K_1(A)\to \K_i(B)$ is a   $(\lambda,h)$-injective controlled morphism. Then according to
lemma \ref{lem-k1-almost-group}, there exists a control pair
$(\lambda',h')$ with $(\lambda,h)\lq(\lambda',h')$ depending only on $(\lambda,h)$ such that
for any $0<\eps<\frac{1}{4\lambda'}$, any $r>0$  and  any $x$ and $x'$  in
$K_1^{\eps,r}(A)$ , then $F^{\eps,r}(x)=F^{\eps,r}(x')$ in $K_i^{\alpha_\F\eps,k_{\F,\eps}r}(B)$ implies that
$\iota_{1}^{\eps,\lambda'\eps,r,h'_\eps r}(x)=\iota_{1}^{\eps,\lambda'\eps,r,h'_\eps r}(x')$ in $K_1^{\lambda'\eps,h'_\eps r}(A)$;
\item It is straightforward to check that if  $\F$ is left $(\lambda,h)$-invertible, then $\F$ is $(\lambda,h)$-injective and that if $\F$ is
 right  $(\lambda,h)$-invertible, then there  exists a control       pair
$(\lambda',h')$ with $(\lambda,h)\lq(\lambda',h')$, depending only on $(\lambda,h)$ such  that    $\F$ is $(\lambda',h')$-surjective.
\item On the other hand, if       $\F$ is $(\lambda,h)$-injective  and $(\lambda,h)$-surjective, then there  exists a control       pair
$(\lambda',h')$ with $(\lambda,h)\lq(\lambda',h')$, depending only on $(\lambda,h)$ such  that  $\F$ is a $(\lambda',h')$-isomorphism.
\end{enumerate}
\end{remark}
\subsection{Controlled exact sequences}
\begin{definition}
Let $(\lambda,h)$ be a control pair,
\begin{itemize}
 \item
Let
$\F=(F^{\eps,r})_{0<\eps<\frac{1}{4\alpha_\F},r>0}:\K_i(A)\to \K_j(B_1)$  be
a $(\alpha_\F,k_\F)$-controlled morphism, and let   $\G=(G^{\eps,r})_{0<\eps<\frac{1}{4\alpha_\G},r>0}:\K_j(B_1)\to \K_l(B_2)$
be a $(\alpha_\G,k_\G)$-controlled morphism, where  $i,j$ and $l$ are  in $\{0,1,*\}$ and
  $A$, $B_1$ and $B_2$ are  filtered $C^*$-algebras.
Then the composition
$$\K_i(A)\stackrel{\F}{\to}\K_j(B_1)\stackrel{\G}{\to}\K_l(B_2)$$ is said to be $(\lambda,h)$-exact at $\K_j(B_1)$ if
 $\G\circ\F=0$ and if
 for any $0<\eps<\frac{1}{4\max\{\lambda\alpha_\F,\alpha_\G\}}$, any $r>0$  and  any $y$ in
$K_j^{\eps,r}(B_1)$  such that  $G^{\eps,r}(y)=0$ in $K_j^{\alpha_\G\eps,k_{\G,\eps}r}(B_2)$, there exists an
  element $x$ in $K_i^{\lambda\eps,h_\eps r}(A)$
such that $$F^{\lambda\eps,h_{\lambda\eps}r}(x)=\iota_j^{\eps,\alpha_\F\lambda\eps,r,k_{\F,\lambda\eps}h_\eps r}(y)$$ in
$K_j^{\alpha_\F\lambda\eps,k_{\F,\lambda\eps}h_\eps r}(B_1)$.
\item  A sequence of controlled morphisms
$$\cdots\K_{i_{k-1}}(A_{k-1})\stackrel{\F_{k-1}}{\to}\K_{i_{k}}(A_{k})\stackrel{\F_{k}}{\to}
\K_{i_{k+1}}(A_{k+1})\stackrel{\F_{k+1}}{\to}\K_{i_{k+2}}(A_{k+2})\cdots$$ is called $(\lambda,h)$-exact if for every $k$,
the composition   $$\K_{i_{k-1}}(A_{k-1})\stackrel{\F_{k-1}}{\to}\K_{i_{k}}(A_{k})\stackrel{\F_{k}}{\to}
\K_{i_{k+1}}(A_{k+1})$$ is $(\lambda,h)$-exact at $\K_{i_{k}}(A_{k})$.
\end{itemize}
\end{definition}

\begin{remark}
 If the composition $\K_i(A)\stackrel{\F}{\to}\K_1(B_1)\stackrel{\G}{\to}\K_l(B_2)$ is $(\lambda,h)$-exact, then according to
lemma \ref{lem-k1-almost-group}, there exists a control pair
$(\lambda',h')$ with $(\lambda,h)\lq(\lambda',h')$ depending only on $(\lambda,h)$, such that
for any $0<\eps<\frac{1}{4\max\{\lambda'\alpha_\F,\alpha_\G\}}$, any $r>0$  and  any $y$ and $y'$  in
$K_1^{\eps,r}(B_1)$ , then $G^{\eps,r}(y)=G^{\eps,r}(y')$ in $K_j^{\alpha_\F\eps,k_{\F,\eps}r}(B)$ implies that there exists an
  element $x$ in $K_i^{\lambda'\eps,h'_\eps r}(A)$
such that $$\iota_j^{\eps,\alpha_{\F}\lambda'\eps,r,k_{\F,\lambda'\eps}h'_\eps r}(y')=
\iota_j^{\eps,\alpha_\F\lambda'\eps,r,k_{\F,\lambda'\eps}h'_\eps r}(y)+F^{\lambda'\eps,h'_{\eps}r}(x)$$ in
$K_1^{\alpha_\F\lambda'\eps,k_{\F,\lambda\eps}h'_\eps r}(B_1)$.
\end{remark}

\section{Extensions of filtered $C^*$-algebras}\label{sec-exact-sequence}
The aim of this section is to establish  a   controlled exact sequence for quantitative $K$-theory with respect to extension of filtered
$C^*$-algebras  admitting  a completely positive cross section that preserves the filtration. We also prove that for these extensions,
the boundary maps are induced by controlled morphisms. As in $K$-theory, one is a map of  exponential type and the other
  is  an index type map, and the later in turn fits in a long $(\lambda,h)$-controlled exact sequence for some universal control pair $(\lambda,h)$.

\subsection{Semi-split filtered extensions}

Let $A$ be a  $C^*$-algebra filtered by $(A_r)_{r>0}$ and let
$J$ be an ideal of $A$. Then  $A/J$ is
filtered by  $((A/J)_r)_{r>0}$, where $(A/J)_r$ is the image of $A_r$ under the projection $A\to A/J$.
 Assume that the $C^*$-algebra extension
$$0\to J\to A\to A/J\stackrel{q}{\to} 0$$ admits a contractive  filtered  cross-section
$s:A/J\to A$, i.e  such that $s((A/ J)_r))\subset A_r$ for any
positive number. For any $x\in
J$ and any number $\eps>0$ there exists a positive number $r$ and  an
element $a$ of  $A_r$
such that $\|x-a\|<\eps$. Let us set $y=a-s\circ q(a)$. Then $y$
belongs to $A_r\cap J$ and moreover
\begin{eqnarray*}
\|y-x\|&=&\|a-x+s\circ q(x-a) \|\\
&\lq&\|a-x\|+\|s\circ q(a-x) \|\\
&\lq&\|a-x\|+\| q(x-a) \|\\
&\lq&2\eps.
\end{eqnarray*}
Hence,  $\ds\bigcup_{r>0} (A_r\cap J)$ is dense in $J$ and therefore $J$ is
filtered by  $(A_r\cap J)_{r>0}$.

\begin{definition}
Let $A$ be a $C^*$-algebra filtered by $(A_r)_{r>0}$ and let $J$ be an
ideal of $A$. The  extension of $C^*$-algebras
$$0\to J\to A\to A/J\to 0$$ is said to be filtered and semi-split (or a semi-split extension of filtered $C^*$-algebras) if there
exists a completely positive cross-section  $$s:A/J\to A$$ such that
$$s((A/ J)_r))\subset A_r$$  for any
number $r>0$. Such a cross-section is said to be semi-split and filtered.
\end{definition}
We have the following analogous of the lifting property for unitaries of the neutral  component.
\begin{lemma}\label{lem-lift} There exists a control pair $(\alpha_e,k_e)$  such that for any
semi-split extension of filtered $C^*$-algebras $$0\longrightarrow J \longrightarrow A \stackrel{q}{\longrightarrow}
A/J\longrightarrow 0$$  with $A$ unital,  the following holds:
 for every  positive numbers  $r$ and  $\eps$ with $\eps<\frac{1}{4\alpha_e}$  and  any $\eps$-$r$-unitary $V$
 homotopic to $I_n$ in
  $\U^{\eps,r}_{n}(A/J)$, then for some integer $j$, there exists a
  $\alpha_e\eps$-$k_{e,\eps} r$-unitary $W$  homotopic to $I_{n+j}$ in
  $\U^{\alpha_e\eps,k_{e,\eps} r}_{n+j}(A)$ and such that
$\|q(W)-\diag(V,I_j)\|<\alpha_e\eps$.
\end{lemma}
\begin{proof}
According to proposition \ref{prop-bounded-hom}, we can assume that
$V$ and $I_n$ are connected by a $C$-Lipschitz homotopy $(V_t)_{t\in[0,1]}$,
for some universal constant $C$. Let
$t_0=0<t_1<\cdots<t_p=1$ be a partition of $[0,1]$ such that
$1/16C<|t_i-t_{i-1}|<1/8C$. Then we get that $\|V_{i-1}-V_i\|< 1/8$
and hence $\|V_{i-1}V_i^*-I_n\|< 1/2$. Let $l_\eps$ be the smallest integer such that
$\sum_{k\gq l_\eps+1} 2^{-k}/k<\eps$ and $\sum_{k\gq l_\eps+1} \log ^k
2/k!<\eps$  and let us consider the polynomial functions
$P_\eps(x)=\sum_{k=0}^{l_\eps}x^k/k!$ and
$Q_\eps(x)=-\sum_{k=1}^{l_\eps}x^k/k$.
% Then we get that
%$\|(1-z)^t-P_\eps(tQ_\eps(z))\|\lq (e^{\log 2}+1)\eps$ for all $t$ in
%$[0,1]$ and $z$ in $\C$ with $|z|<1/2$.
We get then  $\|V_{i-1}V_i^*-P_\eps\circ Q_\eps(1-V_{i-1}V_i^*)\|\lq 3\eps$. Choose a completely positive section $s(A/J)\to A$ such
that $s(1)=1$ and let us  set $W_i^t=P_\eps(s(tQ_\eps(I_n-V_{i-1}V_i^*)))$
for $t$ in $[0,1]$ and $i$ in $\{1,\ldots,p\}$. Since
$V_{i-1}V_i^*$ is closed to the unitary
$V_{i-1}V_i^*(V_{i}V_{i-1}^*V_{i-1}V_i^*)^{-1/2}$, then $W_i^t$ is
uniformly  (in $t$  and $i$) closed to $\exp ts(\log
(V_{i-1}V_{i}^*(V_{i}V_{i-1}^*V_{i-1}V_i^*)^{-1/2}))$ which is unitary (the logarithm is  well defined since
$V_{i-1}V_i^*(V_{i}V_{i-1}^*V_{i-1}V_i^*)^{-1/2}$ is closed to $I_n$)
and hence $W_i^t$ is a $\alpha\eps$-$2l_\eps r$-unitary for some
universal $\alpha$. Hence  $W_i^1$ is a $\alpha\eps$-$2l_\eps
r$-unitary in  $\U^{\eps,r}_{n}(A)$   homotopic to $I_n$ and such that
$\|q(W_i^1)-V_{i-1}V_i^*\|<3\alpha\eps$. If we set now $W=W_1^1\cdots W_p^1$ and since
$p\lq 16C$, then $W$ satisfies the required property.
\end{proof}

\begin{lemma}\label{lemma-bound}
 There exists a control pair $(\alpha,k)$ such that for any semi-split extension of filtered $C^*$-algebras
$0\to J\to A\to A/J\to 0$ with $A$ unital,
any semi-split
filtered cross section $s:A/J\to A$ with $s(1)=1$ and any $\eps$-$r$-projection $p$ in $A/J$ with $0<\eps<\frac{1}{4\lambda}$, there exists an element
$y_p$ in $J_{k_\eps r}$ such that
   $\|1+y_p-e^{2\imath\pi s(k_0(p))}\|<\alpha\eps/3$. In particular $1+y_p$ is a $\alpha\eps$-$k_\eps r$-unitary of  $J^+$;
\end{lemma}
\begin{proof}
Let $l_\eps$ be the smallest integer such that
  $$\displaystyle\sum_{l=l_\eps+1}^{+\infty}10^l/l!<\eps.$$
Let us define  $z_p=\displaystyle\sum_{l=0}^{l_\eps}\frac{(2\imath\pi s(p))^l}{l!}$.
Then $z_p$ belongs to $\M_n(A_{l_\eps r})$ and we have
 \begin{eqnarray*}
\left\|z_p-e^{2\imath\pi
    s(\kappa_0(p))}\right\|&=&\left\|\sum_{l=0}^{l_\eps}\frac{(2\imath\pi
    s(p))^l}{l!}-\sum_{l=0}^{+\infty}\frac{(2\imath\pi s(\kappa_0(p))^l}{l!}\right\|\\
&\lq&\left\|\sum_{l=0}^{l_\eps}\frac{(2\imath\pi s(p))^l-(2\imath\pi s(\kappa_0(p)^l)}{l!}\right\|+\left\|\sum_{l=l_\eps+1}^{+\infty}\frac{(2\imath\pi s(\kappa_0(p))^l}{l!}\right\|\\
&\lq&\|s(p)-s(\kappa_0(p))\|e^{10}+\eps\\
&\lq&(2e^{10}+1)\eps.
\end{eqnarray*}
 If we set $y_p=z_p-s\circ q(z_p)$, then
$y_p\in\M_n(J\cap A_{l_\eps r})$ and
\begin{eqnarray*} \|z_p-(1+y_p)\| &=&\|s\circ q(z_p)-1\|\\
&\lq&\left\|q\left(z_p-e^{2\imath\pi
    s(\kappa_0(p))}\right)\right\|\\
&<&\lambda\eps,
\end{eqnarray*}
with $\lambda=(2e^{10}+1)$.
 Therefore we have  $\|1+y_p-e^{2\imath\pi
    s(\kappa_0(p))}\|<2\lambda\eps$. The end of the statement is then a consequence of lemma \ref{lemma-almost-closed}.

\end{proof}

%\begin{lemma}\label{prop-bound} There is a positive  number $\alpha>1$ and a non
%  increasing function $$\left(0,\frac{1}{4\alpha}\right)\to\N;\,\eps\mapsto
%  k_\eps$$ such that for any   filtered semi-split  exact sequence of
%$C^*$-algebras  $$0\longrightarrow J \longrightarrow A \stackrel{q}{\longrightarrow}
%A/J\longrightarrow 0,$$ any positive number $r$ and any
%  $\eps\in (0,\frac{1}{4\alpha})$, we have maps
%\begin{itemize}
%\item  $\Lambda_{n,J,A}^{\eps,r} : \pe(A/J)\longrightarrow
%  \U_{n}^{\alpha\eps,k_\eps r}(J)$;
%\item  ${\Lambda'}_{n,J,A}^{\eps,r} : \ue(A/J)\longrightarrow
%  \P_{2n}^{\alpha,k_\eps r}(J)$;
%\end{itemize}
%satisfying the following conditions:
%\begin{itemize}
%\item  $\partial_{J,A} [\kappa_0(p)]=[\Lambda_{n,J,A}^{\eps,r}(p)]$ in
%$K_1(J)$ ;
%\item
%  $\partial_{J,A}[u]=[\kappa_0({\Lambda'}_{\alpha,n}^r(u))]-\left[\begin{pmatrix} I_n&0\\0&0\end{pmatrix}\right]$  in
%$K_0(J)$,
%\end{itemize}for any  $p$ in $\pe(A/J)$ and any $u$ in $\ue(A/J)$,
%where $$\partial_{J,A}:K_{*}(A/J)\to K_{*}(J)$$ is the (odd degree) boundary map associated to the
%exact sequence.
%\end{lemma}

\subsection{Controlled boundary maps}\label{subsection-controlled-boundary-maps}
For any extension $0\to J\to A\to A/J\to 0$ of $C^*$-algebras we denote by $\partial_{J,A}:K_{*}(A/J)\to K_{*}(J)$ the associated
(odd degree) boundary map.
\begin{proposition}\label{prop-bound}
There exists a control pair $(\alpha_\DD,k_\DD)$  such that for any    semi-split  extension of filtered
$C^*$-algebras
 $$0\longrightarrow J \longrightarrow A \stackrel{q}{\longrightarrow}
A/J\longrightarrow 0,$$ there exists a $(\alpha_\DD,k_\DD)$-controlled  morphism of odd degree
 $$\DD_{J,A}=(\partial_ {J,A}^{\eps,r})_{0<\eps\frac{1}{4\alpha_\DD},r}: \K_*(A/J)\to
  \K_*(J)$$
 which induces in $K$-theory   $\partial_{J,A}:K_{*}(A/J)\to K_{*}(J)$.

\end{proposition}

\begin{proof}
 Let $s:A/J\to A$ be a semi-split filtered cross-section. Let us first prove the result when $A$ is unital.
\begin{enumerate}
\item Let $p$ be an element of  $\pe(A/J)$. Then
  $\partial_{J,A}([\kappa_0(p)])$ is the class of $e^{2\imath\pi
    s(\kappa_0(p))}$ in $K_1(J)$. Fix a control pair $(\alpha,k)$ as in  lemma \ref{lemma-bound} and
pick any $y_p$ in $M_n(J_{k_\eps r})$ such that
   $\|1+y_p-e^{2\imath\pi s(k_0(p))}\|<\alpha\eps/3$. Then  $1+y_p$ is an $\alpha\eps$-$k_\eps r$-unitary of  $M_n(J^+)$, and according to
lemma  \ref{lemma-almost-closed}, any two such $\alpha\eps$-$k_\eps r$-unitaries are homotopic in $U_n^{3\alpha\eps,k_\eps r}(J^+)$.
Applying lemma  \ref{lemma-bound} to $A/J[0,1]$, we see that the map

$$\pe(A/J)\longrightarrow \U_n^{3\alpha\eps,k_\eps r}(J^+);\,p\mapsto 1+y_p$$ preserves homotopies and hence gives rise
to a bunch of well defined semi-group homomorphism
$$\partial_ {J,A}^{\eps,r}:K_0^{\eps,r}(A/J)\longrightarrow
  K_1^{3\alpha\eps,k_\eps r}(J);\, [p,l]_{\eps,r}\mapsto
  [1+y_p]_{3\alpha\eps,k_\eps r}$$ which in the even case satisfies the required properties for a controlled homomorphism.
 \item In the odd case, we follow the route of \cite[Chapter 8]{we}.
For any element $u$ of  $\ue(A/J)$, pick any element $v$ in some
  $\U_j^{\eps,r}(A/J)$ such that $\diag(u,v)$ is homotopic to
  $I_{n+j}$ in
 $\U_{n+j}^{3\eps,2r}(A/J)$ (we can choose in view of lemma \ref{lem-k1-almost-group}           
 $v=u^*$). According to lemma \ref{lem-lift}, and up to replace $v$
 by $\diag(v,I_k)$ for some integer $k$, there exists an element $w$
 in $\U_{n+j}^{3\alpha_e\eps,2k_{e,3\eps}r}(A)$ such that
 $\|q(w)-\diag(u,v)\|\lq 3\alpha_e\eps$.
Let us set $x=w\diag(I_n,0)w^*$. Then $x$ is an element in
  $\P_{n+j}^{6\alpha_e\eps,4k_{e,3\eps}r}(A)$ such that
  $\|q(x)-\diag(I_n,0)\|\lq 9\alpha_e\eps$.

Let us  set now
$h=x-\diag(I_n,0)-s\circ q(x-\diag(I_n,0))$.
Then
$h$ is a self-adjoint  element of  $\M_{2n}(A_{4k_{e,3\eps}r}\cap J)$ such
that $$\left\|x-\diag(I_n,0)-h\right\|\lq
9\alpha_e\eps,$$ and therefore
$h+\diag(I_n,0)$ belongs to
$\P_{n+j}^{45\alpha_e\eps,4k_{e,3\eps}r}(J)$.
Define then $$\partial_{J,A}^{\eps,r}([u]_{\eps,r})=
\left[h+\diag(I_n,0),n\right]_{450\alpha_e\eps,4k_{e,3\eps}r}.$$
It is straightforward to check that (compare with \cite[Chapter 8]{we}).
\begin{itemize}
\item two choice of elements satisfying the conclusion of   lemma \ref{lem-lift} relatively to $\diag(u,v)$ give
  rise to homotopic elements
  $\P_{n+j}^{450\alpha_e\eps,4k_{e,3\eps}r}(J)$ (this is a consequence
  of lemma \ref{lemma-almost-closed}).
\item Replacing $u$ by $\diag(u,I_m)$
and $v$ by $\diag(v,I_k)$ gives also rise to  the same element of $K_0^{450\alpha_e\eps,4k_{e,3\eps}r}(J)$.
\end{itemize}

Applying now lemma \ref{lem-lift} to the exact sequence
$$0\to J[0,1]\to A[0,1]\to A/J[0,1]\to 0,$$  we get that
$\partial_{J,A}^{\eps,r}([u]_{\eps,r})$
\begin{itemize}
\item only depends on the class of $u$ in $K_1^{\eps,r}(A/J)$;
\item does not depend on the choice of $v$ such that $\diag(u,v)$ is
  connected to $I_{n+j}$ in $\U_{n+j}^{\eps,r}(A/J)$.
\end{itemize}

\end{enumerate}
\begin{itemize}
\item If $A$ is not unital, use the exact sequence
$$0\to J\to \widetilde{A}\to \widetilde{A/J}\to 0$$ to define
$\partial_{J,A}^{\eps,r}$ as the composition
$$K_0^{\eps,r}(A/J)\hookrightarrow K_0^{\eps,r}(\widetilde{A/J})
\stackrel{\partial_{J,\widetilde{A}}^{\eps,r}}{\longrightarrow}
K_1^{450\alpha_e\eps,4k_{s,3\eps}r}(J),$$ where the inclusion in the
composition is induced by the inclusion $A/J\hookrightarrow
\widetilde{A/J}\cong\widetilde{A}/J$.
\item Since the set of filtered semi-split cross-section
$s:A/J\to A$ such that $s((A/J)_r)\subset A_r$ is convex, the
definition of $\partial_{J,A}^{\eps,r}$ actually does not depend on the
choice of such a section.
\item Using lemma \ref{lemma-almost-closed}, it is plain to check that for a suitable control pair
$(\alpha_\DD,k_\DD)$, then
 $\DD_{J,A}=(\partial_ {J,A}^{\eps,r})_{0<\eps\frac{1}{4\alpha_\DD},r}$ is a $(\alpha_\DD,k_\DD)$-controlled morphism
 inducing the (odd degree) boundary map $\partial_{J,A}:K_*(A/J)\to K_*(J)$.
\end{itemize}

\end{proof}
%%%%%%%%%%%%%%%%%%%%%%%%%%%%%%%%%%%%%%%
%%%%%%%%%%%%%%%%%%%%%%%%%%%%%%%%%%%%%%
%\begin{remark}\
%\begin{enumerate}

%\item The map $\Lambda_{n,J,A}^{\eps,r}$ and ${\Lambda'}_{n,J,A}^{\eps,r}$
%  clearly respect  homotopies and thus induced semi-group homomorphisms
%$$K_0^{\eps,r}(J)\longrightarrow K_1^{\alpha\eps,l_\eps r}(A/J)$$ and
%$$K_1^{\eps,r}(J)\longrightarrow K_0^{\alpha\eps,l_\eps r}(A/J);$$
%\item  The map $\Lambda_{n,J,A}^r$ and ${\Lambda'}_{n,J,A}^r$
%  depend on the choice of the completely positive contraction  $s:A/J\to A$  such that
%$s(A_r/A_r\cap J)\subset A_r$. But since the set of such maps is an
%affine space, the homomorphisms induced on $K_0^{\eps,r}(A)$ and
%$K_1^{\eps,r}(A)$ does not depend on the choice of such a cross-section.
%\item If the exact sequence is splitted, i.e if we can choose for $s$ a
%  homomorphism of $C^*$-algebras, then in the proof of the first item of lemma \ref{lem-bound}, the element $z$ lies in the image of $s$ and thus
%  $y=0$. We get therefore that $\Lambda_{n,J,A}^{\eps,r}$ is trivial. In
%  the same way, in the proof of the second item of lemma
%  \ref{lem-bound},
% the element $W'$ lies in the image of $s$ and thus
%  so do $f,r$ and therefore $h=0$. In consequence,
%  ${\Lambda'}_{n,J,A}^{\eps,r}$ is also trivial.
%\end{enumerate}
%\end{remark}
%In consequence we obtain with notations of lemma \ref{lem-bound}
 For a semi-split extension of filtered $C^*$-algebras
 $$0\longrightarrow J \longrightarrow A \stackrel{q}{\longrightarrow}
A/J\longrightarrow 0,$$ we set  $\DD^0_{J,A}:\K_0(A/J)\to\K_1(J)$, for the restriction
of $\DD_{J,A}$ to $\K_0(A/J)$ and  $\DD^1_{J,A}:\K_1(A/J)\to\K_0(J)$, for   the restriction
of $\DD_{J,A}$ to $\K_1(A/J)$.

\begin{remark}\label{rem-fonct}\
\begin{enumerate}
\item Let $A$ and $B$ be two filtered $C^*$-algebras and let
  $\phi:A\to B$ be a filtered homomorphism. Let $I$ and $J$ be respectively
  ideals in $A$ and $B$ and assume that
\begin{itemize}
\item $\phi(I)\subset J$;
\item there exists   semi-split filtered cross-sections   $s:A/I\to A$  and  $s':B/J\to J$   such that
$s'\circ\tilde{\phi}=\phi\circ s$, where $\tilde{\phi}:A/I\to
  B/J$ is the homomorphism induced by $\phi$,
\end{itemize}
then
 $\DD_{J,B} \circ
\tilde{\phi}_*=
\phi_*\circ\DD_{I,A}$.
\item Let
 $0\longrightarrow J \longrightarrow A \stackrel{q}{\longrightarrow}
A/J\longrightarrow 0$ be a split extension of filtered
$C^*$-algebras, i.e there exists a homomorphism of filtered $C^*$-algebras $s:A/J\to A$ such that $q\circ s=Id_{A/J}$. Then we have
$\DD_{J,A}=0$.
\end{enumerate}
\end{remark}
For  a filtered $C^*$-algebra  $A$, we have defined  the suspension and the
cone  respectively as  $SA=C_0((0,1),A)$ and
$CA=C_0((0,1],A)$. Then $SA$ and $CA$ are filtered $C^*$-algebras and
evaluation at the value $1$ gives rise to a semi-split filtered extension of $C^*$-algebras
\begin{equation}\label{equation-bott-extension} 0\to SA\to CA\to A\to 0\end{equation} and in the even case, the corresponding
boundary $\partial_{SA,CA}:K_0(A)\to K_1(SA)$ implements
the suspension isomorphism and has the following easy description when
$A$ is unital:
 if $p$ is a projection, then $\partial_{SA,CA}[p]$ is the class in
 $K_1(SA)$ of the path of unitaries $$[0,1]\to U_n(A);\, t\mapsto pe^{2\imath\pi t}+1-p.$$
Let us  show that we have an analogous description in term of
 almost projection. Notice that if $q$ is an $\eps$-$r$-projection in
 $A$, then $$z_q:[0,1]\to A;\, t\mapsto qe^{2\imath\pi t}+1-q$$ is
 a $5\eps$-$r$-unitary in $\widetilde{SA}$. Using this, we can define a $(5,1)$-controlled morphism
 $\mathcal{Z_A}=(Z_A^{\eps,r})_{0<\eps<1/20,r>0}:\K_0(A)\to\K_1(SA)$ in the following way:
 \begin{itemize}
 \item  for any $q$ in $\P_n^{\eps,r}(A)$ and any integer $k$ let us set   
$$V_{q,k}:[0,1]\to \U_n^{5\eps,r}(\widetilde{SA}): t\mapsto \diag(e^{-2k\imath\pi
  t},1,\ldots,1)\cdot(1-q+qe^{2\imath\pi t});$$
  \item define then  $Z_A^{\eps,r}([q,k]_{\eps,r})=[V_{q,k}]_{5\eps,r}$.
  \end{itemize}

\begin{proposition} There exists a  control pair $(\lambda,h)$ such that for any filtered $C^*$-algebra $A$,
 then  $\DD^0_{CA,SA}\aeq \mathcal{Z_A}$.
\end{proposition}

\begin{proof}Let $[q,k]_{\eps,r}$ be an element of $K_0^{\eps,r}(A)$, with $q$ in $\P_n^{\eps;r}(A)$ and $k$ integer.
We can assume without loss of generality that $n\gq k$. Namely, up to
replace $n$ by $2n$ and using a homotopy between
$\diag(q,0)$ and $\diag(0,q)$ in $\P_{2n}^{\eps,r}(A)$, we can indeed assume
            that $q$ and  $\diag(I_k,0)$  commute.
As in the proof of proposition  \ref{prop-bound}, define $l_\eps$ as the
smallest integer such that $\sum_{l=l_\eps+1}^{\infty}
10^l/l^!<\eps$. Let us consider the following paths in   $M_n(A)$

$$z:[0,1]\lto M_n(A); t\mapsto
\sum_{l=0}^{l_\eps}(2\imath\pi (tq+(1-t)\diag(I_k,0)))^l/l!$$ and

$$z':[0,1]\lto M_n(A); t\mapsto \exp (2\imath\pi\diag(-t
    I_k,0))(1-q+e^{2\imath\pi t}q).$$
Since $q$ and
$I_k$ commutes, then
\begin{eqnarray*}\exp (2\imath\pi(\diag(-t
    I_k,0)+tq))&=&\exp (2\imath\pi\diag(-t
    I_k,0))\cdot \exp(2\imath \pi
tq)\\
%%&=&e^{-2\imath\pi t}\exp (2\imath\pi tq)
\end{eqnarray*}
and hence 
$$z(t)=\exp (2\imath\pi\diag(-t
    I_k,0))\exp (2\imath\pi tq)-\sum_{l=l_\eps+1}^{\infty}(2\imath\pi (tq+(1-t)\diag(I_k,0)))^l/l!.$$ 
We get therefore
\begin{eqnarray*}
\|z(t)-z'(t)\|&\lq&\eps+\|qe^{2\imath\pi t}+(1-q)-\exp 2\imath\pi t q\| \\
&\lq&\eps+2\|\kappa_0(q)- q\|+ \|\exp 2\imath\pi t\kappa_0(q)- \exp
2\imath\pi t q\|\\
&\lq& \eps(5+4e^{4\pi}).
\end{eqnarray*}
Let us set $$y:[0,1]:\lto M_n(A);\,t\mapsto z(t)-1-(1-t)\diag(I_k,0)\sum_{l=1}^{l_\eps}(2\imath\pi )^l/l!-t\sum_{l=1}^{l_\eps}(2\imath\pi q)^l/l!.$$
For some $\alpha_s\gq\alpha_\partial$, we get then that $1+y$ and
$z'$  are homotopic elements in $U_n^{\alpha_s\eps,k_{\partial,\eps}r}(\widetilde{SA})$. Using the semi-split filtered cross-section
$A\to CA;\,a\mapsto [t\mapsto ta]$ for the extension of equation (\ref{equation-bott-extension}), we get in 
view of the proof of proposition \ref{prop-bound},
$$\iota_1^{\alpha_\partial\eps,\alpha_s\eps,k_{\partial,\eps}r}\circ\partial_{SA,CA}^{\eps,r}([q,k]_{\eps,r})=[1+y]_{\alpha_s\eps,k_{\partial,\eps}r},$$
and thus we deduce  $$\iota_1^{\alpha_\partial\eps,\alpha_s\eps,k_{\partial,\eps}r}\circ\partial_{SA,CA}^{\eps,r}([q,k]_{\eps,r})=[z']_{\alpha_s\eps,k_{\partial,\eps}r}.$$
We get the result by using a homotopy of
unitaries in $M_n(\widetilde{SA})$ between
 $$t\mapsto \diag(e^{-2k\pi
  t},1,\ldots,1)$$ and $t\mapsto \exp (2\imath\pi\diag(-t
    I_k,I_{n-k}))$.
\end{proof}
The inverse of the suspension isomorphism is provided, up to  Morita
equivalence  by the  Toeplitz extension:
let us consider the unilateral shift $S$ on $\ell^2(\N)$, i.e the
operator defined on the canonical basis $(e_n)_{n\in\N}$ of $\ell^2(\N)$ by
$S(e_n)=e_{n+1}$ for all integer $n$. Then the Toeplitz algebra
$\T$ is the $C^*$-subalgebra of $\L(\ell^2(\N))$ generated by $S$.
The algebra of compact operators $\K(\ell^2(\N))$ is an ideal of  $\T$
and we get an extension  of $C^*$-algebras
$$0\to\K(\ell^2(\N))\to \T \stackrel{\rho}{\to} C(\S_1)\to 0,$$ called
the Toeplitz extension, where $\S_1$ denote the
unit circle. Let us define $\T_0=\rho^{-1}(C_0((0,1))$, where $C_0(0,1)
$ is viewed as a subalgebra of $C(\S_1)$. We obtain then an extension of  $C^*$-algebras
$$0\to\K(\ell^2(\N))\to \T _0\stackrel{\rho}{\to} C_0(0,1)\to 0.$$ For
any $C^*$-algebra $A$, we can tensorize this exact sequence to obtain an
extension
$$0\to\K(\ell^2(\N))\otimes A\to \T _0\otimes A{\to}SA\to 0$$ which is filtered and semi-split when $A$ is a filtered $C^*$-algebra.
\begin{proposition}\label{prop-toeplitz} There exists a control pair $(\lambda,h)$ such that
 $$\D^1_{\K(\ell^2(\N))\otimes A, \T_0\otimes A}\circ \ZZ_A\aeq \MM_A$$ for any unital filtered $C^*$-algebra $A$.
\end{proposition}
\begin{proof}
 Let $q$ be a
$\eps$-$r$-projection in $M_n(A)$.  We can assume indeed  without loss of generality that $n=1$.The Toeplitz
extension is semi-split by the section induced by the completely
positive map $s:C(\S_1)\longrightarrow \T;\, f\mapsto M_f$,
where if $\pi_0$ stands for the projection $L^2(\S_1)\cong\ell^2(\Z)\to
l^2(\N)$, then $M_f$ is the composition
$$l^2(\N)\hookrightarrow \ell^2(\Z)\cong L^2(\S_1)\stackrel{
  f\cdot}{\to} L^2(\S_1)\stackrel{\pi_0}{\to}l^2(\N),$$ ($f\cdot$ being the pointwise multiplication by $f$). 
Notice first  that $\left(\begin{smallmatrix}S&1-SS^*\\0&S^*\end{smallmatrix}\right)$   is a unitary lift 
 of  $\S_1\to M_2(\C);\, z\mapsto diag(z,\bar{z})$ in $M_2(\T)$ under the homomorphism induced by $\rho:\T\to C(\S_1)$.
Under the section
induced by $s$, we see that
 $z_q$ lifts to $1\otimes (1-q)+ S\otimes q$,  and hence
$$
W=\begin{pmatrix}S&1-SS^*\\0&S^*\end{pmatrix}\otimes q+I_2\otimes(1-q)$$ is a lift in $\U_2^{5\eps,r}(\T_0\otimes A)$ of
$\diag(z_q,z_q^*)$. Since $\|q(1-q)\|<\eps$, we see that $W^*\diag(1,0)W$ is closed to
$$\begin{pmatrix}S^*&0\\1-SS^*&S\end{pmatrix}\begin{pmatrix}1&0\\0&0\end{pmatrix}
\begin{pmatrix}S&1-SS^*\\0&S^*\end{pmatrix}\otimes q^2+
\begin{pmatrix}1&0\\0&0\end{pmatrix}\otimes (1-q)^2.$$ 
Hence, $W^*\diag(1,0)W$ is an element of
$\P_2^{10\eps,2r}(\T_0\otimes A)$ which is closed 
to
$\diag(1, (1-SS^*)\otimes q)$.
%and then,
%following the end of the proof of lemma \ref{lem-bound}, we get the
%result.
Since $$\mathcal{M}_A([q,0]_{\eps,r})=[\diag(0, (1-SS^*)\otimes
q)]_{\eps,r},$$ we get the existence of a positive real $\alpha_t$ such
that the proposition holds.
\end{proof}
\subsection{Long exact sequence}
We follow the route of \cite[Sections 6.3, 7.1 and 8.2]{we} to state for semi-split extensions of filtered $C^*$-algebras $(\lambda,h)$-exact long
exact sequences in quantitative $K$-theory, for some universal control pair $(\lambda,h)$.
\begin{proposition}\label{prop-half-exact} There exists a control pair $(\lambda,h)$ such that  for any   semi-split
extension  of filtered $C^*$-algebras   $$0  \longrightarrow J
\stackrel{\jmath}{\longrightarrow} A \stackrel{q}{\longrightarrow}
A/J\longrightarrow 0,$$ the composition
$$\K_*(J)\stackrel{j_*}{\to}\K_*(A)\stackrel{q_*}{\to}\K_*(A/J)$$ is $(\lambda,h)$-exact at $\K_*(A)$.
\end{proposition}
\begin{proof}
We can assume without loss of generality that $A$ is unital.
In the even case, let $y$ be an element of $K_0^{\eps,r}(A)$ such that
$q_*(y)=0$ in
$K_0^{\eps,r}(A/J)$, let $e$ be an $\eps$-$r$-projection in $M_n(A)$ and let $l$ be a
positive  integer such that $y=[e,k]_{\eps,r}$. Up to stabilization, we
can assume that $k\lq n$ and that  $q(e)$ is
homotopic to $p_k=\diag(I_k,0)$ as an  $\eps$-$r$-projection in
$M_n(A/J)$. According to corollary  \ref{cor-conjugate}, there exists
 up to stabilization a
$\alpha_h\eps$-$k_{h,\eps}r$-unitary $W$ of $M_n(A/J)$ such that
$$\|Wq(e)W^*-p_k\|\lq\alpha_h\eps.$$ The  $3\alpha_s\eps$-$2k_{h,\eps} r$-unitary  $\diag(W,W^*)$ of $M_{2n}(A/J)$ is
homotopic to $I_{2n}$. Let choose as in lemma  \ref{lem-lift}, a control pair $(\alpha,l)$,
an  integer $j$  and a
$\alpha\eps$-$l_\eps r$-unitary $V$ of $M_{2n+j}(A)$ such that
$$\|q(V)-\diag(W,W^*,I_{k+j})\|\lq\alpha\eps.$$ If we set
$e'=V\diag(e,0)V^*$, then $e'$ is a
$4\alpha\eps$-$2l_\eps r$-projection in $M_{2n+j}(A)$. If $s:A/J\to A$
is a semi-split filtered cross-section such that $s(1)=1$,  define $f=e'-s\circ q(e'-\diag(I_n,0))$. We see that $f$
belongs to $M_{2n+j}({J^+})$ and moreover, since $\|f-e'\|\lq
(4\alpha+\alpha_h)\eps$, then according to lemma \ref{lemma-almost-closed}, $f$ is for a
suitable $\lambda$ a
$\lambda\eps$-$2l_\eps r$-projection of  $M_{2n+k}({J^+})$ homotopic to $e'$. Then $x=[f,k]_{\lambda\eps,2l_\eps r}$ defines
a class in $K_0^{\lambda\eps,2l_\eps r}(J)$. As in the proof of  (ii) of lemma \ref{lem-conjugate-proj}   we can 
choose $\lambda$ big enough  so that  $\diag(e',I_{2n+j})$ and $\diag(e,0,I_{2n+j})$ are homotopic
$\lambda\eps$-$2k_{h,\eps} r$-projections  of
$M_{2n}(A)$ and hence we get the result in the even case.

For the odd case, let $y$ be an element in  $K_1^{\eps,r}(A)$ such
that $q_*(y)=0$ in $K_1^{\eps,r}(A/J)$ and let us choose an
$\eps$-$r$-unitary $V$ in some $M_n(A)$ such that $y=[V]_{\eps,r}$. In
view of lemma \ref{lem-lift} and  up to enlarge
the size of the matrix $V$, we can assume that $\|q(V)-q(W)\|\lq\alpha_e\eps$ with $W$ a
$\alpha_e\eps$-$k_{e,\eps} r$-unitary of $M_n(A)$ homotopic to $I_n$. Hence
$W^*V$ and $V$ are homotopic  $3\alpha_e\eps$-$(k_{e,\eps}+1)r$-unitary of $M_n(A)$.
If we
set
$$U=W^*V+s\circ q(I_n-W^*V),$$ then the coefficients of the matrix
$U-I_n$  lie  in $J$. Moreover, since $$\|U-W^*V\|\lq (2\alpha_e+1)\eps,$$ we
obtain that $U$ is a $\lambda\eps$-$(l_\eps+1)r$-unitary for some
$\lambda\gq 1$. Hence,  $x=[U]_{\lambda\eps, (k_{e,\eps}+1)r}$ defines a
class in $K_1^{\lambda\eps,(k_{e,\eps}+1)r}(J)$ with the required property.
\end{proof}
% \begin{remark}\label{rem-half-exact}
% Since for every $z$ in $K_*^{\eps,r}(A)$ there is an element $z'$ in
% $K_*^{\eps,r}(A)$ such that $\iota_*^{\eps,3\eps,r,2r}(z+z')=0$ in
% $K_*^{3\eps,2r}(A)$, we get from proposition  \ref{prop-half-exact}
% that if $y$ and $y'$ are element in  $K_*^{\eps,
%   r}(A)$ such that $q_*(y)=q_*(y')$ in $K_*^{ \eps,
%   r}(A/J)$, then there exists
% $x$ in  $K_*^{3\lambda\eps,2n_{3\eps}r}(J)$ such that
% $$\iota_*^{\eps,3\alpha\eps,r,2n_{3\eps}r}(y')=\jmath_*^{\eps,3\alpha\eps,r,2n_{3\eps}r}(x)+\iota_*^{\eps,3\alpha\eps,r,2n_{3\eps}r}(y)$$ in
%  $K_*^{\eps,3\alpha\eps,r,2n_{3\eps}r}(A)$.
% \end{remark}
\begin{proposition}\label{prop-exact-boundary} There exists  a control pair
$(\lambda,h)$  such that for any   semi-split extension of filtered $C^*$-algebras
   $$0  \longrightarrow J
\stackrel{\jmath}{\longrightarrow} A \stackrel{q}{\longrightarrow}
A/J\longrightarrow 0,$$  the composition
$$\K_1(A)\stackrel{q_*}{\to}\K_1(A/J)\stackrel{\DD^1_{J,A}}{\to}\K_0(J)$$ is $(\lambda,h)$-exact at $\K_1(A/J)$.
\end{proposition}
\begin{proof}
We can assume without loss of generality that $A$ is unital.
Let $y$ be an element of $K_1^{\eps,r}(A/J)$ such that $\partial_{J,A}^{\eps,r}(y)=0$ in
$K_0^{\alpha_\partial\eps,k_{\partial,\eps} r}(A/J)$ and let $U$ be an $\eps$-$r$-unitary
of $M_n(A/J)$ such that $y=[U]_{\eps,r}$. With notation of lemma
\ref{lem-lift}, let $j$ be an integer and $W$ be a $3\alpha_e\eps$-$2k_{e,3\eps}
r$-unitary in $M_{2n+j}(A)$ such that $$\|q(W)-\diag(U,U^*,I_j)\|\lq
\alpha\eps.$$ Set $x=W\diag(I_n,0)W^*$ and $h=x-\diag(I_n,0)-s\circ q(x-\diag(I_n,0)$ as in the proof of proposition
\ref{prop-bound}.  Since $\partial_{J,A}^{\eps,r}(y)=0$, we can up to take a larger $n$ assume that
$h+\diag(I_n,0)$ is homotopic to $\diag(I_n,0)$ as an $\alpha_\DD\eps$-$k_{\DD,\eps}r$-projection of $M_{2n+j}(\tilde{J})$. Since
$x$ is close to $h+\diag(I_n,0)$, we get from   corollary \ref{cor-conjugate} that up to take a larger $j$,
there exists  for a control pair $(\alpha,l)$, depending only on the control pairs $(\alpha_h,k_h)$ and $(\alpha_\D,k_\D)$ of corollary \ref{cor-conjugate}
and lemma \ref{lemma-bound},
 an $\alpha\eps$-$l_\eps r$-unitary $V'$ in
$M_{2n+j}(\widetilde{J})$ such that
$$|W\diag(I_n,0)W^*-V'\diag(I_n,0)V'^*\|\lq\alpha\eps.$$
Then $V=\rho_J(V')V'^{-1}W^*$ is a
$10(\alpha+\alpha_e)\eps$-$(l_\eps+k_{e,\eps})r$-unitary  in $M_{2n+j}(A)$ such that
 $$\|q(V)-\diag(U,U^*,I_j)\|\lq
\alpha\eps.$$
Since for a suitable constant $\alpha'$ depending only on $\alpha$ we have
$$\|\rho_J(V')\diag(I_n,0)\rho_J({V'^*})-\diag(I_n,0)\|\lq\alpha'\eps,$$
we obtain that
$$\|V\diag(I_n,0)V^*-\diag(I_n,0)\|\lq\alpha''\eps$$ and $$\|V^*\diag(I_n,0)V-\diag(I_n,0)\|\lq\alpha''\eps$$for
some  constant $\alpha''$ depending only on $\alpha'$ that we can choose indeed larger than $(10\alpha+\alpha_e)$. Hence
the $n\times n$-left upper corner $X$ of $V$ is a
$\alpha''\eps$-$(l_\eps+l'_\eps)r$-unitary  in
$M_{n}(A)$ such that $\|q(X)-U\|\lq\alpha''\eps$. Hence
we get the result.
\end{proof}
\begin{proposition}\label{prop-exact-range-bound} There exists  a control pair
$(\lambda,h)$  such that for any   semi-split
extension of   filtered $C^*$-algebras $$0  \longrightarrow J
\stackrel{\jmath}{\longrightarrow} A \stackrel{q}{\longrightarrow}
A/J\longrightarrow 0,$$  the composition
$$\K_1(A/J)\stackrel{\DD^1_{J,A}}{\to}\K_0(J)\stackrel{\jmath_*}{\to}\K_0(A)$$ is $(\lambda,h)$-exact in $\K_0(J)$.
\end{proposition}
\begin{proof}It is enough to prove the result for $A$ unital.  Let $y$ be an element of $K_0^{\eps,r}(J)$ such that
$\jmath^{\eps,r}_*(y)=0$ in
$K_0^{\eps,r}(A)$,  let $e$ be an $\eps$-$r$-projection in $M_n({J^+})$ and $k$ be a
positive  integer such that $y=[e,k]_{\eps,r}$. If we set
$p_k=\diag(I_k,0)$, we can indeed assume without
loss of generality that $\|q(e)-p_k\|\lq 2\eps$ (where
${J^+}$ is viewed as a subalgebra of $A$).  Up to
stabilization, we can also assume that  $e$ is
homotopic to $p_k$ as an  $\eps$-$r$-projection in
$M_n(A)$. According to corollary  \ref{cor-conjugate},
there exists up to stabilization a
$\alpha_h\eps$-$k_{h,\eps} r$-unitary $W$ of $M_n(A)$ such that
$$\|e-Wp_kW^*\|\lq\alpha_h\eps.$$ Up to replace $n$ by
  $2n$, $W$ by
  $\diag(W,W^*)$ and  $e$ by $\diag(e,0)$, we can assume that $W$ is a
  $3\alpha_h\eps$-$2k_{h,\eps} r$-unitary  homotopic to $I_n$. Since
\begin{eqnarray*}
\|q(W)p_k q(W^*)-p_k\|&\lq&
\|q(W)p_kq(W^*)-q(e)\|+\|q(e)-p_k\|\\
&<& (2+\alpha_h)\eps,
\end{eqnarray*}  then
$$\|q(W^*)p_k q(W)-p_k\|< (2+4\alpha_h)\eps.$$ Hence for an $\alpha'>1$ depending only on $\alpha_h$,
the left-up $n\times n$ corner $V_1$ and the right bottom corner $V_2$
 of
$q(W)$ are   $\alpha'\eps$-$2k_{e,\eps} r$-unitaries of
$M_n({A/J})$ such that
$$\|q(W)q(W^*)-\diag(V_1,V_2)\diag(V_1,V_2)^*\|<(\alpha_h+\alpha')\eps$$ and
$$\|q(W^*)q(W)-\diag(V_1,V_2)^*\diag(V_1,V_2)\|<(\alpha_h+\alpha')\eps.$$ Hence $q(W)$ is close to $\diag(V_1,V_2)$ and hence
there is a $\lambda>1$ depending only on $\alpha_e$ such that as a
$\lambda\eps$-$2k_{h,\eps} r$-unitary of $M_n(A/J)$, then
$\diag(V_1,V_2)$ is homotopic to $q(W)$ and hence to $I_n$. We can indeed choose $\lambda$ big enough
such that if we set  $x=[V_1]_{\lambda\eps,2k_{e,\eps} r}$,
then \begin{eqnarray*}
\partial_{J,A}^{\lambda\eps,2k_{e,\eps}
  r}(x)&=&[e,k]_{\lambda\alpha_\partial\eps,k_{\partial,\alpha\eps}2k_{e,\eps r}}\\
&=&\iota_*^{\eps,r,\lambda\eps,2k_{e,\eps} r}(y).
\end{eqnarray*}
\end{proof}
From propositions \ref{prop-half-exact}, \ref{prop-exact-boundary} and \ref{prop-exact-range-bound}
 we can derive the analogue of the long exact
sequence in $K$-theory.
\begin{theorem}\label{th-long-exact-sequence}
 There exists  a control pair
$(\lambda,h)$  such that for any   semi-split
extension of filtered $C^*$-algebras   $$0  \longrightarrow J
\stackrel{\jmath}{\longrightarrow} A \stackrel{q}{\longrightarrow}
A/J\longrightarrow 0,$$  the sequence
$$\K_1(J)\stackrel{\jmath_*}{\longrightarrow}\K_1(A)\stackrel{q_*}{\longrightarrow}
\K_1(A/J)\stackrel{\DD_{J,A}}{\longrightarrow}\K_0(J)\stackrel{\jmath_*}{\longrightarrow}\K_0(A)\stackrel{q_*}{\longrightarrow}\K_0(A/J)$$
is $(\lambda,h)$-exact.
\end{theorem}
As a consequence, using the exact sequence
\begin{equation}
\label{equ-exact-sequence}
0\to SA\to CA\to A\to 0,
\end{equation} and in view of lemma \ref{lem-contractible} and point (iii) of remark \ref{rem-inj-surj-inv},
we deduce  in the setting of the semigroup
$K_*^{\eps,r}(\bullet)$ the analogue of the suspension isomorphism in $K$-theory.
\begin{corollary}\label{cor-suspension} Let $\DD^1_{A}=\DD^1_{SA,CA}:\K_1(A)\to\K_0(SA)$  be the controlled boundary
morphism associated to the semi-split and filtered  extension of equation (\ref{equ-exact-sequence}) for a filtered $C^*$-algebra $A$.
 \begin{itemize}
 \item There exists a control pair $(\lambda,h)$ such that
for any filtered $C^*$-algebra $A$, then $\DD^1_{A}$ is $(\lambda,h)$-invertible.
\item  Moreover,
we can choose a $(\lambda,h)$-inverse which is natural: there exists a control pair
 $(\alpha_\beta,k_\beta)$ and for any filtered $C^*$-algebra $A$  a  $(\lambda,h)$-controlled morphism
$\BB^0_A=(\beta_A^{\eps,r})_{0<\eps<\frac{1}{4\alpha_\beta},r>0}:\K_0(SA)\to\K_1(A)$  which is an
$(\lambda,h)$-inverse for $\DD^1_{A}$ and such that
$\BB^0_B\circ f_S= f\circ  \BB^0_A$  for any homomorphism $f:A\to B$ of filtered $C^*$-algebras,
where $f_S:SA\to SB$ is the suspension of
the homomorphism $f$.\end{itemize}
\end{corollary}

\subsection{The mapping cones}\label{subsection-mapping-cones}
We end this section by proving that the mapping cones construction can be performed in the
 framework  of quantitative $K$-theory.
Let $$0\to J\to A\stackrel{q}{\to} A/J{\to}0$$ be a filtered semi-split extension of $C^*$-algebras. Let us set
$A/J[0,1)= C_0([0,1),A/J)$ and define the mapping cone of $q$:
$$C_q=\{(x,f)\in A\oplus A/J[0,1);\,\text{ such that }f(0)=q(x)\}.$$ Using a  semi-split filtered cross-section
 for $q$, we see that  $C_q$
is filtered by $$\left(C_q\cap(A_r\oplus A/J[0,1))_r\right)_{r>0}.$$
Let us set $$e_q:J\to C_q;\,x\mapsto (x,0)$$ and $$\phi_q:SA/J\to C_q;\,f\mapsto (0,f).$$ We have then a semi-split extension
 of filtered $C^*$-algebras
$$0\to J\stackrel{e_j}{\to}C_q\stackrel{\pi_2}{\to}A/J[0,1)\to 0,$$ where $\pi_2$ is the projection on the second factor of
$A\oplus A/J[0,1)$.
\begin{lemma}\label{lemma-eq-inv}There exists a control pair $(\lambda,h)$ such that  $e_{q,*}$ is $(\lambda,h)$-invertible for any
  semi-split extension of filtered $C^*$-algebras $0\to J\to A\stackrel{q}{\to} A/J\to 0$.
\end{lemma}
\begin{proof}
The even case is a consequence of theorem \ref{th-long-exact-sequence}. We deduce the odd case from the even one using
corollary \ref{cor-suspension}.
\end{proof}
It is a standard fact in $K$-theory  that the boundary of an extension of $C^*$-algebras  $0\to J\to A\stackrel{q}{\to} A/J\to 0$
can be obtain
using the equality $$e_{q,*}\circ \partial_{J,A}=\phi_{q,*}\circ\partial_{A/J},$$ where $\partial_{A/J}=\partial_{SA/J,CA/J}$
stands for the boundary map of the
extension $$0\to SA/J\to CA/J\to A/J\to 0$$ (corresponding to the evaluation at $1$). We have a similar result in quantitative $K$-theory:
\begin{lemma}\label{lemma-mapping-connes}
 With above notations, we have $e_{q,*}\circ \D_{J,A}=\phi_{q,*}\circ\D_{A/J}$, where $\D_{A/J}$ stands for $\D_{SA/J,CA/J}$.
\end{lemma}
\begin{proof}
 We can assume without loss of generality that $A$ is unital. Let us fix a  semi-split filtered cross-section $s:A/J\to A$  such that
 $s(1)=1$.
Let $p$ be an $\eps$-$r$ projection in $A/J$. Using the notations of the proof  of proposition \ref{lemma-bound},  define for $t$ in $[0,1]$
\begin{itemize}
 \item $\displaystyle x_t=\sum_{l=1}^{l_\eps}\frac{(2\imath\pi t s(p))^l-t(2\imath\pi)^ls(p^l)}{l!}$ in $A$;
\item $\displaystyle f_t:[0,1]\to
A/J:\sigma \mapsto \sum_{l=1}^{l_\eps}\frac{\left((2\imath\pi (1-\sigma)t +\sigma)p\right)^l-((1-\sigma)t +\sigma)(2\imath\pi p)^l}{l!}.$\end{itemize}
Then, $(1+(y_t,f_t))_{t\in[0,1]}$ is a path of $\alpha\eps$-$k_\eps r$ unitary in $C_q^+$ with  $x_0=0$ and $f_1=0$.
Moreover,
\begin{itemize}
 \item
$x_1$ belongs to $J$ and satisfies the conclusion of lemma  \ref{lemma-bound}  starting from the $\eps$-$r$-projection $p$ and 
with respect to  the semi-split extension of
filtered $C^*$-algebras
$0\to J\to A\stackrel{q}{\to} A/J\to 0$  and to the semi-split filtered cross-section $s$;
\item $f_0$ belongs to $SA/J$ and satisfies the conclusion of lemma  \ref{lemma-bound}  starting from the $\eps$-$r$-projection $p$
and with respect to  the semi-split extension of
filtered $C^*$-algebras
$0\to SA/J\to CA/J  {\to} A/J\to 0$ corresponding to evaluation at $1$
and to the semi-split filtered cross-section $A/J\mapsto CA/J; a\mapsto [t\mapsto ta]$.
\end{itemize}
Hence, following the construction of proposition \ref{prop-bound} in the even case, we obtain that
$e_{q,*}\circ \D_{J,A}$ and $\phi_{q,*}\circ\D_{A/J}$ coincide on $\K_0(A/J)$.

\smallskip

Let us check now the odd case. Let $u$ be an  $\eps$-$r$-unitary in $M_n(A/J)$. Pick any $\eps$-$r$-unitary in  some
  $M_j(A/J)$ such that $\diag(u,v)$ is homotopic to
  $I_{n+j}$ in
 $\U_{n+j}^{3\eps,2r}(A/J)$. According to lemma  \ref{lem-lift}, and up to replace $v$
 by $\diag(v,I_k)$ for some integer $k$, there exists an element $w$
 in $\U_{n+j}^{3\alpha_e\eps,2k_{e,3\eps}r}(A)$ homotopic to $I_{n+j}$ as a $3\alpha_e\eps$-$2k_{e,3\eps}r$-unitary and such that
 $\|q(w)-\diag(u,v)\|\lq 3\alpha_e\eps$. Let $(w_t)_{t\in[0,1]}$ be a path in $\U_{n+j}^{3\alpha_e\eps,2k_{e,3\eps}r}(A)$ with
$w_0=I_{n+j}$ and $w_1=w$ and set $y_t=q(w_t)\diag(I_n,0)q(w_t^*)$. As in the proof of proposition \ref{prop-bound},
we see that  $y_t$ is an element in
  $\P_{n+j}^{12\alpha_e\eps,4k_{e,3\eps}r}(A/J)$ such that
  $\|y_1-\diag(I_n,0)\|\lq 9\alpha_e\eps$. Define $$g:[0,1]\to M_{n+j}(A/J);\,t\mapsto  y_t-\diag(I_n,0)-t(y_1-\diag(I_n,0)).$$
Then $g+\diag(I_n,0)$ is the element of $\P_{n+j}^{12\alpha_e\eps,4k_{e,3\eps}r}(S^+A/J)$ that we get from $u$ and $v$ when
we perform
the construction of proposition \ref{prop-bound} in the odd case with respect to the extension
$0\to SA/J\to CA/J\to A/J\to 0$. Let us set now
$x_t=w_t\diag(I_n,0)w_t^*$ and  $h_t= x_t-\diag(I_n,0)-ts\circ q(x_1-\diag(I_n,0))$ for $t$ in $[0,1]$.
Then  $\diag(I_n,0)+h_t$ belongs to
 $\P_{n+j}^{12\alpha_e\eps,4k_{e,3\eps}r}(A)$ and $\diag(I_n,0)+h_1$ is the element of
$\P_{n+j}^{12\alpha_e\eps,4k_{e,3\eps}r}(J)$ that we get from $u$ and $v$ when we perform
the construction of proposition \ref{prop-bound} in the odd case with respect to the extension
$0\to J\to A\stackrel{q}{\to}A/J\to 0$. Eventually, if we define $$H_t:[0,1]\to  M_{n+j}(A/J);\, \sigma\mapsto g_{(1-\sigma)t+\sigma},$$
then $\left((h_t,H_t)+\diag(I_n,0)\right)_{t\in[0,1]}$ is a homotopy in $\P_{n+j}^{12\alpha_e\eps,4k_{e,3\eps}r}(C_q^+)$ between
  $\left((0,g)+\diag(I_n,0)\right)$ and  $\left((h_1,0)+\diag(I_n,0)\right)$. Thus we obtain the result in the odd case.

\end{proof}
As a consequence, we get that the controlled suspension morphism is compatible with the controlled boundary maps.
\begin{proposition}\label{prop-boundary-suspension}
There exists a control pair $(\lambda, h)$ such that for any semi-split  extension  of
filtered $C^*$-algebras  $0\to J\to A\to A/J\to 0$,  the following  diagrams are   $(\lambda, h)$-commutative:
$$\begin{CD}
\K_0(A/J) @>\DD_{A/J}>> \K_1(SA/J)  \\
         @V\DD_{J,A} VV
         @VV\DD_{SJ,SA}V\\
\K_1(J) @>\DD_{J}>> \K_0(SJ)
\end{CD}
$$
and
$$\begin{CD}
\K_1(A/J) @>\DD_{A/J}>> \K_0(SA/J)  \\
         @V\DD_{J,A} VV
         @VV\DD_{SJ,SA}V\\
\K_0(J) @>\DD_{J}>> \K_1(SJ)
\end{CD},
$$
where $\DD_{J}$ and $\DD_{A/J}$ stands respectively for the controlled suspension  morphisms $\DD_{SJ,CJ}$
and $\DD_{SA/J,CA/J}$.
\end{proposition}
\begin{proof}
Let $q_s:SA\to SA/J$ the suspension of the homomorphism $q:A\to A/J$. Applying lemma \ref{lemma-mapping-connes} to
the extensions $0\to J\to A\to A/J\to 0$ and  $0\to SJ\to SA\to SA/J\to 0$ and  using the naturality of controlled boundary maps
mentioned in remark \ref{rem-fonct}, we get
%  and under the canonical identification between $SC_q$ and $C_{q_S}$
\begin{eqnarray*}
 e_{q_s,*}\circ \D_{SJ,SA}\circ\D_{A/J}&=&\phi_{q_s,*}\circ\D_{SA/J}\circ\D_{A/J}\\
&=&\D_{SC_q}\circ\phi_{q,*}\circ\D_{A/J}\\
&=&\D_{SC_q}\circ e_{q,*}\circ\D_{J,A}\\
&=&e_{q_s,*}\circ\D_{J}\circ \D_{J,A}
\end{eqnarray*}
The proposition is then a consequence of lemma \ref{lemma-eq-inv}.

\end{proof}

\section{Controlled Bott periodicity}
The aim of this section is to prove that there exists a control pair $(\lambda,h)$ such that  given  a filtered $C^*$-algebra $A$,
then Bott periodicity $K_0(A)\stackrel{\cong}{\to} K_0(S^2A)$  is induced in $K$-theory by a
$(\lambda,h)$-isomorphism $\K_0(A)\to\K_0(S^2A)$.
As an application, we use  the controlled boundary  morphism of proposition \ref{prop-bound}       to close
 the  controlled exact  sequence  of \ref{th-long-exact-sequence} into a six-term $(\lambda,h)$-exact sequence for
 some universal control pair $(\lambda,h)$.
This will be achieved  by using  the full power of $KK$-theory.
\subsection{Tensorization in $KK$-theory}\label{subsection-tensorisation}
Let $A$ be  a $C^*$-algebra and let $B$ be a   $C^*$-algebra filtered by $(B_r)_{r>0}$. Within all this section, we will assume for sake
of simplicity that $B_r$ is closed for every positive number $r$ (which is the case for Roe algebras and crossed product algebras). Let us define $A\ts B_r$ as the closure in the spatial tensor
product $A\ts B$ of the algebraic tensor product of $A$ and $B_r$. Then the $C^*$-algebra $A\ts B$ is filtered by $(A\ts B_r)_{r>0}$.
Moreover, if $J$ is a semi-split ideal of $A$, i.e $0\to J\to A\to A/J\to 0$ is a semi-split extension of $C^*$algebras, then
$$0\to J\ts B\to A\ts B\to A/J\ts B\to 0$$ is a   semi-split extension of filtered $C^*$-algebras.
Recall from \cite{kas} that for $C^*$-algebras  $A_1$, $A_2$ and $D$, G. Kasparov defined a tensorization map
$$\tau_D:KK_*(A_1,A_2)\to KK_*(A_1\ts D,A_2\ts D)$$ in the following way: let $z$ be an element in $KK_*(A_1,A_2)$ represented by
a $K$-cycle $(\pi,T,\E)$, where
\begin{itemize}
\item $\E$ is a right $A_2$-Hilbert module;
 \item  $\pi$ is a
representation of $A_1$ into  the algebra $\L(\E)$ of adjointable operators of
$\E$;
\item $T$ is a self-adjoint operator on $\E$ satisfying  the $K$-cycle conditions, i.e. $[T,\pi(a)]$,
  $\pi(a)(T^2-\Id_{\E})$  are compact operators on $\E$ for any $a$ in $A_1$.
\end{itemize}
Then $\tau_D(z)\in KK_*(A_1\ts D,A_2\ts D)$ is represented by the $K$-cycle $(\pi\ts Id_D,T\ts Id_D,\E\ts D)$.

In what follows, we show that if $A_1$ and $A_2$  are $C^*$-algebras, if $B$ is a filtered $C^*$-algebra and
  if $z$ is an element in $KK_*(A_1,A_2)$, then the homomorphism $K_*(A_1\ts B)\to K_*(A_2\ts B)$ provided by
 left multiplication by $\tau_B(z)$
is induced by a controlled morphism. Moreover, we have some compatibility results  with respect to Kasparov product.
As an outcome, we obtain a controlled version of the Bott periodicity  that induces in $K$-theory the Bott periodicity.
\begin{proposition}\label {prop-tensor}
Let $A_1$ and $A_2$ be $C^*$-algebras, let  $B$ be a filtered $C^*$-algebra and let
$z$ be an element in $KK_1(A_1,A_2)$. Then there exists an $(\alpha_\D,k_\D)$-controlled morphism
 $$\T_B(z)=(\tau^{\eps,r}_B(z))_{0<\eps<\frac{1}{4\alpha_\DD},r>0}:\K_*(A_1\ts B)\to\K_*(A_2\ts B)$$ of degree $1$
inducing in $K$-theory the right multiplication by $\tau_B(z)$.
\end{proposition}
\begin{proof}
Recall  that $z$ can be indeed  represented by a odd  $A_1$-$A_2$-$K$-cycle  $(\pi,T,\H\ts A_2)$,
where  $\H$ is a separable Hilbert space, $\pi$ is a
representation of $A_1$ in the algebra $\L(\H\ts A_2)$ of adjointable operators of
$\H\ts A_2$ and  $T$ is a self-adjoint operator in   $\L(\H\ts A_2)$ satisfying  the $K$-cycle conditions.
Let us set $P_B=\frac{\Id_{\H\otimes A_2\ts  B}+T\ts Id_B}{2}$, $\pi_B=\pi\ts Id_B$ and define the $C^*$-algebra
$$E^{(\pi,T)}=\{(x,y)\in A_1\ts B\bigoplus \mathcal{L}(\H\ts A_2\ts B)\text{ such
  that } P_B \cdot \pi_B(x) \cdot P_B-y \in \mathcal{K}(\H)\otimes A_2\ts
B\}.$$ Since $P_B$ has no propagation, the
$C^*$-algebra $E^{(\pi,T)}$ is filtered by $(E^{(\pi,T)}_r)_{r>0}$ with
 $$E^{(\pi,T)}_r=\{(x,\,P_B\cdot \pi_B (x)\cdot P_B+y);\, x\in A_1\ts B_r\text{ and } y\in \mathcal{K}(\H)\otimes A_2\ts
B_r\}.$$
The extension of filtered $C^*$-algebras
\begin{equation}\label{equ-sequence-Kcycle}0\longrightarrow \mathcal{K}(\H)\otimes A_2\ts
B\longrightarrow E^{(\pi,T)} \longrightarrow
A_1\otimes B\longrightarrow 0\end{equation} is  semi-split
by the cross-section $$s:A_1\ts B  \to E^{(\pi,T)};\, x \mapsto (x,P_B \cdot\pi_B(x)
\cdot P_B).$$

Let us show that the  associated controlled boundary (degree one) map
   $$\DD_{ \mathcal{K}(\H)\otimes A_2\ts
B,E^{(\pi,T)}}: \K_*(A_1\ts B)\to\K_*(\K(\H)\ts A_2\ts B)$$ only depends on the class  $z$ of
$(\pi,T,\H\otimes A_2)$ in $KK_1(A_1,A_2)$.
Assume that $(\pi,T,\H\otimes A_2[0,1])$ is a  $A_1$-$A_2[0,1]$-$K$
-cycle providing a homotopy between two
$A_1$-$A_2$-$K$-cycles $(\pi_0,T_0,\H\otimes A_2)$ and $(\pi_1,T_1,\H\otimes
A_2)$.
 For $t\in[0,1]$ we denote by
\begin{itemize}
 \item $e_t:A_2[0,1]\ \to A_2$ the evaluation at $t$;
 \item $F_t\in \mathcal{L}(\H\otimes A_2)$ the fiber at $t$ of
   an  operator $F\in\mathcal{L}(\H\otimes A_2[0,1])$;
\item $\pi_{t}:A_1\to \mathcal{L}(\H\otimes A_2)$ the representation  induced by
  $\pi$ at the fiber $t$;
\item $s_t:A_2\ts B \to E^{(\pi_t,T_t)};\, x\mapsto (x,P_{t,B}\cdot
  \pi_{t,B} (x) \cdot P_{t,B})\quad$ (with $P=\frac{T+1}{2}$);
\end{itemize}

Then the homomorphism
 $E^{(\pi,T)}\to E^{(\pi_t,T_t)};\, (x,y)\mapsto (x,y_t)$
satisfies the conditions of remark \ref{rem-fonct} (with  $s:A_2\ts B\to E^{(\pi,T)};\, x \mapsto (x,P_B\cdot\pi_B (x)
\cdot P_B)$ and  $s_t:A_2\ts B\to E^{(\pi,T_t)}$)  and thus
we get that
$$(\Id_{\mathcal{K}(\H)}\otimes e_t\ts Id_B
)_*\circ\DD_{\mathcal{K}(\H)\otimes
  A_1\ts B[0,1],E^{(\pi,T)}}=\DD_{\mathcal{K}(\H)\otimes A_1\ts B,E^{(\pi_t,T_t)}},$$ and according to lemma \ref{lem-contractible}, we deduce
 that $$\DD_{\mathcal{K}(\H)\otimes
A_1\ts B_2,E^{(\pi_0,T_0)}}=\DD_{\mathcal{K}(\H)\otimes
A_1\ts B,E^{(\pi_1,T_1)}}.$$ This shows  that for a $A_1$-$A_2$-$K$-cycle  $(\pi,T,\H\otimes A_2)$,
then $\DD_{\mathcal{K}(\H)\otimes
A_1\ts B,E^{(\pi,T)}}$ depends only on the class $z$ of
$(\pi,T,\H\otimes A_2)$ in $KK_1(A_1,A_2)$. Finally we define
$$\TT_B(z)=(\tau_B^{\eps,r}(z))_{0<\eps<\frac{1}{4\alpha_\DD}}\stackrel{\text{def}}{=\!=}\MM_{A_2\ts B}^{-1}\circ \DD_{\mathcal{K}(\H)\otimes
A_1\ts B,E^{(\pi,T)}},$$ where
\begin{itemize}
\item $(\pi,T,\H\ts A_2)$ is any
$A_1$-$A_2$-$K$-cycles representing $z$;
\item $\MM_{A_2\ts B}$ is  the Morita equivalence (see example \ref{example-controlled-morphism}).
\end{itemize}
  The result then follows from the observation that  up to the Morita equivalence
$$K_*(\mathcal{K}(\H)\otimes A_2\ts
B)\stackrel{\cong}{\to} K_*(
A_2\ts B),$$ the boundary $\partial_{\mathcal{K}(\H)\otimes
A_1\ts B,E^{(\pi,T)}}$ corresponding to the exact sequence  (\ref{equ-sequence-Kcycle}) is induced by right multiplication by $\tau_B(z)$.
\end{proof}

\begin{remark}\label{rem-tensor}  Let $B$ be a filtered $C^*$-algebra.
\begin{enumerate}
\item  For any $C^*$-algebras $A_1$ and $A_2$ and any  elements $z$ and $z'$ in $KK_1(A_1,A_2)$ then
$$\TT_B(z+z')=\TT_B(z)+\TT_B(z').$$
\item Let  $0\to J\to A\to A/J\to 0$ be  a semi-split extension of filtered
  $C^*$-algebras and let $[\partial_{J,A}]$ be the element of
  $KK_1(A/J,J)$ that implements the boundary map
  $\partial_{J,A}$. Then we have
$$\TT_B([\partial_{J,A}])=\DD_{J\ts B,A\ts B}.$$
\item For any  $C^*$-algebras $A_1$, $A_2$ and $D$ and any $K$-cycle $(\pi,T,\H\ts A_2)$ for $KK_1(A_1,A_2)$, we have a natural identification
between $E^{(\pi_D,T_D)}$ and $E^{(\pi,T)}\ts D$. Hence,
for any element $z$ in $KK_1(A_1,A_2)$ then $\TT_B(\tau_D(z))=\TT_{B\ts D}(z)$.
\end{enumerate}
\end{remark} For a   a filtered $C^*$-algebra  $B$ and a homomorphism  $f:A_1\to A_2$ of $C^*$-algebras,
we set
$f_{B}:A_1\ts B\to A_2\ts B$  for  the filtered homomorphism  induced by $f$.
\begin{proposition}\label{prop-bifunctoriality}
   Let $B$ be a filtered $C^*$-algebra and let $A_1$ and $A_2$ be two $C^*$-algebras.
\begin{enumerate}
 \item  For any $C^*$-algebra $A_1'$, any homomorphism of  $C^*$-algebras $f:A_1\to A'_1$ and any $z$ in $KK_1 (A'_1,A_2)$,
we have
$\T_B(f^*(z))=\T_B(z)\circ f_{B,*}$;
\item For any $C^*$-algebra
$A'_2$, any  homomorphism of $C^*$-algebras $g:A_2\to A'_2$ and any $z$ in $KK_1 (A_1,A_2)$, we have
 $\TT_B(g_*(z))=g_{B,*}\circ \TT_B(z)$.
\end{enumerate}
\end{proposition}
\begin{proof}\
 \begin{enumerate}

\item Let $
A'_1$ be a filtered $C^*$-algebra, let $f:A_1\to A'_1$ be a homomorphism of  $C^*$-algebras and let
  $(\pi,T,H\otimes A_2)$ be an odd  $A'_1$-$A_2$-$K$-cycle. With the
  notations of the proof of proposition \ref{prop-tensor}, the homomorphism
$$f^E:E^{f^*(\pi,T)}\to E^{(\pi,T)};\,(x,y)\mapsto (f_{B}(x),y )$$  fits in the  commutative
diagram
$$
\begin{CD}
0 @>>>   \K(\H)\otimes A_2\ts B  @>>>E^{f^*(\pi,T)}  @>>> A_1\ts B @>>>  0\\
 @.         @V=VV     @Vf^EVV         @VVf_B V\\
 0 @>>>  \K(\H)\otimes A_2\ts B  @>>>E^{(\pi,T)} @>>>  A'_1\ts B  @>>>  0
\end{CD}.
$$

Moreover $f_B$ and $f^E$ intertwines the semi-split and filtered
cross-sections $$A_1\ts B \to E^{f^*(\pi,T)};\, x \mapsto (x,P_B
\cdot\pi_B\circ f_B(x)\cdot P_B)$$ and
$$A'_1\ts B\to E^{(\pi,T)};\, x \mapsto (x,P_B
\cdot\pi_B(x)\cdot P_B)$$
 and thus, we get by remark \ref{rem-fonct} that
$$\T_B(f^*(z))=\T_B(z)\circ f_*$$  for all $z$ in $KK_1 (A'_1,A_2)$.

\item Let $
A'_2$ be a $C^*$-algebra and  let $g:A_2\to A'_2$ be a homomorphism of $C^*$-algebras. For any element
  $F$ in $\mathcal{L}(\H\otimes A_2)$,  let us denote by
 $$\tilde{F}=F\ts_{A_2} Id_{A'_2}\in \mathcal{L}(\H\otimes A_2\ts_{A_2} A'_2).$$Notice that $\H\otimes A_2\ts_{A_2} A'_2$ can be viewed
as a right $A'_2$-Hilbert-submodule of $\H\ts A'_2$ and under this identification, for any
 $F$ in $\K(\H)\otimes A_2$, then  $\tilde{F}$ is the restriction
to  $\H\otimes A_2\ts_{A_2} A'_2$
of the homomorphism
$(Id_{\K(\H)}\ts g)(F)$.
Let $z$ be an element of $KK_1(A_1,A_2)$
represented by a $K$-cycle
  $(\pi,T,\H\ts A_2)$. Consider the $A_1$-$A_2$-$K$-cycle $(\pi',T',\H'\ts A_2)$ with
$\H'=\H_1\oplus\H_2\oplus\H_3$, where $\H_1$, $\H_2$ and $\H_3$ are three copies of $\H$,
$\pi'=0\oplus 0 \oplus \pi$ and
$T'=Id_{\H_1\ts A_2}\oplus Id_{\H_2\ts A_2} \oplus T$. Then $(\pi',T',\H'\ts A_2)$  is again a  $K$-cycle representing  $z$ and $g_*(z)$
is represented by the $K$-cycle $(\pi'',T'',\E)$, where
\begin{itemize}
\item $\E=\H_1 \otimes A'_2\bigoplus \H_2\otimes A'_2\bigoplus \H_3\otimes A_2\ts_{A_2} A'_2$;
\item $\pi''=0\oplus 0\oplus \tilde{\pi}$;
\item $T''=\Id_{\H_1\ts A'_2}\oplus \Id_{\H_2\ts A'_2}\oplus \tilde{T}$.
\end{itemize}
 Using Kasparov stabilization theorem,
we get that $\H_2\otimes A'_2\bigoplus \H_3\otimes A_2\ts_{A_2} A'_2$
 is isomorphic as a right-$A'_2$-Hilbert module to $\H\otimes A'_2$
and hence, using this identification, we can represent $g_*(z)$ using a standard  right-$A'_2$-Hillbert module, as in the proof of
proposition \ref{prop-tensor}.
Then, under the  above identification   $\H_2\otimes A'_2\bigoplus \H_3\otimes A_2\ts_{A_2} A'_2\cong \H\otimes A'_2$,
\begin{eqnarray*}
 g_E:E^{(\pi,T)}&\to &E^{g_*(\pi,T)}\\
(x,y)&\mapsto& (x,P''_B\pi''(x)P''_B+(Id_{\K(\H')\ts B}\ts g)( y -P'_B\pi'(x)P'_B))\end{eqnarray*}
restricts to a
homomorphism $\K(\H_1\oplus\H_2\oplus \H_3)\ts A_2\ts B\to \K(\H_1\oplus \H)\ts A'_2\ts B$.

We get now a commutative diagram
$$
\begin{CD}
0 @>>>   \K(\H_1\oplus\H_2\oplus \H_3)\otimes A_2\ts B  @>>>E^{(\pi',T')}  @>>> A_1\ts B @>>>  0\\
 @.         @Vg_E VV     @Vg_EVV         @VV= V\\
 0 @>>>  \K(\H_1\oplus\H)\otimes A'_2\ts B @>>>E^{(\pi'',T'')} @>>>  A_1\ts B@>>>  0
\end{CD}.
$$
Hence, we get by remark \ref{rem-fonct} that
$$\DD_{\K(\H)\otimes A'_2\ts B,E^{(\pi'',T'')}}=g_{E,*}\circ \D_{\K(\H)\otimes A_2\ts B,E^{(\pi',T')}}.$$ But the restriction of
$g_E$ to the corner  $\K(\H_1)\ts A_2\ts B$  of   the $C^*$-algebra $\K(\H_1\oplus\H_2\oplus \H_3)\otimes A_2\ts B$ is $Id_{\K(\H_1)}\ts g\ts Id_{B}$.
Since  the Morita equivalence
 $$\MM_{A'_2\ts B}:\K_*(A'_2\ts B)\stackrel{\cong}{\to}\K_*(\K(\H_1\oplus\H)\ts A'_2\ts B)$$ can be implemented  by an inclusion
of  $A'_2\ts B$ in a corner of $\K(\H_1)\ts A'_2\ts B$, and similarly for the  Morita equivalence
 $$\MM_{A_2\ts B}:\K_*(A_2\ts B)\stackrel{\cong}{\to}\K_*(\K(\H_1\oplus\H_2\oplus \H_3)\ts A_2\ts B),$$
we deduce that the two following compositions coincide:
$$\K_* (A_2\ts B))\stackrel{g_{B,*}}{\longrightarrow}\K_*(A'_2\ts B)
\stackrel{\MM_{A'_2\ts B}}{\longrightarrow}\K_*(\K(\H_1\oplus\H)\ts(A'_2\ts B))
$$ and 
\begin{equation*}\begin{split} 
\K_*(A_2\ts B)\stackrel{\MM_{A_2\ts B}}{\longrightarrow}\K_*(\K(\H_1\oplus\H_2&\oplus\H_3)\ts A_2\ts B)\\
& \stackrel{g_{E,*}}{\longrightarrow}
\K_* (\K(\H_1\oplus\H)\ts A'_2\ts B).\end{split}\end{equation*} Hence we get 
$$\TT_B(g_*(z))=g_*\circ \TT_B(z)$$
for any $z$ in $KK_1(A_1,A_2)$.
\end{enumerate}
 \end{proof}

Let us now extend the definition of   $\TT_B$  to the even case.
Consider for a suitable control pair $(\alpha_\BB,k_\BB)$ and any filtered $C^*$-algebra $A$
the $(\alpha_\BB,k_\BB)$-controlled morphism of odd degree $\BB_A: \K_*(SA)\to \K_*(A)$ defined
\begin{itemize}
 \item  by $\BB_A^0$ on $\K_0(SA)$ as  in corollary \ref{cor-suspension};
\item  by $\MM_A^{-1}\circ \DD_{\K(\ell^2(\N))\otimes A, \T _0\otimes A}$ on
 $\K_1(SA)$  using the Toeplitz extension
$$0\to\K(\ell^2(\N))\otimes A\to \T _0\otimes A{\to}SA\to 0$$
 (see the discussion at the end of section \ref{subsection-controlled-boundary-maps}).
\end{itemize}
Then, according to corollary \ref{cor-suspension} and proposition \ref{prop-toeplitz}, there exists a control pair
$(\lambda,h)$ such that   $\BB_A$ is
a right $(\lambda,h)$-inverse for $\DD_{SA,CA}$ for  any filtered $C^*$-algebra $A$. Let us set $\alpha_\TT=\lambda\alpha_\BB$
and $k_\TT=h*k_\BB$.
\smallskip

Now, let  $B$ be  a filtered $C^*$-algebra, let $A_1$ and $A_2$ be $C^*$-algebras,
then define for any $z$ in $KK_0(A_1,A_2)$ the $(\alpha_\TT,k_\TT)$-controlled morphism
$$\TT_B(z)=(\tau_B^{\eps,r})_{0<\eps<\frac{1}{4\alpha_\TT},r>0}:\K_*(A_1\ts B)\to \K_*(A_2\ts B)$$ by
$$\TT_B(z)\defi \BB_{A_2\ts B}\circ
\TT_{B}(z\otimes_{A_2}[\partial_{A_2}])$$
where
\begin{itemize}
 \item

$[\partial_{A_2}]=[\partial_{SA_2,CA_2}]\in KK_1(A_2,SA_2)$ corresponds to the boundary of the exact sequence
$0\to SA_2\to CA_2\to A\to 0$;
\item $\otimes_{A_2}$ stands for Kasparov product.
 \end{itemize}
Up to compose on the left with $\iota_*^{\alpha_\DD\eps,\alpha_\TT\eps,k_\DD r,k_\TT r}$, we can
 in the odd case  define $\TT_B(\bullet)$ also as an
$(\alpha_\TT,k_\TT)$-controlled morphism.
\begin{theorem}\label{thm-tensor}
Let $B$ be a filtered $C^*$-algebra, let $A_1$ and $A_2$ be $C^*$-algebras
\begin{enumerate}
\item For any element $z$ in $KK_*(A_1,A_2)$, then $\TT_B(z):\K_*(A_1\ts B)\to \K_*(A_2\ts B)$ is
 a $(\alpha_\TT,k_\TT)$-controlled morphism
with same degree as $z$ that induces  in $K$-theory  right  multiplication by $\tau_B(z)$.
\item  For any  elements $z$ and $z'$ in $KK_*(A_1,A_2)$ then
$$\TT_B(z+z')=\TT_B(z)+\TT_B(z').$$
\item Let $A'_1$ be a filtered $C^*$-algebras and  let $f:A_1\to A'_1$ be a homomorphism of  $C^*$-algebras, then
$\T_B(f^*(z))=\T_B(z)\circ f_{B,*}$  for all $z$ in $KK_* (A'_1,A_2)$.
\item Let $
A'_2$ be a $C^*$-algebra and  let $g:A'_2\to A_2$ be a homomorphism of $C^*$-algebras then
$\TT_B(g_{*}(z))=g_{B,*}\circ \TT_B(z)$
for any $z$ in $KK_*(A_1,A'_2)$.
\item $\TT_B([Id_{A_1}])\stackrel{(\alpha_\TT,k_\TT)}{\sim} \Id _{\K_*({A_1}\ts B)}$.
\item For any $C^*$-algebra $D$ and any element $z$ in $KK_*(A_1,A_2)$, we have  $\TT_B(\tau_D(z))=\TT_{B\ts D}(z)$.
\end{enumerate}
\end{theorem}
\begin{proof}Since $\BB_{A_2\ts B}$ is a right $(\lambda,h)$-inverse for
$\DD_{SA_2\ts B,CA_2\ts B}$, it induces in $K$-theory a right inverse (indeed an inverse)
for the (degree $1$) boundary map
$$\partial_{SA_2\ts B,CA_2\ts B}:K_*(A_2\ts B)\to K_*(SA_2\ts B).$$
 But since $\TT_B(z\ts_{A_2}[\partial_{SA_2\ts B,CA_2\ts B}])$ induces in $K$-theory
 right multiplication by $z\ts_{A_2}[\partial_{SA_2\ts B,CA_2\ts B}]$, we eventually get that
$\TT_B(z\ts_{A_2}[\partial_{SA_2\ts B,CA_2\ts B}])$ induced in $K$-theory the composition
$$K_*(A_1\ts B)\stackrel{\ts_{A_1\ts B}\tau_B(z)}{\lto}K_*(A_2\ts B)\stackrel{\partial_{SA_2\ts B,CA_2\ts B}}{\lto}
K_*(SA_2\ts B)$$ and hence we get the first point.

Point { (ii)} is a consequence of  remark
\ref{rem-tensor}. Point { (iii)} is a  consequence  of  proposition \ref{prop-bifunctoriality}. Point {(iv)} is a consequence of
proposition \ref{prop-bifunctoriality} and of the  naturality of
$\BB_\bullet$ (see remark \ref{rem-fonct} and corollary \ref{cor-suspension}),
point { (v)}  holds by definition of $\BB_\bullet$. Point (vi) is a consequence of point  (iii) of remark \ref{rem-tensor}.
\end{proof}
We end this section by proving the compatibility of $\TT_B$
with Kasparov product.

\begin{theorem}\label{thm-product-tensor} There exists a control pair $(\lambda,h)$ such that the following holds :

  let  $A_1,\,A_2$ and $A_3$  be $C^*$-algebras and let $B$ be a  filtered $C^*$-algebra. Then for any
$z$ in $KK_*(A_1,A_2)$ and  any  $z'$ in
$KK_*(A_2,A_3)$, we have
$$\TT_B(z\ts_{A_2} z')\aeq  \TT_B(z') \circ\TT_B(z).$$
\end{theorem}
\begin{proof}
We first deal with the case $z$ even. According to
\cite[Lemma 1.6.9]{laff-inv}, there exists a $C^*$-algebra $A_4$
and homomorphisms $\theta:A_4 \to A_1$ and $\eta:A_4\to A_2$ such that
\begin{itemize}
\item the element $[\theta]$ of $KK_*(A_4,A_1)$ induced by $\theta$ is
   invertible.
\item $z=\eta_*([\theta]^{-1})$.
\end{itemize} Since $\theta_*([\theta]^{-1})=[Id_{A_1}]$ in $KK_*(A_1,A_1)$, we get in view of remark \ref{rem-composition}
and of points  (iii), (iv) and (v) of theorem \ref{thm-tensor} that
$$\TT_B(z\ts_{A_2} z')\aeq \TT_B(\theta^*(z\ts_{A_2} z'))\circ \TT_B([\theta]^{-1}),$$
with $(\lambda,h)=(\alpha_\TT^2,k_\TT*k_\TT)$. But by bi-functoriality of $KK$-theory, we have
$\theta^*(z\ts_{A_2} z')=\eta^*(z')$  and then the result is a consequence of points (iii) and  (iv)  of theorem \ref{thm-tensor}.
We can proceed similarly when $z'$ is even. Let us prove now the result when $z$ and $z'$ are odd.
 Then $[\partial_{A_2}]=[\partial_{SA_2,CA_2}]$ is an invertible element in $KK_1(A_2,SA_2)$ and
$z\ts_{A_2} z'=z\ts_{A_2}[\partial_{A_2}]\ts_{SA_2}[\partial_{A_2}]^{-1}\ts_{A_2 }z'$ and hence using the even case, we get that
\begin{equation}\label{equation-prod1}
 \TT_B(z\ts_{A_2} z')\aeq  \TT_B([\partial_{A_2}]^{-1}\ts_{A_2}z') \circ\TT_B(z\ts_{A_2}[\partial_{A_2}]).\end{equation}
But
\begin{eqnarray}
 \nonumber  \TT_B([\partial_{A_2}]^{-1}\ts_{A_2}z') &=& \BB_{A_3\ts B}
\circ \TT_B([\partial_{A_2}]^{-1}\ts_{A_2}z'\ts_{A_3}[\partial_{A_3}])\\
&\aeqp& \BB_{A_3\ts B}
\circ \TT_B(z'\ts_{A_3}[\partial_{A_3}])\circ\TT_B([\partial_{A_2}]^{-1})\label{equation-prod2}
\end{eqnarray}
for
some control pair $(\lambda',h')$, depending only on $(\lambda,h)$ and $(\alpha_\TT,k_\TT)$,
 where  equation  (\ref{equation-prod2}) holds by the even
case applied to $z'\ts_{A_3}[\partial_{A_3}]$ and  $[\partial_{A_2}]^{-1}$. Hence, for a control pair $(\lambda'',h'')$-depending only
on $(\lambda,h)$, we get applying  the even case to $[\partial_{A_2}]^{-1}$
and $z\ts_{A_2}[\partial_{A_2}]$ that
\begin{equation}\label{equation-prod3}
  \TT_B(z\ts_{A_2} z')\stackrel{(\lambda'',h'')}{\sim}  \BB_{A_3\ts B}
\circ \TT_B(z'\ts_{A_3}[\partial_{A_3}]) \circ\TT_B(z).
\end{equation}In view of this equation, we deduce the odd case from  the controlled Bott periodicity,
which will be proved in the next lemma: if we set $[\partial]=[\partial_{C_0(0,1),C_0(0,1]}]\in KK_1(\C,C_0(0,1))$, then
there exists a controlled $(\alpha,k)$ such that $\TT_A([\partial]^{-1})$ is an  $(\alpha,k)$-inverse for
$\DD_{A}$ for any filtered $C^*$-algebra $A$.
 Indeed, from this claim and since for some control pair $(\alpha',k')$, the $(\alpha_\BB,k_\BB)$-controlled morphism $\BB_A$
is for every filtered $C^*$-algebra $A$ a right  $(\alpha',k')$-inverse for $\TT_A([\partial])$, we get that
$$\TT_A([\partial]^{-1})\stackrel{(\alpha'',k'')}{\sim}\BB_A$$ for some controlled pair $(\alpha'',k'')$
depending only on  $(\alpha',k')$ and $(\alpha_\TT,k_\TT)$.
Noticing by using point (vi) of theorem \ref{thm-tensor}, that
 $\TT_{A_3\ts B}([\partial]^{-1})=\TT_{B}([\partial_{A_3}]^{-1}) $, the proof of the theorem in the odd case is
then  by equation   (\ref{equation-prod3})  a consequence of the even case applied to
$[\partial_{A_3}]^{-1}$ and $z'\ts_{A_3}[\partial_{A_3}]$
\end{proof}

\subsection{The controlled  Bott isomorphism}
We prove in this subsection a controlled version of Bott periodicity. The proof use the even case of
theorem \ref{thm-product-tensor}
  and is needed  for the proof of the odd case. Let $A=(A_r)_{r>0}$ be a filtered $C^*$-algebra and let us assume that $A_r$ is
closed for every positive number $r$.
 Let us denote for short as before  $\DD_{SA,CA}$ by $\DD_{A}$ and $[\partial_{SA,CA}]$ by $[\partial_{A}]$ for any filtered
$C^*$-algebra $A$  and let us set
$[\partial]=[\partial_\C]$.
\begin{theorem}\label{theo-bott}
There exists a control pair  $(\alpha,k)$ such that for every filtered $C^*$-algebra $A$, then
$\TT_A([\partial]^{-1})$ is an  $(\alpha,k)$-inverse for  $\DD_{A}$.
\end{theorem}
\begin{proof}
Consider  the even element  $z=[\partial]\ts_{S}[\partial_{S}]$ of  $KK_*(\C,S^2)$, where  $S=C_0(0,1)$
and $S^2=SS$. The lemma is a consequence of the following
claim: there exists a control pair  $(\lambda,h)$  such that $\DD_{SA}\circ \DD_A\aeq \TT_A(z)$ for any $C^*$-algebra $A$.
Before proving the claim, let us  see how it implies the lemma. Notice first that
by   point (ii) of remark \ref{rem-tensor}, we have  $\DD_A=\TT_A([\partial])$. Since by associativity of Kasparov product 
$[\partial]^{-1}\ts_\C z=[\partial_S]$, we get from
  theorem \ref{thm-product-tensor} applied to the even case, that there exists a control pair $(\lambda',h')$ such that for any filtered
$C^*$-algebra $A$, then $\TT_A(z)\circ \TT_A([\partial]^{-1})\circ \DD_A\aeq \DD_{SA}\circ \DD_A$. Using the claim and since
$z$ is an invertible element of $KK_*(\C,S^2)$, we obtain  from theorem  \ref{thm-product-tensor} applied to the even  case that there exists
 a control
pair $(\alpha,k)$ such that $\TT_A([\partial]^{-1})$ is a left $(\alpha,k)$-inverse for $\DD_A$. Using associativity of the Kasparov
product, we see that $[\partial]=z\ts_{S^2}[\partial_S]^{-1}$. Then applying twice theorem  \ref{thm-product-tensor}, on one hand
to $[\partial]^{-1}$  and $z\ts_{S^2}[\partial_S]^{-1}$ and on the other hand to $[\partial]^{-1}\ts z$ and $[\partial_S]^{-1}$, we get that
there exists a control
pair $(\alpha',k')$ such that
$\TT_A([\partial])\circ \TT_A([\partial]^{-1})\stackrel{(\alpha',k')}{\sim}\TT_{SA}([\partial]^{-1})\circ \TT_{SA}([\partial])$.
But according to what we have seen before,
$\TT_{SA}([\partial]^{-1})\circ \TT_{SA}([\partial])\stackrel{(\alpha,k)}{\sim}\Id_{\K_*(SA)}$.

\smallskip

Let us now prove the claim.  It is known that up to Morita equivalence, $[\partial_A]^{-1}$ is the element of $KK_1(SA,A)$ corresponding
to the boundary element of the Toeplitz extension
$$0\to\K(\ell^2(\N))\otimes A\to \T _0\otimes A{\to}SA\to 0.$$
Let us respectively denote by $\DD_A^0:\K_0(A)\to\K_1(SA)$   and $\DD_A^1:\K_1(A)\to\K_0(SA)$ the restriction of $\DD_A$ to
$\K_0(A)$ and $\K_1(A)$. According to proposition \ref{prop-toeplitz},  there exists a control pair
$(\lambda',h')$ such that, on even elements
 \begin{equation}\label{equ-even}\TT_A([\partial]^{-1})\circ \DD^0_A\aeqp \Id_{\K_0(A)}.\end{equation}
        Since $[\partial_S]=[\partial]^{-1}\ts z$,   we get by left composition  by  $\TT_A(z)$  in equation (\ref{equ-even}) and  by using
 theorem  \ref{thm-product-tensor} in the even case   that there exists a control pair $(\lambda,h)$ depending only on
$(\lambda',h')$ and such that
that $\DD^1_{SA}\circ \DD^0_A\aeq \TT^0_A(z)$ (here
$\TT^0_A(z):\K_0(A)\to\K_0(S^2A)$  stands for the  restriction of $\TT_A(z)$ to $\K_0(A)$).
For the odd case, we know from corollary \ref{cor-suspension}    that  there exists a control pair
$(\lambda'',h'')$ such that
$\DD^1_{S^2A}:\K_1(S^2A)\to\K_0(S^3A)$ is $(\lambda'',h'')$-invertible.
Using the previous case, and since by associativity
of the Kasparov product, we have
$[\partial_A]\ts_{SA}\tau_{SA}(z)=\tau_A(z)\ts [\partial_{S^2A}]$, we get by applying twice  theorem \ref{thm-product-tensor} in the even case that there
exists a control pair
$(\lambda''',h''')$ such that
$\DD^1_{S^2A}\circ\DD^0_{SA}\circ\DD^1_{A}\stackrel{(\lambda''',h''')}{\sim}\DD^1_{S^2A}\circ\TT^1_A(z)$, where
$\TT^1_A(z):\K_1(A)\to \K_1(S^2A)$ is the restriction of $\TT_A(z)$ to $\K_1(A)$. Since
$\DD^1_{S^2A}:\K_1(S^2A)\to\K_0(S^3A)$ is $(\lambda'',h'')$-invertible, we get the result by remark
\ref{rem-composition}.
\end{proof}

\subsection{The six term $(\lambda,h)$-exact sequence}

Recall from proposition \ref{prop-boundary-suspension} that
there exists a control pair $(\lambda, h)$ such that for any semi-split  extension of
filtered $C^*$-algebras  $0\to J\to A\to A/J\to 0$,  the following  diagrams are   $(\lambda, h)$-commutative:
$$\begin{CD}
\K_0(A/J) @>\DD_{A/J}>> \K_1(SA/J)  \\
         @V\DD_{J,A} VV
         @VV\DD_{SJ,SA}V\\
\K_1(J) @>\DD_{J}>> \K_0(SJ)
\end{CD}
$$
and
$$\begin{CD}
\K_1(A/J) @>\DD_{A/J}>> \K_0(SA/J)  \\
         @V\DD_{J,A} VV
         @VV\DD_{SJ,SA}V\\
\K_0(J) @>\DD_{J}>> \K_1(SJ)
\end{CD}
$$
As a consequence, by using  theorem \ref{theo-bott} and      proposition  \ref{th-long-exact-sequence}, we get
\begin{theorem}\label{thm-six term}  There exists a control  pair $(\lambda,h)$  such that for any  semi-split
extension of filtered    $C^*$-algebras $$0  \longrightarrow J
\stackrel{\jmath}{\longrightarrow} A \stackrel{q}{\longrightarrow}
A/J\longrightarrow 0,$$ with $A_r$ closed for every positive number $r$, then the following six-term sequence is $(\lambda,h)$-exact
$$\begin{CD}
\K_0(J) @>\jmath_*>> \K_0(A)  @>q_*>>\K_0(A/J)\\
    @A\DD_{J,A} AA @.     @V\DD_{J,A} VV\\
\K_1(A/J) @<q_*<< \K_1(A)@<\jmath_*<< \K_1(J)
\end{CD}
$$
\end{theorem}
\begin{remark}
Let us consider with notations of section \ref{subsection-mapping-cones} the semi-split extension of filtered $C^*$-algebras
\begin{equation}\label{equ-puppe-extension} 0\to SA/J\stackrel{\phi_q}{\to}C_q \stackrel{\pi_1} {\to}A\to 0,\end{equation}where 
$\pi_1:C_q\to A$ is the projection on the first factor of $C_q$. Since  we have a 
semi-split extension of filtered algebras $0\to J\stackrel{e_j}{\to}C_q\stackrel{\pi_2}{\to}A/J[0,1)\to 0,$ and since $A/J[0,1)$ 
is a contractible filtered $C^*$-algebra, we see  in view of theorem  \ref{thm-six term} that $e_{j,*}:\K_*(J)\to \K_*(C_q)$
 is a controlled isomorphism. It is then plain to check that up to the controlled isomorphism $e_{j,*}$ and 
$\DD_{A/J}:\K_*(SA/J)\to\K_*(A/J)$, we get from the semi-split extension of filtered $C^*$-algebras
of equation (\ref{equ-puppe-extension}) (for a possibly different control pair) the   controlled six-term exact sequence of theorem
 \ref{thm-six term}.
\end{remark}

If we apply theorem  \ref{thm-six term} to a filtered and split extension, we get:
\begin{corollary}\label{cor-split-ext}
 There exists a control pair $(\lambda,h)$ such that for every split extension of filtered $C^*$-algebra
$0\to J\to A\to A/J\to 0$, with $A_r$ closed for every positive number $r$ and any filtered split cross-section $s:A/J\to A$, then
$$\K_*(J)\oplus\K_*(A/J)\longrightarrow \K_*(A);\, (x,y)\mapsto \jmath_*(x)+s_*(y)$$ is $(\lambda,h)$-invertible.
\end{corollary}

\section{Quantitative $K$-theory for crossed product $C^\ast$-algebras}\label{section-cross-product}

In this section, we study quantitative $K$-theory for crossed product $C^\ast$-algebras and discuss its applications to
 $K$-amenability.

Let $\Ga$    be a  finitely generated group. A
$\Gamma$-$C^*$-algebra is
 a separable $C^*$-algebra equipped with an action of $\Ga$ by
automorphisms. Recall that  the convolution algebra $C_c(\Ga,A)$ of
finitely supported  $A$-valued functions on $\Ga$ admits two canonical $C^*$-completions,
the reduced crossed product  $A\rtr  \Gamma$ and the maximal  crossed product
$A\rtm  \Gamma$. Moreover, there is a canonical epimorphism
 $\lambda_{\Ga,A}:A\rtm  \Gamma\to A\rtr  \Gamma$ which is the identity on $C_c(\Ga,A)$.

\subsection{Lengths and propagation}

Recall that a length on $\Gamma$ is a map $\ell:\Gamma\to\R^+$ such
that
\begin{itemize}
\item $\ell(\gamma)=0$ if and only if $\gamma$ is the identity element  $e$  of $\Gamma$;
\item  $\ell(\gamma\gamma')\lq\ell(\gamma)+\ell(\gamma')$ for all
  element $\gamma$ and $\gamma'$ of $\Gamma$.
\item $\ell(\gamma)=\ell(\gamma^{-1})$.
\end{itemize}
In what follows, we will assume that $\ell$ is a word length arising from a finite generating symmetric  set $S$, i.e
$\ell(\gamma)=\inf\{d\text{ such that }\ga=\ga_1\cdots\ga_d\text{ with } \ga_1,\ldots,\ga_d\text{ in }S\}$.
Let us denote by $B(e,r)$ the
ball centered at the neutral  element of $\Ga$ with  radius $r$, i.e
$B(e,r)=\{\gamma\in\Gamma\text{ such that }\ell(\gamma)\lq r\}$. For any positive number $r$, we set
$$(A\rtr \Gamma)_r\defi\{f\in C_c(\Ga,A)\text{ with support in
}B(e,r)\}.$$ Then the $C^*$-algebra $A\rtr \Gamma$ is filtered by
$((A\rtr \Gamma)_r)_{r>0}$. In the same way, setting $(A\rtm \Gamma)_r\defi\{f\in C_c(\Ga,A)\text{ with support in
}B(e,r)\}$,
 then the $C^*$-algebra $A\rtm  \Gamma$ is filtered by
$((A\rtm \Gamma)_r)_{r>0}$ (notice that as sets, $ (A\rtr \Gamma)_r=(A\rtm \Gamma)_r$). It is straightforward to check that
two word  lengths give  rise for $A\rtr \Gamma$ (resp. for $A\rtm\Gamma$) to quantitative $K$-theories related by a $(1,c)$-controlled isomorphism
 for a
constant $c$.

For a homomorphism  $f:A\to B$ of $\Ga$-$C^*$-algebras, we denote respectively by $f_{\Ga,red}:A\rtr\Ga\to B\rtr\Ga$ and
$f_{\Ga,max}:A\rtm\Ga\to B\rtm\Ga$  the homomorphisms respectively
induced by $f$ on the reduced and on the maximal crossed product.

For any semi-split extension of $\Ga$-$C^*$-algebras
$0  \longrightarrow J
\stackrel{\jmath}{\longrightarrow} A \stackrel{q}{\longrightarrow}
A/J\longrightarrow 0$, we have semi-split extensions of  filtered $C^*$-algebras

$$0  \longrightarrow J\rtr\Ga
\stackrel{\jmath_{\Ga,red}}{\longrightarrow} A\rtr\Ga \stackrel{q_{\Ga,red}}{\longrightarrow}
A/J\rtr\Ga\longrightarrow 0$$

and

$$0  \longrightarrow J\rtm\Ga
\stackrel{\jmath_{\Ga,max}}{\longrightarrow} A\rtm\Ga \stackrel{q_{\Ga,max}}{\longrightarrow}
A/J\rtm\Ga\longrightarrow 0$$and hence, by theorem \ref{thm-six term}, we get:

\begin{proposition}
  There exists a control  pair $(\lambda,h)$  such that for any  semi-split
extension of $\Ga$-$C^*$-algebras    $$0  \longrightarrow J
\stackrel{\jmath}{\longrightarrow} A \stackrel{q}{\longrightarrow}
A/J\longrightarrow 0,$$ the following six-term sequences  are  $(\lambda,h)$-exact
$$\begin{CD}
\K_0(J\rtr\Ga) @>\jmath_{\Ga,red,*}>> \K_0(A\rtr\Ga)  @>q_{\Ga,red,*}>>\K_0(A/J\rtr\Ga)\\
    @A\DD_{J\rtr\Ga,A\rtr\Ga} AA @.     @V\DD_{J\rtr\Ga,A\rtr\Ga} VV\\
\K_1(A/J\rtr\Ga) @<q_{\Ga,red,*}<< \K_1(A\rtr\Ga)@<\jmath_{\Ga,red,*}<< \K_1(J\rtr\Ga)
\end{CD}
$$
and
$$\begin{CD}
\K_0(J\rtm\Ga) @>\jmath_{\Ga,max,*}>> \K_0(A\rtm\Ga)  @>q_{\Ga,max,*}>>\K_0(A/J\rtm\Ga)\\
    @A\DD_{J\rtr\Ga,A\rtm\Ga} AA @.     @V\DD_{J\rtm\Ga,A\rtm\Ga} VV\\
\K_1(A/J\rtm\Ga) @<q_{\Ga,max,*}<< \K_1(A\rtm\Ga)@<\jmath_{\Ga,max,*}<< \K_1(J\rtm\Ga)
\end{CD}
$$
\end{proposition}

\subsection{Kasparov transformation}
In this subsection we see how a slight modification of the argument used in section \ref{subsection-tensorisation} allowed to define a controlled version of the Kasparov transformation
compatible with Kasparov product.

Notice first that every
 element $z$ of $KK^\Ga_*(A,B)$ can be represented by a
$K$-cycle, $(\pi,T,\H\otimes B)$, where
\begin{itemize}
\item $\H$ is a separable Hilbert
space;
\item the right Hilbert $B$-module $\H\otimes B$ is acted upon by
  $\Gamma$;
\item  $\pi$ is an equivariant
representation of $A$ in the algebra $\L(\H\ts B)$ of adjointable operators on
$\H\otimes B$;
\item $T$ is a self-adjoint operator on $\H\otimes B$ satisfying  the $K$-cycle conditions, i.e. $[T,\pi(a)]$,
  $\pi(a)(T^2-\Id_{\H\otimes B})$ and  $\pi(a)(\gamma(T)-T)$ belongs to
  $\K(\H)\otimes B$, for every $a$ in $A$ and $\gamma\in\Gamma$.
\end{itemize}
Let $T_\Ga=T\otimes_{B}\Id_{B\rtr\Gamma}$ be the adjointable element  of $(\H\otimes B)\otimes_B
B\rtr\Gamma\cong \H\otimes B\rtr\Gamma$ induced by $T$ and let  $\pi_\Gamma$ be the
representation of $A\rtr\Gamma$ in the algebra
$\mathcal{L}(\H\otimes B\rtr\Gamma)$ of
adjointable operators of $\H\otimes B\rtr\Gamma$ induced by $\pi$. Then $(\pi_\Ga,T_\Ga,\H\otimes B\rtr\Ga)$ is a
$A\rtr\Ga$-$B\rtr\Ga$-$K$-cycle and the Kasparov transform \cite{kas} of $z$   is the class $J_\Ga^{red}(z)$ of this $K$-cycle in
$KK_*(A\rtr\Ga,B\rtr\Ga)$.
In the odd case, let us set $P=\frac{\Id_{\H\otimes B}+T}{2}$. Then $P$ induces an adjointable
operator
$P_\Gamma=P\otimes_{B}\Id_{B\rtr\Gamma}$  of $(\H\otimes B)\otimes_B
B\rtr\Gamma\cong \H\otimes B\rtr\Gamma$. Let us define
$$E^{(\pi,T)}=\{(x,y)\in A\rtr\Gamma\oplus \mathcal{L}(\H\otimes B\rtr\Gamma)\text{ such
  that } P_\Gamma\cdot \pi_\Gamma(x) \cdot P_\Gamma-y \in \mathcal{K}(\H)\otimes
B\rtr\Gamma\}.$$ Since $P_\Gamma$ has no propagation, the
$C^*$-algebra $E^{(\pi,T)}$ is filtered by $(E^{(\pi,T)}_r)_{r>0}$ with $$E^{(\pi,T)}_r=\{(x,\,P_\Gamma\cdot \pi_\Gamma(x) \cdot
P_\Gamma+y);\, x\in (A\rtr\Gamma)_r\text{ and } y\in \mathcal{K}(\H)\otimes
(B\rtr\Gamma)_r\}.$$
The extension of $C^*$-algebras
$$0\longrightarrow \mathcal{K}(\H)\otimes
B\rtr\Gamma\longrightarrow E^{(\pi,T)} \longrightarrow
A\rtr\Gamma\longrightarrow 0$$ is filtered semi-split
by the cross-section $$s:A\rtr\Gamma\to E^{(\pi,T)};\, x \mapsto (x,P_\Gamma \cdot\pi_\Gamma(x)
\cdot P_\Gamma).$$
Let us show that  $\DD_{\mathcal{K}(\H)\otimes
B\rtr\Gamma,E^{(\pi,T)}}$  only depends on the class of
$(\pi,T,\H\otimes B)$ in $KK_1^\Gamma(A,B)$.
Assume that $(\pi,T,\H\otimes B[0,1])$ is a $\Gamma$-equivariant $A$-$B[0,1]$-$K$-cycle
 providing a homotopy between two $\Gamma$-equivariant
$A$-$B$-$K$-cycles $(\pi_0,T_0,\H\otimes B)$ and $(\pi_1,T_1,\H\otimes
B)$.
 For $t\in[0,1]$ we denote by
\begin{itemize}
 \item $e_t:B[0,1]\rtr\Gamma\to B\rtr\Gamma$ the evaluation at $t$;
 \item $F_t\in \mathcal{L}(\H\otimes B\rtr\Gamma)$ the fiber at $t$ of
   an  operator $F\in\mathcal{L}(\H\otimes B[0,1]\rtr\Gamma)$;
\item $\pi_{\Gamma,t}$ the representation of $\AG$ induced by
  $\pi_\Gamma$ at the fiber $t$;
\item $s_t:\AG\to E^{(\pi_t,T_t)};\, x\mapsto (x,P_{\Gamma,t}\cdot
  \pi_{\Gamma,t} \cdot P_{\Gamma,t})\quad$ (with $P=\frac{T+1}{2}$);
\end{itemize}

Then the homomorphism
 $E^{(\pi,T)}\to E^{(\pi_t,T_t)};\, (x,y)\mapsto (x,y_t)$
satisfies the conditions of remark \ref{rem-fonct} (with  $s:A\rtr\Gamma\to E^{(\pi,T)};\, x \mapsto (x,P_\Gamma \cdot\pi_\Gamma(x)
\cdot P_\Gamma)$ and $s_t:\AG\to E^{(\pi_t,T_t)}$)  and thus
we get that
$$(\Id_{\mathcal{K}(\H)}\otimes e_t
)_*\circ\DD_{\mathcal{K}(\H)\otimes
  B[0,1]\rtr\Gamma,E^{(\pi,T)}}=\DD_{\mathcal{K}(\H)\otimes B\rtr\Gamma,E^{(\pi_t,T_t)}},$$ and according
to lemma \ref{lem-contractible}, we deduce
 that $$\DD_{\mathcal{K}(\H)\otimes
B\rtr\Gamma,E^{(\pi_0,T_0)}}=\DD_{\mathcal{K}(\H)\otimes
B\rtr\Gamma,E^{(\pi_1,T_1)}}.$$ This shows  that for a $\Gamma$-equivariant $A$-$B$-$K$-cycles $(\pi,T,\H\otimes B)$,
then $\DD_{\mathcal{K}(\H)\otimes
B\rtr\Gamma,E^{(\pi,T)}}$ depends only on the class $z$ of
$(\pi,T,\H\otimes B)$ in $KK_1^\Gamma(A,B)$. Eventually, if we  define
$$\JR(z)=\MM_{B\rtr\Gamma}^{-1}\circ \DD_{\mathcal{K}(\H)\otimes
B\rtr\Gamma,E^{(\pi,T)}},$$ where
\begin{itemize}
\item $(\pi,T,\H\otimes B)$ is any
$\Gamma$-equivariant $A$-$B$-$K$-cycles representing $z$;
\item $\MM_{B\rtr\Gamma}$ is  the Morita equivalence (see example \ref{example-controlled-morphism}).
\end{itemize}
 we get as in section \ref{subsection-tensorisation}
\begin{proposition}\label{prop-kas}
Let $A$ and $B$ be $\Gamma$-$C^*$-algebras.
 Then for  any element $z$ of
$KK^\Ga_1(A,B)$, there is a odd degree $(\alpha_\DD,k_\DD)$-controlled morphism
$$\JR(z)=(J_{\Ga}^{red,\eps,r}(z))_{0<\eps<\frac{1}{4\alpha_\DD},r>0}:\K_*(\AG)\to\K_*(B\rtr\Gamma)$$ such that
\begin{enumerate}
 \item $\JR(x)$ induces in $K$-theory the right multiplication by $J_{\Ga}^{red}(z)$;
\item  $\JR$ is additive, i.e
  $$\JR(z+z')=\JR(z)+\JR(z').$$
\item Let $
A'$ be a $\Gamma$-$C^*$-algebra and  let $f:A\to A'$ be a homomorphism $\Gamma$-$C^*$-algebras, then
$$\JR(f^*(z))=
\JR(z)\circ f_{\Ga,red,*}$$ for any $z$ in $KK_1^\Gamma(A',B)$.
\item Let $B'$ be a $\Gamma$-$C^*$-algebra and  let $g:B\to B'$ be a homomorphism of $\Gamma$-$C^*$-algebras, then
$$\JR(g_*(z))=g_{\Ga,red,*}\circ
\JR(z)$$ for any $z$ in $KK_1^\Gamma(A,B)$.
\item If $$0\to J\to A\to A/J\to 0$$ is a semi-split exact sequence of
  $\Ga$-$C^*$-algebras, let $[\partial_{J,A}]$ be the element of
  $KK^\Ga_1(A/J,J)$ that implements the boundary map
  $\partial_{J,A}$. Then we have
$$\JR([\partial_{J,A}])=\D_{J\rtr\Ga,A\rtr\Ga}.$$
\end{enumerate}
\end{proposition}

We can now define  $\JR$ for even element in the following way.
Set $\alpha_\JJ=\alpha_\TT\alpha_\DD$ and $k_\JJ=k_\TT*k_\DD$. If
 $A$ and $B$ are
$\Gamma$-$C^*$-algebra and if  $z$ is an element in $KK_0^\Gamma(A,B)$, then we set with notation
of section \ref{subsection-tensorisation}
$$\JR(z)=(J_{\Ga}^{red,\eps,r}(z))_{0<\eps<\frac{1}{4\alpha_\TT},r}\defi \TT_{B\rtr\Ga}([\partial]^{-1})\circ
\JR(z\otimes_{B}[\partial_{SB}]).$$
According to theorem \ref{theo-bott}, there exists a control pair $(\lambda,h)$ such that for any $\Ga$-$C^*$-algebra $A$, then
 $\JR([Id_A])\aeq \Id_{\K_*(A\rtr\Ga)}$.
Up to compose with $\iota_*^{\alpha_\DD\eps,\alpha_\JJ\eps,k_{\DD,\eps}r,k_{\JJ,\eps}r}$, we can assume indeed that
 $\JR(\bullet)$ is also, in the  odd case a  $(\alpha_\JJ,k_\JJ)$-controlled morphism. As for theorem \ref{thm-tensor}, we get.

\begin{theorem}\label{thm-kas}
Let $A$ and $B$ be $\Gamma$-$C^*$-algebras.
\begin{enumerate}
\item For any element  $z$  of
$KK^\Ga_*(A,B)$, then
$$\JR(z): \K_*(\AG)\to
\K_*(B\rtr\Gamma)$$
is a $(\alpha_\JJ,k_\JJ)$-controlled morphism of same degree as $z$ that induces in $K$-theory right multiplication by
$J_\Ga^{red}(z)$.
\item  For any $z$ and
  $z'$ in $KK_*^\Gamma(A,B)$, then
  $$\JR(z+z')=\JR(z)+\JR(z').$$

\item For any $\Gamma$-$C^*$-algebra $A'$, any homomorphism  $f:A\to A'$  of $\Gamma$-$C^*$-algebras and  any $z$ in $KK_*^\Gamma(A',B)$, then
$\JR(f^*(z))=
\JR(z)\circ f_{\Ga,*}$.
\item For any $\Gamma$-$C^*$-algebra $B'$, any homomorphism $g:B\to B'$ of
  $\Gamma$-$C^*$-algebras and  any
  $z$ in $KK_*^\Gamma(A,B)$, then $\JR(g_*(z))=g_{\Ga,*}\circ
\JR(z)$.
\end{enumerate}
\end{theorem}

Using the same argument as in the proof of theorem \ref{thm-product-tensor}, we see that $\JR$ is compatible with Kasparov
products.

\begin{theorem}\label{thm-product} There exists a control pair $(\lambda,h)$ such that the following holds:
for every  $\Ga$-$C^*$-algebras $A,\,B$ and $D$, any elements $z$ in
$KK_*^\Gamma(A,B)$ and   $z'$ in
$KK_*^\Gamma(B,D)$, then
$$\JR(z\otimes_B z')\aeq \JR(z')\circ \JR(z).$$
\end{theorem}

We can perform a similar construction for maximal cross products.
\begin{theorem}\label{thm-kasp-max}
Let $A$ and $B$ be $\Gamma$-$C^*$-algebras.
\begin{enumerate}
\item For any element  $z$  of
$KK^\Ga_*(A,B)$, there exists a $(\alpha_\JJ,k_\JJ)$-controlled morphism
$$\JM(z)=(J_{\Ga}^{max,\eps,r}(z))_{0<\eps<\frac{1}{4\alpha_\JJ},r}: \K_*(A\rtm\Ga)\to
\K_*(B\rtm\Gamma)$$
with  same degree as $z$ that induces in $K$-theory right multiplication by
$J^{max}_\Ga(z)$ and such that $\lambda_{\Ga,B,*}\circ \JM(z)=\JR(z)\circ\lambda_{\Ga,A,*}$.
\item  For any $z$ and
  $z'$ in $KK_*^\Gamma(A,B)$, then
  $$\JM(z+z')=\JM(z)+\JM(z').$$
\item For any $\Gamma$-$C^*$-algebra $A'$, any homomorphism  $f:A\to A'$  of $\Gamma$-$C^*$-algebras and  any $z$ in $KK_*^\Gamma(A',B)$, then
$\JM(f^*(z))=
\JM(z)\circ f_{\Ga,max,*}$.
\item For any $\Gamma$-$C^*$-algebra $B'$, any homomorphism $g:B\to B'$ of
  $\Gamma$-$C^*$-algebras and  any
  $z$ in $KK_*^\Gamma(A,B)$, then $\JM(g_*(z))=g_{\Ga,max,*}\circ
\JM(z)$.

\end{enumerate}
Moreover, there exists a controlled pair $(\lambda,h)$ such that,
\begin{itemize}
 \item for any $\Ga$ algebra $A$, then $\JM([Id_A])\aeq \Id_{\K_*(A\rtm\Ga)}$;
\item For any semi-split extension of $\Ga$ algebras $0\to J\to A\to A/J\to 0$, then $\JM([\partial_{J,A}])\aeq \DD_{J,A}$.
\end{itemize}
\end{theorem}

\begin{theorem}\label{thm-product-max} There exists a control pair $(\lambda,h)$ such that the following holds:
for every  $\Ga$-$C^*$-algebras $A,\,B$ and $D$, any elements $z$ in
$KK_*^\Gamma(A,B)$ and   $z'$ in
$KK_*^\Gamma(B,D)$, then
$$\JM(z\otimes_B z')\aeq \JM(z')\circ \JM(z).$$
\end{theorem}
\subsection{Application to $K$-amenability}
The original  definition  of $K$-amenability is due to J. Cuntz \cite{cuntz}. For our purpose, it is more convenient
to use the equivalent definition given by P. Julg and A. Valette in \cite{julgvalette}.
If $\Ga$ is a discrete group, let us denote by $1_\Ga$ the class in  $KK_0^\Ga(\C,\C)$ of the $K$-cycle $(Id_\C,0,\C)$, where
$\C$ is provided with the trivial action on $\Ga$.
\begin{definition}
 Let $\Ga$ be a discrete group. Then $\Ga$ is $K$-amenable if $1_\Ga$ can be represented by a $K$-cycle such that
the action of $\Ga$ on the underlying Hilbert space is weakly contained in the regular representation.
\end{definition}
(The previous  definition indeed also makes sense for locally compact groups.)
\begin{example}
 Amenable groups are obviously $K$-amenable. Typical example  on non-amenable $K$-amenable groups are free groups
 \cite{cuntz}.
More generally, J. L. Tu proved in \cite{tumoy} that group which  satisfies the strong Baum-Connes conjecture (i.e with $\ga=1$) are
$K$-amenable. Examples of such group are groups with the Haagerup property \cite{hk} and fundamental groups of compact and
oriented $3$-manifolds \cite{mop}.
\end{example}

For a $\Ga$-$C^*$-algebra $B$ and an element $T$ of $\L(\H\ts B)$, where $\H$ is a separable Hilbert space, let us set
$T_{\Ga,max}=T\ts_B Id_{B\rtm\Ga}$ and $T_{\Ga,red}=T\ts_B Id_{B\rtr\Ga}$.
If $A$ is a $\Ga$-$C^*$-algebra and $\pi:A\to  \L(\H\ts B)$ is a $\Ga$-equivariant representation,
let $\pi_{\Ga,red}:A\rtr\Ga\to  \L(\H\ts B\rtr\Ga)$ and
$\pi_{\Ga,max}:A\rtm\Ga\to  \L(\H\ts B\rtm\Ga)$ be respectively the reduced and the maximal  representation induced by $\pi$.
Then, we have the following (compare with the proof of \cite[proposition 3.4]{julgvalette}).
\begin{proposition}\label{prop-Kamenable}
 Let $\Ga$ be a $K$-amenable discrete group and let $A$ and $B$ be $\Ga$-$C^*$-algebras. 
Then any elements of $KK_*^\Ga(A,B)$
 can be represented by a $K$-cycle $(\pi,T,\H\ts B)$ such that the homomorphism
$\pi_{\Ga,max}:A\rtm\Ga \to \L(\H\ts B\rtm\Ga)$ factorises
through the homomorphism $\lambda_{\Ga,A}: A\rtm\Ga\to A\rtr\Ga$, i.e there exists a homomorphism
 $$\pi_{\Ga,red,max}:A\rtr\Ga\to  \L(\H\ts B\rtm\Ga)$$ such that
 $$\pi_{\Ga,max}=\pi_{\Ga,red,max}\circ\lambda_{\Ga,A}.$$
As a consequence, for any $\Ga$-$C^*$-algebra $A$, then $$\lambda_{\Ga,A,*}:K_*(A\rtm\Ga)\to K_*(A\rtr\Ga)$$
 is an isomorphism \cite{cuntz}.
\end{proposition}We have the following analogous result for quantitative $K$-theory.
\begin{theorem}
 There exists a control pair $(\lambda,h)$ such that 
$$\lambda_{\Ga,A,*}:\K_*(A\rtm\Ga)\to \K_*(A\rtr\Ga)$$ is a $(\lambda,h)$-isomorphism 
for every $\Ga$-$C^*$-algebra $A$.
\end{theorem}
\begin{proof}
 Let  $(\pi,T,\H\ts SA)$ be a $\Ga$-equivariant $K$-cycle as in proposition
 \ref{prop-Kamenable} representing the element  $[\partial_A]$ of $KK_1^\Ga(A,SA)$  corresponding to the extension
$$0\to SA\to CA\to A\to0.$$ Let then choose  $\pi_{\Ga,A,red,max}:A\rtr\Ga\to  \L(\H\ts B\rtm\Ga)$ such that
 $\pi_{\Ga,max}=\pi_{\Ga,red,max}\circ \lambda_{\Ga,A}$. Let us set $P=\frac{T+Id_{\H\ts SA}}{2}$ and then define
\begin{equation*}\begin{split} E^{(\pi,T)}_{red}=\{(x,y)\in A\rtr\Gamma&\oplus \mathcal{L}(\H\otimes SA\rtr\Gamma)\text{ such that }\\
  &P_{\Gamma,red}\cdot \pi_{\Gamma,red}(x) \cdot P_{\Gamma,red}-y \in \mathcal{K}(\H)\otimes
SA\rtr\Gamma\}, \end{split}\end{equation*}
 \begin{equation*}\begin{split}
E^{(\pi,T)}_{max}=\{(x,y)\in A\rtm\Gamma&\oplus \mathcal{L}(\H\otimes SA\rtm\Gamma)\text{ such that }\\&P_{\Ga,max}\cdot \pi_{\Gamma,max}(x) \cdot P_{\Gamma,max}-y \in \mathcal{K}(\H)\otimes
SA\rtm\Gamma\}\end{split}\end{equation*} and
\begin{equation*}\begin{split}  E^{(\pi,T)}_{red,max}=\{(x,y)\in &A\rtr\Gamma\oplus \mathcal{L}(\H\otimes SA\rtm\Gamma)
\text{ such that }\\&P_{\Gamma,max}\cdot \pi_{\Gamma,red,max}(x) \cdot P_{\Gamma,max}-y \in \mathcal{K}(\H)\otimes
A\rtm\Gamma\}\end{split}\end{equation*}Then $E^{(\pi,T)}_{red}$, $E^{(\pi,T)}_{max}$ and $E^{(\pi,T)}_{red,max}$ are respectively filtered by
$$\{(x,\,P_{\Gamma,red}\cdot \pi_{\Gamma,red}(x) \cdot
P_{\Gamma,red}+y);\, x\in A\rtr\Gamma_r\text{ and } y\in \mathcal{K}(\H)\otimes
SA\rtr\Gamma_r\},$$
$$\{(x,\,P_{\Gamma,max}\cdot \pi_{\Gamma,max}(x) \cdot
P_{\Gamma,max}+y);\, x\in SA\rtm\Gamma_r\text{ and } y\in \mathcal{K}(\H)\otimes
SA\rtm\Gamma_r\}$$ and
$$\{(x,\,P_{\Gamma,max}\cdot \pi_{\Gamma,red,max}(x) \cdot
P_{\Gamma,max}+y);\, x\in A\rtr\Gamma_r\text{ and } y\in \mathcal{K}(\H)\otimes
SA\rtm\Gamma_r\}.$$
Moreover, the extension of $C^*$-algebras
\begin{equation*}0\longrightarrow \mathcal{K}(\H)\otimes
SA\rtr\Gamma\longrightarrow E^{(\pi,T)}_{red} \longrightarrow
A\rtr\Gamma\longrightarrow 0,\end{equation*}
\begin{equation*}\label{equ-ext-red-max}0\longrightarrow \mathcal{K}(\H)\otimes
SA\rtm\Gamma\longrightarrow E^{(\pi,T)}_{max} \longrightarrow
A\rtm\Gamma\longrightarrow 0\end{equation*} and
\begin{equation*}0\longrightarrow \mathcal{K}(\H)\otimes
SA\rtm\Gamma\longrightarrow E^{(\pi,T)}_{red,max} \longrightarrow
A\rtr\Gamma\longrightarrow 0\end{equation*} provided by the projection on the first factor are respectively
semi-split by the filtered cross-sections
$$s_{red}:A\rtr\Gamma\to E^{(\pi,T)}_{red};\, x \mapsto (x,P_{\Gamma,red} \cdot\pi_{\Gamma,red}(x)
\cdot P_{\Gamma,red}),$$
$$s_{max}:A\rtm\Gamma\to E^{(\pi,T)}_{max};\, x \mapsto (x,P_{\Gamma,max} \cdot\pi_{\Gamma,max}(x)
\cdot P_{\Gamma,max})$$ and
$$s_{red,max}:A\rtr\Gamma\to E^{(\pi,T)}_{max};\, x \mapsto (x,P_{\Gamma,max} \cdot\pi_{\Gamma,red,max}(x)
\cdot P_{\Gamma,max}).$$ Let us set $$f_1:E^{(\pi,T)}_{max}\to E^{(\pi,T)}_{red,max}: (x,y)\mapsto (\lambda_{\Ga,A,*}(x),y)$$
 and
 $$f_2:E^{(\pi,T)}_{red,max}\to E^{(\pi,T)}_{red}:(x,y)\mapsto (x,y\ts_{A\rtm\Ga}Id_{A\rtr\Ga}).$$
The the three above extensions fit in a commutative diagram
$$
\begin{CD}
0 @>>>   \mathcal{K}(\H)\otimes
SA\rtm\Gamma@>>>E^{(\pi,T)}_{max}  @>>> A\rtm\Ga @>>>  0\\
 @.         @V=VV     @V f_1VV         @VV\lambda_{\Ga,A} V\\
 0 @>>>   \mathcal{K}(\H)\otimes
SA\rtm\Gamma@>>>E^{(\pi,T)}_{red,max}  @>>> A\rtr\Ga @>>>  0\\
 @.         @V\lambda_{\Ga,\mathcal{K}(\H)\otimes
SA}VV     @Vf_2VV         @VV= V\\
0 @>>>   \mathcal{K}(\H)\otimes
SA\rtr\Gamma@>>>E^{(\pi,T)}_{red}  @>>> A\rtr\Ga @>>>  0
\end{CD}
$$ which satisfy the conditions of remark \ref{rem-fonct} relatively to $s_{red}$, $s_{max}$ and $s_{red,max}$, and hence
we deduce
\begin{equation}\label{equ-red-max1}
\DD_{ \mathcal{K}(\H)\otimes SA\rtm\Gamma,E^{(\pi,T)}_{red,max}}\circ \lambda_{A,\Ga,*}=
\DD_{ \mathcal{K}(\H)\otimes SA\rtm\Gamma,E^{(\pi,T)}_{max}}
\end{equation} and
\begin{equation}\label{equ-red-max2}
\lambda_{ \mathcal{K}(\H)\otimes SA,\Ga,*}\circ \DD_{ \mathcal{K}(\H)\otimes SA\rtm\Gamma,E^{(\pi,T)}_{red,max}}
=\DD_{ \mathcal{K}(\H)\otimes SA\rtr\Gamma,E^{(\pi,T)}_{red}}
\end{equation}

Let us set then $$\DD'_A=\MM_{SA\rtm\Ga}^{-1}\circ \DD_{SA\rtm\Gamma,E^{(\pi,T)}_{red,max}}:\K_*(A\rtr\Ga)\to\K_*(SA\rtm\Ga).$$
Since we have by definition of the quantitative Kasparov transformation the equalities  $$\JJ_{\Ga,red}([\partial_A])=\MM_{SA\rtr\Ga}^{-1}\circ \DD_{SA\rtr\Gamma,E^{(\pi,T)}_{red}}$$ and
$$\JJ_{\Ga,max}([\partial_A])=\MM_{SA\rtm\Ga}^{-1}\circ \DD_{SA\rtm\Gamma,E^{(\pi,T)}_{max}},$$ we deduce  by 
 using equations (\ref{equ-red-max1}) and (\ref{equ-red-max2}),  theorems \ref{thm-kas}, \ref{thm-kasp-max},
\ref{thm-product} and \ref{thm-product-max} and naturality of Morita equivalence, that there exists a control pair $(\lambda,h)$ such that
$\JJ_{\Ga,max}([\partial_A]^{-1})\circ \DD'_A$ is a $(\alpha,h)$-inverse for $\lambda_{\Ga,A,*}$.
\end{proof}

\section{The quantitative Baum-Connes conjecture}
In this section, we formulate a quantitative version for the Baum-Connes
conjecture and we prove  it  for a large class of groups.
\subsection{The Rips complex}
Let $\Ga$ be a finitely generated group equipped with a lenght $\ell$ arising from a finite and symmetric generating set. Recall
that for any positive number $d$, then the $d$-Rips complex  $P_d(\Ga)$ is the set
of finitely supported probability measures  on $\Gamma$ with support of diameter
less than $d$ for the distance induced by $\ell$. We equip $P_d(\Ga)$
with the distance induced by the norm
$\|h\|=\sup\{\|h(\ga)\|;\,\ga\in\Gamma\}$ for $h\in C_0(\Gamma,\C)$. Since  $\ell$  is a proper function, i.e. $B(e,r)$ is finite for
every  positive number $r$, we see that  $P_d(\Ga)$ is a finite dimension and locally finite
simplicial complexe and the action of $\Ga$ by left translations is
simplicial, proper  and cocompact. Let us denote by
\begin{itemize}
\item  $V_d(\Gamma)$ the
closed subset of elements  of  $P_d(\Ga)$ with support in $B{(e,d)}$.
\item  $W_d(\Gamma)$ the
closed subset of elements  of  $P_d(\Ga)$ with support in $B{(e,2d)}$;
\end{itemize}
Then $V_d(\Gamma)$ is a  compact subset of $W_d(\Gamma)$ and contains
a fundamental domain for the action of
$\Gamma$ on $P_d(\Ga)$.
\begin{lemma}
The compact  $V_d(\Gamma)$ is contained  in the interior of
$W_d(\Gamma)$.
\end{lemma}
\begin{proof}
Let $h$ be an element in $V_d(\Gamma)$ and choose an element  $\gamma$ in $B{(e,d)}$ such that
$h(\gamma)>0$.
Then if $g$ is an element of $P_d(\Ga)$ such that $\|g-h\|<h(\gamma)$,
we get that $g(\gamma)\neq 0$ and thus every element $\gamma'$ of the
support of $g$ satisfies $\ell(\gamma^{-1}\gamma')<d$. Hence $g$
belongs to $W_d(\Gamma)$.
\end{proof}
\begin{lemma}\label{lem-cutoff}
There is a continuous function $\phi:P_d(\Ga)\to[0,1]$
compactly  supported   in  $W_d(\Gamma)$ such that $$\sum_{\gamma\in\Gamma}\gamma(\phi)=1.$$
\end{lemma}
\begin{proof}
Let $\psi:P_d(\Ga)\to[0,1]$ a continuous function compactly supported
in  the interior of
$W_d(\Gamma)$ and such that $\psi(x)=1$ if $x$ belongs to
$V_d(\Gamma)$.
Since $V_d(\Gamma)$ contains a fundamental domain for the action  of $\Ga$ on $P_d(\Ga)$, we get that
$\sum_{\gamma\in\Gamma}\psi(\gamma x)>0$ for all $x$ in $P_d(\Ga)$ (notice that the sum
 $\sum_{\gamma\in\Gamma}\psi(\gamma x)$
is locally finite).  We define then
$\phi(x)=\frac{\psi(x)}{\sum_{\gamma\in\Gamma}\psi(\gamma x)}$  for any $x$ in $P_d(\Ga)$.
\end{proof}
Let us define $s_{\Gamma,d}$ as the cardinality of the finite set
$$\{\gamma\in\Gamma\text{ such that }\gamma W_d(\Gamma)\cap
W_d(\Gamma)\neq\emptyset\}.$$
Then for   any function $\phi$ as in lemma \ref{lem-cutoff}, the
function
$$e_\phi:\Gamma\to C_0(P_d(\Ga));\, \gamma\mapsto\sum_{\gamma\in\Gamma}\phi^{1/2}\gamma(\phi^{1/2})$$
 is a projection of
 $C_0(P_d(\Ga))\rtr\Gamma$ with  propagation  less than $s_{\Gamma,d}$. Moreover,
 since the set of function satisfying the condition of lemma
 \ref{lem-cutoff} is an affine space, we get that for any positive
 number $\eps$ and $r$ with $\eps <1/4$ and $r\gq s_{\Gamma,d}$, the class
$$[e_\phi,0]_{\eps,r}\in K_0^{\eps,r}(C_0(P_d(\Ga))\rtr\Gamma)$$ does
 not depend on the chosen  function $\phi$. Let us set then $r_{\Gamma,d,\eps}=k_{\JJ,\eps/\alpha_J}s_{\Gamma,d}$.
Recall that $k_\JJ$ can be chosen non increasing and in this case,  $r_{\Gamma,d,\eps}$ is non decreasing in
 $d$ and non increasing in $\eps$.
\begin{definition}\label{def-quantitative-assembly-map}
For any $\Gamma$-$C^*$-algebra $A$ and any positive numbers $\eps$, $r$ and
$d$ with $\eps<1/4$ and
$r\gq r_{\Gamma,d,\eps}$, we define the quantitative assembly map
\begin{eqnarray*}
\mu_{\Gamma,A,*}^{\eps,r,d}: KK_*^\Gamma(C_0(P_d(\Ga)),A)&\to&
K_*^{\eps,r}(\AG)\\
z&\mapsto&\big(J_\Gamma^{red,\frac{\eps}{\alpha_J},\frac{r}{k_{J,{\eps}/{\alpha_J}}}}(z)\big)\left([e_\phi,0]_{\frac{\eps}{\alpha_J},\frac{r}{k_{J,{\eps}/{\alpha_J}}}}\right).
\end{eqnarray*}
\end{definition}
Then according to  theorem \ref{thm-kas}, the map  $\mu_{\Gamma,A}^{\eps,r,d}$
is a  homomorphism of   groups  (resp. semi-groups) in even (resp. odd) degree. For any positive numbers
 $d$ and $d'$ such that $d\lq d'$, we denote by
$q_{d,d'}:C_0(P_{d'}(\Ga))\to C_0(P_d(\Ga))$ the homomorphism induced by the restriction
  from $P_{d'}(\Ga)$ to $P_d(\Ga)$. It is straightforward to check that if
  $d$, $d'$ and  $r$  are positive numbers such that $d\lq d'$ and
  $r\gq r_{\Gamma,d',\eps}$, then
$\mu_{\Gamma,A}^{\eps,r,d}=\mu_{\Gamma,A}^{\eps,r,d'}\circ
q_{d,d',*}$. Moreover, for every positive numbers
$\eps,\,\eps',\,d,\,r$ and $r'$ such that $\eps\lq\eps'\lq
1/4$, $r_{\Ga,d,\eps}\lq r$, $r_{\Ga,d,\eps'}\lq r'$, and
$r<r'$, we get by definition of a controlled morphism that
\begin{equation}\label{equation-assembly}
\iota^{\eps,\eps',r,r'}_*\circ \mu_{\Ga,A,*}^{\eps,r,d}=\mu_{\Ga,A,*}^{\eps',r',d}.
\end{equation}
Furthermore, the quantitative assembly maps are natural in the $\Ga$-$C^*$-algebra, i.e.
if $A$ and $B$ are $\Ga$-$C^*$-algebras and if $\phi:A\to B$ is a
$\Ga$-equivariant homomorphism, then
$$\phi_{\Ga,red,*,\eps,r}\circ\mu_{\Gamma,A,*}^{\eps,r,d}=\mu_{\Gamma,B,*}^{\eps,r,d}\circ\phi_*$$
 for every positive numbers $r$ and $\eps$ with $r\gq
r_{\Gamma,d,\eps}$ and $\eps<1/4$.
These quantitative assembly maps are related to the usual assembly
maps in the following way: recall from \cite{BCH} that there is  a bunch of assembly maps  with
coefficients in a $\Ga$-$C^*$-algebra $A$   defined by
\begin{eqnarray*}
\mu_{\Gamma,A,*}^{d}: KK_*^\Gamma(C_0(P_d(\Ga)),A)&\to&
K_*(\AG)\\
z&\mapsto&[e_\phi]\otimes_{C_0(P_d(\Ga))\rtimes\Gamma}J_\Gamma(z).
\end{eqnarray*}
For every positive numbers $r$ and $\eps$ with $r\gq
r_{\Gamma,d,\eps}$ and $\eps<1/4$, we have
\begin{equation}\label{equ-compatibity-assembly-maps} \iota^{\eps,r}_*\circ
\mu_{\Gamma,A,*}^{\eps,r,d}=\mu_{\Gamma,A,*}^{d}.\end{equation}

 Recall that since
$\mu_{\Gamma,A,*}^{d'}\circ q_{d,d',*} =\mu_{\Gamma,A,*}^{d}$ for
all positive numbers $d$ and $d'$ with $d\lq d'$, the family of assembly maps
$( \mu_{\Gamma,A}^{d})_{d>0}$ gives rise to a homomorphism
$$\mu_{\Ga,A,*}:\lim_{d>0}KK_*^\Gamma(C_0(P_d(\Ga)),A)\longrightarrow
K_*(\AG)$$ called the Baum-Connes assembly map.

%%%%%%%%%%%%%%%%%%%%%%%%%%%%%%%%%

\subsection{Quantitative statements}
Let us consider for a  $\Ga$-$C^*$-algebra $A$ and positive numbers
$d,d',r,r',\eps$ and $\eps'$ with $d\lq d'$, $\eps'\lq\eps< 1/4$,  $r_{\Gamma,d,\eps}\lq
r$ and  $ r'\lq
r$ the following statements:
\begin{description}
\item[$QI_{\Ga,A,*}(d,d',r,\eps)$]  for any element $x$ in
  $KK_*^\Gamma(C_0(P_d(\Ga)),A)$, then
  $\mu_{\Gamma,A,*}^{\eps,r,d}(x)=0$
  in $K_*^{\eps,r}(\AG)$ implies that
  $q_{d,d'}^*(x)=0$ in $KK_*^\Gamma(C_0(P_{d'}(\Ga)),A)$.
\item[$QS_{\Ga,A,*}(d,r,r',\eps,\eps')$] for every $y$
  in  $K_*^{\eps',r'}(\AG)$, there exists an element $x$ in
  $KK_*^\Gamma(C_0(P_d(\Ga)),A)$ such that
$$\mu_{\Gamma,A,*}^{\eps,r,d}(x)=\iota_*^{\eps',\eps,r',r}(y).$$
\end{description}
Using equation (\ref{equ-compatibity-assembly-maps}) and remark \ref{rem-hom} we get
\begin{proposition}
Assume that for all  positive number $d$ there exists  a positive number $\eps$ with $\eps<1/4$
for which the following holds:

\medskip

for any positive number  $r$ with $r\gq r_{\Ga,d,\eps}$, there
exists a positive number $d'$ with $d'\gq d$ such that
$QI_{\Ga,A}(d,d',r,\eps)$ is satisfied.

\medskip

Then $\mu_{\Ga,A,*}$ is one-to-one.
\end{proposition}
We can also easily prove the following:
\begin{proposition}
Assume that there exists  a positive number $\eps'$ with $\eps'<1/4$
such that the following holds:

\medskip

for any  positive number $r'$ , there exist
positive numbers $\eps,d$ and  $r$ with $\eps'\lq\eps<1/4$, $
r_{\Ga,d,\eps}\lq r$ and $r'\lq  r$ such that $QS_{\Ga,A}(d,r,r',\eps,\eps')$ is
true.

\medskip

Then  $\mu_{\Ga,A,*}$ is onto.
\end{proposition}

The following results provide  numerous
examples of finitely generated groups that satisfy the quantitative statements.

\begin{theorem}\label{thm-quantBC-injectivity} Let $A$ be a $\Ga$-$C^*$-algebra. Then the following assertions are equivalent:
\begin{enumerate}
\item $\mu_{\Ga,\ell^\infty(\N,\K(\H)\otimes A),*}$ is one-to-one,
\item  For any positive numbers $d$,
  $\eps$ and $r\gq r_{\Ga,d,\eps}$ with $\eps<1/4$ and  $r\gq r_{\Ga,d}$, there
exists   a positive number $d'$ with $d'\gq d$ for which
$QI_{\Ga,A}(d,d',r,\eps)$ is satisfied.
\end{enumerate}
\end{theorem}
\begin{proof}
 Assume that condition (ii) holds.

Let $x$ be an element in some $KK^\Ga_*(C_0(P_d(\Ga)),\ell^\infty(\N,\K(\H)\otimes A))$ such
that
$$\mu_{\Ga,\ell^\infty(\N,\K(\H)\otimes A),*}^{d}(x)=0.$$ Using  equation (\ref{equ-compatibity-assembly-maps}), we get that  $\iota^{\eps',r'}_*(\mu_{\Gamma,A,*}^{\eps',r',d}(x))=0$ for any $\eps'$ in $(0,1/4)$ and $r'\gq r_{\Ga,d,\eps'}$ and hence, by remark
\ref{rem-hom}, we can find $\eps$ and
 $r>r_{\Ga,d,\eps}$ such that $\mu_{\Ga,\ell^\infty(\N,\K(\H)\otimes A),*}^{\eps,r,d}(x)=0$. Recall from \cite[Proposition 3.4]{oy}
that we have an isomorphism
\begin{equation}\label{equ-morita-equiv-khom} KK^\Ga_0(C_0(P_{d}(\Ga)),\ell^\infty(\N,\K(\H)\otimes
A))\stackrel{\cong}{\longrightarrow}KK^\Ga_0(C_0(P_d(\Ga)),A)^\N\end{equation}
induced on the $j$ th factor  and up to the Morita equivalence
$$KK^\Ga_0(C_0(P_d(\Ga)),A)\cong KK^\Ga_0(C_0(P_d(\Ga)),\K(\H)\otimes
A)$$ by the $j$ th projection  $\ell^\infty(\N,
\K(\H)\otimes A)\to
 \K(\H)\otimes A$. Let  $(x_i)_{i\in\N}$  be the element of $KK^\Ga_0(C_0(P_d(\Ga)),A)^\N$
corresponding to $x$ under this identification and let  $d'\gq d$ be a number such that $QI_{\Ga,A}(d,d',r,\eps)$ holds.
Naturality of the quantitative assembly maps implies that $\mu_{\Ga,A,*}^{\eps,r,d}(x_i)=0$ and hence that $q_{d,d',*}(x_i)=0$
in $KK^\Ga_*(C_0(P_{d'}(\Ga)),A)$
for every integer $i$. Using once again the isomorphism of equation (\ref{equ-morita-equiv-khom}), we get that $q_{d,d',*}(x)=0$ in
$KK^\Ga_*(C_0(P_{d'}(\Ga)),\ell^\infty(\N,\K(\H)\otimes A)$ and hence $\mu_{\Ga,\ell^\infty(\N,\K(\H)\otimes A),*}$ is one-to-one.

\medskip

 Let us prove the converse in the even case, the odd case being similar.
Assume that there exists  positive numbers $d,\,\eps$ and $r$
  with
  $\eps<1/4$  and  $r\gq
  r_{\Ga,d,\eps}$ and such that for all $d'\gq d$, the condition
$QI_{\Ga,A}(d,d',r,\eps)$ does not hold. Let us prove that
$\mu_{\Ga,\ell^\infty(\N,\K(\H)\otimes A),*}$ is not one-to-one.
Let $(d_i)_{i\in\N}$ be an increasing and  unbounded sequence of
positive numbers such that $d_i\gq d$ for all integer $i$.  For all
integer $i$, let $x_i$ be an element in $KK^\Ga_0(C_0(P_d(\Ga)),A)$ such
that $\mu_{\Ga,A,*}^{\eps,r,d}(x_i)=0$ in $K_0(\AG)$  and
$q_{d,d_i,*}(x_i)\neq 0$ in
$KK^\Ga_0(C_0(P_{d_i}(\Ga)),A)$.  Let $x$ be the element of  $KK^\Ga_0(C_0(P_{d}(\Ga)),\ell^\infty(\N,
\K(\H)\otimes A))$ corresponding to $(x_i)_{i\in\N}$ under the identification of equation (\ref{equ-morita-equiv-khom}).
Let $(p_i)_{i\in\N}$ be a family of $\eps$-$r$-projections, with $p_i$ in some
$M_{l_i}(\widetilde{\AG})$
and $n$ an integer
  such that $$\mu_{\Ga,\ell^\infty(\N,\K(\H)\otimes
    A),*}^{\eps,r,d}(x)=[(p_i)_{i\in\N},n]_{\eps,r}$$ in
 $K_0^{\eps,r}(\ell^\infty(\N,\K(\H)\otimes
    A)\rtr\Ga)$. By
 naturality of $\mu_{\Ga,\bullet,*}^{\eps,r,d}$, we get that
 $[p_i,n]_{\eps,r}=0$ in $K_0^{\eps,r}(\AG)$ for all integer $i$. We
 see by using
  proposition \ref{prop-bounded-hom} that then
$\iota_*^{\eps,r}([(p_i)_{i\in\N},n])=0$ in $K_0(\ell^\infty(\N,\K(\H)\otimes
    A)\rtr\Ga)$.
We eventually obtain  that
$\mu^{d}_{\Ga,A}(x)=\iota_*^{\eps,r}\circ\mu^{\eps,r,d}_{\Ga,A}(x)=0$.
Since $q_{d,d_i,*}(x)\neq 0$ for every integer $i$, we get that
$\mu_{\Ga,\ell^\infty(\N,\K(\H)\otimes A),*}$ is not one-to-one.
\end{proof}
\begin{theorem}\label{thm-quantBC-surjectivity} There exists $\lambda>1$ such that for any  $\Ga$-$C^*$-algebra, the following
 assertions are equivalent:
\begin{enumerate}
\item $\mu_{\Ga,\ell^\infty(\N,\K(\H)\otimes A),*}$ is onto;
\item For any  positive numbers $\eps$ and $r'$  with $\eps<\frac{1}{4\lambda}$, there exist
positive numbers $d$ and  $r$ with  $
r_{\Ga,d,\eps}\lq r$ and $r'\lq r$ for which  $QS_{\Ga,A}(d,r,r',\lambda\eps,\eps)$
is satisfied.
\end{enumerate}
\end{theorem}

\begin{proof} Choose $\lambda$ as in remark \ref{rem-hom}.
 Assume that condition (ii) holds. Let $z$ be an element in $K_*(\ell^\infty(\N,\K(\H)\otimes A)\rtr\Ga)$
 and let $y$ be an element in
 $K^{\eps,r'}_*(\ell^\infty(\N,\K(\H)\otimes A)\rtr\Ga)$ such that $\iota^{\eps,r'}_*(y)=z$, with $0<\eps<\frac{1}{4\lambda}$ and $r'>0$.
Let $y_i$ be the image of $y$  under the composition
\begin{equation}\label{equ-composition-surj}
 K^{\eps,r'}_*(\ell^\infty(\N,\K(\H)\otimes A)\rtr\Ga)\to
K^{\eps,r'}_*(\K(\H)\otimes\AG)\stackrel{\cong}{\to}
K^{\eps,r'}_*(\AG), \end{equation} where the first map is induced by the
evaluation $\ell^\infty(\N,\K(\H)\otimes A)\longrightarrow
\K(\H)\otimes A$ at $i$ and the second map is the Morita equivalence
of proposition \ref{prop-morita}. Let $d$ and $r$ be numbers with
$r\gq r'$ and  $r\gq r_{\Ga,d,\eps}$ and such that $QS_{\Ga,A}(d,r,r',\lambda\eps,\eps)$ holds. Then  for any integer $i$, 
there exists
a $x_i$ in
$KK^\Ga_*(C_0(P_{d}(\Ga)),A)$ such that $\mu_{\Ga,A,*}^{\lambda\eps,r,d}(x_i)=\iota_*^{\eps,\lambda\eps,r',r}(y_i)$ 
in $K_*^{\eps,r}(A\rtr\Ga)$.
Let $$x\in KK^\Ga_*(C_0(P_{d}(\Ga)),\ell^\infty(\N,\K(\H)\otimes
A))$$ be  the element corresponding to $(x_i)_{i\in\N}$ under the identification of equation (\ref{equ-morita-equiv-khom}). 
By naturality of the
quantitative assembly maps, we get according to
proposition \ref{prop-bounded-hom} and up to replace $\lambda$ by $3\lambda$ (for the odd case)  that $$\mu_{\Ga,\ell^\infty(\N,\K(\H)\otimes
A)),*}^{\lambda\eps,r,d}(x)=\iota_*^{\eps,\lambda\eps,r',r}(y)$$ in $K_*^{\eps,r}(\ell^\infty(\N,\K(\H)\otimes A)\rtr\Ga)$. We have hence
$$\mu_{\Ga,\ell^\infty(\N,\K(\H)\otimes
A)),*}^{d}(x)=\iota_*^{\eps,r'}(y)=z,$$ and therefore $\mu_{\Ga,\ell^\infty(\N,\K(\H)\otimes A),*}$ is onto.

\medskip

Let us prove the converse in the even case, the odd case being similar.
 Assume that there exist positive numbers  $\eps$ and $r'$ with  $\eps<\frac{1}{4\lambda}$ such that for 
all 
  positive numbers $r$ and $d$ with $r\gq r'$ and  $r\gq r_{\Ga,d,\eps}$,
  then $QS_{\Ga,A}(d,r,r',\lambda\eps,\eps)$ does not hold. Let us prove
  then that $\mu_{\Ga,\ell^\infty(\N,\K(\H)\otimes A),*}$ is not onto. Let $(d_i)_{i\in\N}$ and $(r_i)_{i\in\N}$
  be increasing and unbounded sequences of positive numbers such that
  $r_i\gq r_{\Ga,d_i,\lambda\eps}$ and $r_i\gq r'$. Let  $y_i$ be an element in
  $K^{\eps,r'}_0(\AG)$ such that $\iota_*^{\eps,\lambda\eps,r',r_i}(y_i)$
  is not in the range of $\mu_{\Ga,A,*}^{\lambda\eps,r_i,d_i}$. There
  exists an element $y$ in
  $K^{\eps,r'}_0(\ell^\infty(\N,\K(\H)\otimes A)\rtr\Ga)$ such
  that for every integer $i$, the image of $y$ under the composition of equation (\ref{equ-composition-surj})
is $y_i$.
Assume that for some $d'$, there is an $x$ in  $KK^\Ga_0(C_0(P_{d'}(\Ga)),\ell^\infty(\N,\K(\H)\otimes
A))$ such that $\iota_*^{\eps,r'}(y)=\mu_{\Ga,\ell^\infty(\N,\K(\H)\otimes
    A),*}^{d'}(x)$. Using remark  \ref{rem-hom},
  we see that  there exists a positive number $r$ with  $r'\lq r$ and
 $r_{\Ga,d',\lambda\eps}\lq r$ and
  such that $$\iota_*^{\eps,\lambda\eps,r',r}\circ\mu_{\Ga,\ell^\infty(\N,\K(\H)\otimes
    A),*}^{\eps,r',d'}(x)=\iota_*^{\eps,\lambda\eps,r',r}(y).$$ But then, if
  we choose $i$ such that $r_i\gq r$ and $d_i\gq d'$ we get by using
  naturality  of the assembly map and equation (\ref{equation-assembly}) that $\iota_*^{\eps,\lambda\eps,r',r_i}(y_i)$ belongs to the
  image of $\mu_{\Ga,A,*}^{\lambda\eps,r_i,d_i}$, which contradicts our assumption.

\end{proof}

Replacing in the proof of (ii) implies (i)  of theorems  \ref{thm-quantBC-injectivity}  and   \ref{thm-quantBC-surjectivity} the algebra
$\ell^\infty(\N,\K(\H)\otimes A)$ by $\prod_{i\in\N}(\K(\H)\otimes
A_i)$ for a family $(A_i)_{i\in\N}$ of $\Ga$-$C^*$-algebras, we can prove  the
following result.
\begin{theorem}\label{thm-quant-surj}Let $\Ga$ be a discrete group.
\begin{enumerate}
\item Assume that   for any  $\Ga$-$C^*$-algebra
  $A$, the assembly map $\mu_{\Ga,A,*}$ is one-to-one.
    Then for any positive numbers $d$,
  $\eps$ and $r\gq r_{\Ga,d,\eps}$ with $\eps<1/4$ and  $r\gq r_{\Ga,d}$, there
exists   a positive number $d'$ with $d'\gq d$ such that
$QI_{\Ga,A}(d,d',r,\eps)$ is satisfied  for every $\Ga$-$C^*$-algebra
  $A$;
\item Assume that   for any  $\Ga$-$C^*$-algebra
  $A$, the assembly map $\mu_{\Ga,A,*}$ is onto.  Then for some  $\lambda>1$ and for any  positive numbers $\eps$
and $r'$ with  $\eps<\frac{1}{4\lambda}$, there exist
positive numbers $d$ and  $r$ with $
r_{\Ga,d,\eps}\lq r$ and $r'\lq r$ such that  $QS_{\Ga,A}(d,r,r',\lambda\eps,\eps)$
is satisfied for every $\Ga$-$C^*$-algebra
  $A$.
\end{enumerate}
In particular, if  $\Ga$ satisfies the Baum-Connes conjecture with
coefficients, then $\Ga$ satisfies points  {\rm (i)} and {\rm
  (ii)} above.
\end{theorem}
Recall from \cite{sty,y1} that if $\Ga$
 coarsely embeds
 in a Hilbert space, then  $\mu_{\Ga,A,*}$ is one-to-one for every
 $\Ga$-$C^*$-algebra $A$. Hence we get:
\begin{corollary}
If $\Ga$
 coarsely embeds
 in a Hilbert space,  then for any positive numbers $d$,
  $\eps$ and $r\gq r_{\Ga,d,\eps}$ with $\eps<1/4$ and  $r\gq r_{\Ga,d}$, there
exists   a positive number $d'$ with $d'\gq d$ such that
$QI_{\Ga,A}(d,d',r,\eps)$ is satisfied  for every $\Ga$-$C^*$-algebra
  $A$;
\end{corollary}

The quantitative assembly maps admit  maximal versions
defined with notations of definition \ref{def-quantitative-assembly-map}
for any $\Gamma$-$C^*$-algebra $A$ and any positive number $\eps$, $r$ and
$d$ with $\eps<1/4$ and
$r\gq r_{\Gamma,d,\eps}$, as
\begin{eqnarray*}
\mu_{\Gamma,A,max,*}^{\eps,r,d}: KK_*^\Gamma(C_0(P_d(\Ga)),A)&\to&
K_*^{\eps,r}(A\rtm\Ga)\\
z&\mapsto&
\big(J_\Gamma^{max,\frac{\eps}{\alpha_J},\frac{r}{k_{J,{\eps}/{\alpha_J}}}}(z)\big)
\left([e_\phi,0]_{\frac{\eps}{\alpha_J},\frac{r}{k_{J,{\eps}/{\alpha_J}}}}\right).
\end{eqnarray*}
 As  in the reduced case, we have using
the same notations
\begin{itemize}
\item for any positive number $d$ and $d'$ such that $d\lq d'$, then
$$\mu_{\Gamma,A,{max},*}^{\eps,r,d}=\mu_{\Gamma,A,{max},*}^{\eps,r,d'}\circ
q_{d,d',*}.$$
\item for every positive numbers
$\eps,\,\eps',\,d,\,r$ and $r'$ such that $\eps\lq\eps'\lq
1/4$, $r_{\Ga,d,\eps}\lq r$, $r_{\Ga,d,\eps'}\lq r'$, and
$r<r'$, then
\begin{equation*}
\iota^{\eps,\eps',r,r'}_*\circ \mu_{\Ga,A,{max},*}^{\eps,r,d}=\mu_{\Ga,A,{max},*}^{\eps',r',d}.
\end{equation*}
\item the maximal quantitative assembly maps are natural in the
  $\Ga$-$C^*$-algebras.
\end{itemize}
Moreover,   by theorem \ref{thm-kasp-max}(i),    the maximal quantitative assembly maps are compatible with the
reduced  ones, i.e
$\mu_{\Gamma,A,*}^{\eps,r,d}=\lambda_{\Ga,A,*}^{\eps,r}\circ\mu_{\Gamma,A,{max},*}^{\eps,r,d}$.
The surjectivity of the Baum-Connes assembly map $\mu_{\Ga,A,*}$  implies that 
the map $$\lambda_{\Ga,A,*}:K_*(A\rtm\Ga)\to K_*(A\rtr\Ga)$$ is onto. We have a similar statement in the setting of quantitative $K$-theory.
%As a consequence, noticing that in the proof of $\text{(i)}\Rightarrow\text{(ii)}$ in theorem \ref{thm-quantBC-surjectivity}, we can choose $\eps=\lambda\eps'$,we obtain

\begin{theorem} There exists $\lambda>1$ such the following holds : let $\Ga$ be a discrete group and
   assume  that   for any  $\Ga$-$C^*$-algebra
  $A$, the assembly map $\mu_{\Ga,A,*}$ is onto. Then for any positive
  numbers $\eps$ and $r$, with $\eps<\frac{1}{4\lambda}$, there exists a positive
  number  $r'$ with $r'\gq r$ such that
\begin{itemize}
\item for any $\Ga$-$C^*$-algebra $A$;
\item for any $x$ in  $K_*^{\eps,r}(A\rtr\Ga)$,
\end{itemize}
there exists $y$ in
$K_*^{\lambda\eps,r'}(A\rtm\Ga)$ such that
$\lambda_{\Ga,A,*}^{\lambda\eps,r'}(y)=\iota_*^{\eps,\lambda\eps,r,r'}(x)$.
\end{theorem}

\section{Further comments}
The definition of quantitative $K$-theory can be extended to the framework of 
filtered Banach algebras, i.e. Banach algebra $A$
equipped with a family
$(A_r)_{r>0}$ of  linear subspaces   indexed by positive numbers such that:
\begin{itemize}
\item $A_r\subset A_{r'}$ if $r\lq r'$;
\item $A_r\cdot A_{r'}\subset A_{r+r'}$;
\item the subalgebra $\ds\bigcup_{r>0}A_r$ is dense in $A$.
\end{itemize}
Since we no more have an  involution, we need  to introduce instead a norm control for almost idempotents. Let $\eps$ be in $(0,1/4)$ 
and let $r$ and $N$ be positive numbers. An element $e$ of $A$ is an $\eps$-$r$-$N$-idempotent if
\begin{itemize}
 \item $e$ is in $A_r$;
\item  $\|e^2-e\|<\eps$;
\item  $\|e\|<N$;
\end{itemize}
Similarly, if $A$ is a unital, an element $x$ in $A$ is called  $\eps$-$r$-$N$-invertible
if 
\begin{itemize}
 \item $x$ is in $A_r$;
\item  $\|x\|<N$;
\item  there exists an element $y$ in $A_r$ such that $\|y\|<N$, $\|xy-1\|<\eps$ and $\|yx-1\|<\eps$.
\end{itemize}

Quantitative $K$-theory can then be defined in the setting of $\eps$-$r$-$N$-idempotents and of $\eps$-$r$-$N$-invertibles.
We obtain in this way a bunch of
abelian semi-groups $(K_*^{\eps,r,N}(A))_{\eps\in(0,1/4),r>,N>1}$. Let us set for  a
 fixed  $N>1$ $$\K_*^N(A)=(K_*^{\eps,r,N}(A))_{\eps\in(0,1/4),r>0}.$$ If $A$ is   a filtered $C^*$-algebra and  $e$   an 
 $\eps$-$r$-$N$-idempotent in $A$, then 
there is an obvious $(1,1)$-controlled morphism $\K_0(A)\to \K_0^N(A)$. 
Approximating $((2e^*-1)(2e-1)+1)^{1/2}e((2e^*-1)(e-1)+1)^{-1/2}$ by using  a
power  serie (compare with the proof of lemma \ref{lem-conjugate}), we get  that for every $N>1$, there exists a control pair
 $(\lambda_N,h_N)$
 such that    $\K_0(A)\to \K_0^N(A)$  is a $(\lambda_N,h_N)$-controlled isomorphism. Using the polar decomposition, we have 
a similar statement in the odd case.

\bibliographystyle{plain}

\begin{thebibliography}{}

\end{thebibliography}


\begin{thebibliography}{10}

\bibitem{A}
Michael   Atiyah.
\newblock Elliptic operators, discrete groups and von {N}eumann algebras.
\newblock In {\em Colloque ``{A}nalyse et {T}opologie'' en l'{H}onneur de
  {H}enri {C}artan ({O}rsay, 1974)}, pages 43--72. Ast\'erisque, No. 32--33.
  Soc. Math. France, Paris, 1976.

\bibitem{BCH}
Paul   Baum, Alain Connes, and Nigel Higson.
\newblock Classifying space for proper actions and {$K$}-theory of group
  {$C^\ast$}-algebras.
\newblock In {\em {$C^\ast$}-algebras: 1943--1993 ({S}an {A}ntonio, {TX},
  1993)}, volume 167 of {\em Contemp. Math.}, pages 240--291. Amer. Math. Soc.,
  Providence, RI, 1994.

\bibitem{C}
Alain Connes.
\newblock A survey of foliations and operator algebras.
\newblock In {\em Operator algebras and applications, {P}art {I} ({K}ingston,
  {O}nt., 1980)}, volume~38 of {\em Proc. Sympos. Pure Math.}, pages 521--628.
  Amer. Math. Soc., Providence, R.I., 1982.

\bibitem{CS}
Alain Connes and Georges Skandalis.
\newblock The longitudinal index theorem for foliations.
\newblock {\em Publ. Res. Inst. Math. Sci.}, 20(6):1139--1183, 1984.

\bibitem{CM}
Alain Connes and Henri Moscovici.
\newblock Cyclic cohomology, the {N}ovikov conjecture and hyperbolic groups.
\newblock {\em Topology}, 29(3):345--388, 1990.

\bibitem{cuntz}
Joachim Cuntz.
\newblock {$K$}-theoretic amenability for discrete groups.
\newblock {\em J. Reine Angew. Math.}, 344:180--195, 1983.

\bibitem{gwy}
Guihua Gong, Qin Wang, and Guoliang Yu.
\newblock Geometrization of the strong {N}ovikov conjecture for residually
  finite groups.
\newblock {\em J. Reine Angew. Math.}, 621:159--189, 2008.

\bibitem{hk}
Nigel Higson and Gennadi  Kasparov.
\newblock {$E$}-theory and {$KK$}-theory for groups which act properly and
  isometrically on {H}ilbert space.
\newblock {\em Invent. Math.}, 144(1):23--74, 2001.

\bibitem{hry}
Nigel Higson, John Roe, and Guoliang Yu.
\newblock A coarse {M}ayer-{V}ietoris principle.
\newblock {\em Math. Proc. Cambridge Philos. Soc.}, 114(1):85--97, 1993.

\bibitem{julgvalette}
Pierre Julg and Alain Valette.
\newblock {$K$}-theoretic amenability for {${\rm SL}_{2}({\bf Q}_{p})$}, and
  the action on the associated tree.
\newblock {\em J. Funct. Anal.}, 58(2):194--215, 1984.

\bibitem{kas}
Gennadi  Kasparov.
\newblock Equivariant {$KK$}-theory and the {N}ovikov conjecture.
\newblock {\em Invent. Math.}, 91(1):147--201, 1988.

\bibitem{laff-inv}
Vincent Lafforgue.
\newblock {$K$}-th\'eorie bivariante pour les alg\`ebres de {B}anach,
  groupo\"\i des et conjecture de {B}aum-{C}onnes. {A}vec un appendice
  d'{H}erv\'e {O}yono-{O}yono.
\newblock {\em J. Inst. Math. Jussieu}, 6(3):415--451, 2007.

\bibitem{mop}
Michel Matthey, Herv{\'e} Oyono-Oyono, and Wolfgang Pitsch.
\newblock Homotopy invariance of higher signatures and 3-manifold groups.
\newblock {\em Bull. Soc. Math. France}, 136(1):1--25, 2008.

\bibitem{oy}
Herv{\'e} Oyono-Oyono and Guoliang Yu.
\newblock {$K$}-theory for the maximal {R}oe algebra of certain expanders.
\newblock {\em J. Funct. Anal.}, 257(10):3239--3292, 2009.

\bibitem{r}
John Roe.
\newblock Coarse cohomology and index theory on complete Riemannian manifolds.
\newblock {\em Mem. Amer. Math. Soc.}, 104(497), 1993.

\bibitem{sty}
Georges Skandalis, Jean-Louis. Tu, and Guoliang Yu.
\newblock The coarse {B}aum-{C}onnes conjecture and groupoids.
\newblock {\em Topology}, 41(4):807--834, 2002.

\bibitem{tumoy}
Jean-Louis Tu.
\newblock La conjecture de {B}aum-{C}onnes pour les feuilletages moyennables.
\newblock {\em $K$-Theory}, 17(3):215--264, 1999.

\bibitem{we}
N. E.   Wegge-Olsen.
\newblock {\em {$K$}-theory and {$C\sp *$}-algebras}.
\newblock Oxford Science Publications. The Clarendon Press Oxford University
  Press, New York, 1993.
\newblock A friendly approach.

\bibitem{y2}
Guoliang Yu.
\newblock The {N}ovikov conjecture for groups with finite asymptotic dimension.
\newblock {\em Ann. of Math. (2)}, 147(2):325--355, 1998.

\bibitem{y1}
Guoliang Yu.
\newblock The coarse {B}aum-{C}onnes conjecture for spaces which admit a
  uniform embedding into Hilbert space.
\newblock {\em Invent. Math.}, 139:201--240, 2000.

\end{thebibliography}

\end{document}